\documentclass{amsart}

\usepackage[utf8]{inputenc}
\usepackage[T1]{fontenc}
\usepackage[pdfencoding=auto, psdextra]{hyperref}
\usepackage{amsmath}
\usepackage{amssymb}
\usepackage{amsthm}
\usepackage{mathtools}  
\usepackage{microtype}
\usepackage{graphicx}
\usepackage{subcaption}
\usepackage{clrscode3e} 
\usepackage{algorithm} 
\usepackage[noend]{algorithmic}
\usepackage{etoolbox}
\usepackage{marginnote}
\usepackage{bbm}
\usepackage{siunitx}
\usepackage{enumitem}
\usepackage[font=small,labelfont=bf]{caption} 
\usepackage[section]{placeins} 
\usepackage[textsize=tiny]{todonotes}
\usepackage{nicefrac}  
\usepackage{scalerel}[2016/12/29]  
\usepackage{fixmath}

\usepackage{marginnote}

\newcommand*{\mailto}[1]{\href{mailto:#1}{\nolinkurl{#1}}}
\newcommand{\Dx}{{\Delta x}}

\newcommand{\R}{\mathbb{R}}
\renewcommand{\id}{\text{\normalfont{id}}}

\newcommand{\arxiv}[1]{\href{http://arxiv.org/pdf/#1}{arXiv:#1}} 

\newcommand{\D}{\mathcal{D}}
\newcommand{\F}{\mathcal{F}}
\newcommand{\M}{\mathcal{M}}

\newcommand{\T}{\boldsymbol{\scaleobj{0.75}{\mathcal{T}}}}

\DeclareMathOperator*{\argmin}{arg\,min}

\apptocmd{\lim}{\limits}{}{}

\numberwithin{equation}{section}
{
    \theoremstyle{plain}
    \newtheorem{definition}{Definition}[section]
    \newtheorem{remark}[definition]{Remark}
    \newtheorem{prop}[definition]{Proposition}
    \newtheorem{example}[definition]{Example}

    \newtheorem{theorem}[definition]{Theorem}
    \newtheorem{corollary}[definition]{Corollary}
    \newtheorem{lemma}[definition]{Lemma}   
}

\newlist{thmlist}{enumerate}{1}
\setlist[thmlist]{label=(\roman{thmlisti}), 
, noitemsep}

\title[Convergence rate for numerical $\alpha$-dissipative solutions]{Rate of convergence for numerical  \fontsize{15pt}{17pt}{$\alpha$}-dissipative solutions of the Hunter--Saxton equation}

\author[T. Christiansen]{Thomas Christiansen}
\address{Department of Mathematical Sciences\\ NTNU Norwegian University of Science and Technology\\ NO-7491 Trondheim\\ Norway}
\email{\mailto{thomas.christiansen@ntnu.no}}
\urladdr{\url{https://www.ntnu.edu/employees/thomachr}}

\author[K. Grunert]{Katrin Grunert}
\address{Department of Mathematical Sciences\\ NTNU Norwegian University of Science and Technology\\ NO-7491 Trondheim\\ Norway}
\email{\mailto{katrin.grunert@ntnu.no}}
\urladdr{\url{https://www.ntnu.edu/employees/katrin.grunert}}

\thanks{Research supported by the grant {\it Wave Phenomena and Stability --- a Shocking Combination (WaPheS)} from the Research Council of Norway.}  

\begin{document}

\counterwithin{equation}{section}
\raggedbottom
 \allowdisplaybreaks
 
 \begin{abstract}
 We prove that $\alpha$-dissipative solutions to the Cauchy problem of the Hunter--Saxton equation, 
 where $\alpha \in W^{1, \infty}(\R, [0, 1))$, can be computed numerically with order $\mathcal{O}(\Dx^{\nicefrac{1}{8}}+\Dx^{\nicefrac{\beta}{4}})$ in $L^{\infty}(\R)$, provided there exist constants $C > 0$ and $\beta \in (0, 1]$ such that the initial spatial derivative $\bar{u}_{x}$ satisfies $\|\bar{u}_x(\cdot + h) - \bar{u}_x(\cdot)\|_2 \leq Ch^{\beta}$ for all $h \in (0, 2]$.  The derived convergence rate is exemplified by a number of numerical experiments. 
\end{abstract}

\maketitle

\vspace{-0.8cm}
 \section{Introduction}
 
 This paper is concerned with the Cauchy problem for the Hunter--Saxton (HS) equation, which takes the form  
   \begin{equation}\label{eq:HS}
 	u_t(t, x) +uu_x(t, x) = \frac{1}{4}\int_{-\infty}^xu_x^2(t, z)dz - \frac{1}{4}\int_x^{\infty}u_x^2(t, z)dz, \hspace{0.4cm} u\vert_{t=0} = \bar{u}. 
\end{equation}
The above equation was derived in \cite{DynamicsDirector} via a first-order asymptotic expansion around constant equilibrium states of the nonlinear variational wave equation, $\psi_{tt} + c(\psi)(c(\psi)\psi_x)_x = 0$. One can therefore view solutions of \eqref{eq:HS} as describing the long-time behavior of these perturbed equilibrium states. 

The HS equation has many intriguing properties, but of particular importance and of main concern from a numerical perspective is the fact that solutions experience {\em wave breaking} -- a phenomenon characterized by pointwise blow-ups of the spatial derivative $u_x$ within finite time. Classical solutions therefore cease to exist, and one has to study weak solutions. Despite $u_x$ developing singularities at certain points in space-time, $u(t, \cdot)$ remains continuous for all $t\geq 0$, in fact H{\"o}lder continuous, while $u_x(t, \cdot) \in L^2(\R)$. In addition, wave breaking also gives rise to energy concentrations on sets of measure zero, and in order to extend weak solutions beyond wave breaking, one has to ambiguously choose how to manipulate this concentrated energy.  

Two natural choices immediately come to mind: \emph{i)} one can choose to reinsert all the concentrated energy at every wave breaking occurrence, leading to {\em conservative solutions}, or \emph{ii)} all the concentrated energy can be removed from the system, yielding {\em dissipative solutions}. Well-posedness has been established for both kinds of solutions, see \cite{BressanDissipative, LipschitzConservative, DafermosCharacteristics, UniquenessConservative}. Yet a more general and flexible concept is that of $\alpha$-dissipative solutions,  proposed in \cite{AlphaCH} for the related Camassa--Holm equation and in \cite{AlphaHS} for the HS equation. As the name suggests, the idea is to remove an $\alpha$-fraction of the concentrated energy, but in addition, one may allow $\alpha$ to belong to $W^{1, \infty}(\R, [0, 1)) \cup \{1\}$, such that the amount of energy removed can depend on the spatial location at which wave breaking takes place. This is the kind of solution we will focus on in this work. Let us emphasize that the existence of such solutions has been shown in \cite{AlphaHS}, but uniqueness is still an open question. 

To keep track of the corresponding energy, which distinguishes the various continuations, it is common to augment the wave profile $u$ by a nonnegative, finite Radon measure $\mu$ that represents the energy density, see e.g., \cite{LipschitzConservative, AlphaHS}. This measure encodes all the information about wave breaking, and its absolutely continuous part satisfies $d\mu_{\mathrm{ac}} = u_x^2dx$. Moreover, $\mu_{\mathrm{sing}}$, its singular part, tells us where energy has concentrated. However, let us also point out that wave breaking can occur in the form of infinitesimal energy concentrations, which are described by $\mu_{\mathrm{ac}}$, see for instance the cusped wave profile in Example~\ref{ex:plainCusp}, which is discussed in \cite[Ex. 5.2]{Alpha1} and \cite[Ex. 3]{NumericalConservative} as well. 

All the aforementioned continuations require the energy to be nonincreasing in time and are therefore governed by the following system 
\begin{subequations}\label{eq:reformulatedHS}
 \begin{align}
 	u_t +uu_x &= \frac{1}{2}F - \frac{1}{4}F_{\infty}(t), \label{eq:reform1}\\
	\mu_t  + (u\mu)_x & \leq 0, \label{eq:reform2}
\end{align} 
\end{subequations}
 where $F(t, x) = \mu(t, (-\infty, x))$ denotes the cumulative energy associated with the measure $\mu(t)$. This system does not provide us with enough information to distinguish different types of weak solutions, due to the measure-valued transport inequality. In other words, except from the particular case of equality in \eqref{eq:reform2}, which leads to conservative solutions, the above system contains no information about the exact rate of energy dissipation. 
To resolve this issue, one introduces a coordinate transformation, namely a mapping $L$, which transforms the Eulerian initial data $(\bar{u}, \bar{\mu}, \bar{\nu})$ into a quadruplet $\bar{X} = (\bar{y}, \bar{U}, \bar{V}, \bar{H})$ in Lagrangian coordinates. Here $\bar{\nu}$ is a measure added purely for technical reasons in the sense that the same solution $(u, \mu)(t)$ is recovered for any $t\geq 0$, irregardless of which $\bar \nu$ we pick initially, see \cite[Lem. 2.13]{LipschitzAlpha} for details. The crux is that, under this transformation, \eqref{eq:reformulatedHS} rewrites into a system of ODEs that attains unique and global solutions and it also gives us control of the exact loss of energy upon wave breaking. More precisely, see \cite[Sec. 2]{AlphaHS}, the time evolution of the $\alpha$-dissipative solution $X=(y,U,V,H)$ with initial data $\bar{X}=(\bar{y}, \bar{U}, \bar{V}, \bar{H})$ is governed by 
 \begin{subequations}\label{eq:LagrSystemIntro}
 \begin{align}
 	y_t(t, \xi) &= U(t, \xi), \\
	U_t(t, \xi) &= \frac{1}{2}V(t, \xi) - \frac{1}{4}V_{\infty}(t), \\
	V(t, \xi) &= \int_{-\infty}^{\xi} \big(1 - \alpha(y(\tau(\eta), \eta)) \chi_{\{\omega \mid t \geq \tau(\omega) > 0 \}}(\eta) \big) \bar{V}_{\xi}(\eta)d\eta, \\
	H_t(t, \xi) &= 0, 
\end{align}
\end{subequations}
where $V_{\infty}(t) = \displaystyle \lim_{\xi \rightarrow \infty}V(t, \xi)$ represents the total Lagrangian energy and $\tau: \R \rightarrow [0, \infty]$ is the {\em wave breaking function} given by  
\begin{align*}
 	\tau(\xi) &= \begin{cases}
	0, & \text{ if }\bar{y}_{\xi}(\xi) = 0 = \bar{U}_{\xi}(\xi), \\
	-2\frac{\bar{y}_{\xi}(\xi)}{\bar{U}_{\xi}(\xi)}, & \text{ if } \bar{U}_{\xi}(\xi) < 0, \\
	\infty, & \text{otherwise}. 
\end{cases}
\end{align*}

 Solving \eqref{eq:LagrSystemIntro} amounts to following the sought solution $(u, F)$ along generalized characteristics, and to recover $(u,F)$, one introduces yet another  nonlinear mapping $M$.  The nonlinear nature of the mappings $L$ and $M$ severely complicates the upcoming convergence rate analysis.  
 
The pointwise blow-up of $u_x$ at wave breaking introduces difficulties with {\em numerical stability}, especially for finite difference schemes, but also for standard continuous Galerkin methods. By relaxing the continuity hypothesis and allowing for discontinuities across element interfaces, the authors in \cite{DG1} introduce a local discontinuous Galerkin (LDG) method for \eqref{eq:HS}. This scheme is based on a central flux formulation and its spatial derivative has a nonincreasing $L^2(\R)$-norm, i.e., the method is {\em energy stable}. However, none of the presented numerical results experiences wave breaking, and as no rigorous convergence analysis is conducted, it is unclear whether this method is able to handle wave breaking. Moreover, the same authors revisit the LDG method in \cite{DG2}, and replace the central flux with an upwind flux, which again leads to an energy stable method. They also propose a new DG method, which turns out to coincide with the finite difference scheme developed in \cite[Sec. 6]{FDDissipative} when choosing piecewise constant functions. Since the upwind scheme in \cite[Sec. 6]{FDDissipative} converges towards dissipative solutions, even when wave breaking occurs, so does the DG method in this case.  

There is an intrinsic numerical diffusion in traditional finite difference schemes and they therefore naturally give rise to dissipative solutions as shown in \cite{FDDissipative}.  Moreover, such solutions are characterized by the following Oleinik-type condition, see \cite{DafermosCharacteristics, FDDissipative}, 
\begin{equation*}
	u_x(t, x) \leq \frac{2}{t} \hspace{1.2cm} \text{for all } t >0 \text{ and a.e. } x \in \R,
\end{equation*}
which provides additional stability and makes them more amendable for numerics. In spite of that, a scheme based on the method of characteristics, which equivalently can be expressed as a finite difference scheme, was shown to converge to conservative solutions in \cite{NumericalConservative}. Furthermore, several geometrically oriented finite difference methods were proposed in \cite{GeometricInt}, albeit without any convergence analysis.  

Very recently, a convergent numerical algorithm for $\alpha$-dissipative solutions was proposed in \cite{Alpha1} for $\alpha \in [0, 1]$ and subsequently extended to $\alpha \in W^{1, \infty}(\R, [0, 1))$ in \cite{Alpha2}. The latter method, to which we derive a convergence rate, relies on the observation that if the initial data $\bar{u}$ in \eqref{eq:HS} is piecewise linear, then so is the associated $\alpha$-dissipative solution $u(t, \cdot)$ for all $t\geq 0$. Moreover, this solution can be computed explicitly after possibly resolving a finite number of wave breaking occurrences. Since any initial data $\bar{u}$ in \eqref{eq:HS} can be approximated in $L^{\infty}(\R)$ by a sequence $\{\bar{u}_{\Dx}\}_{\Dx > 0}$ of piecewise linear functions, one introduces a piecewise linear projection operator $P_{\Dx}$ that preserves the continuity of $\bar{u}$, the total energy $\bar{F}_{\infty}$, and the relation $d\bar{\mu}_{\mathrm{ac}} = \bar{u}_x^2dx$. Especially the last property is essential to ensure that the numerical approximation, for each fixed $\Dx > 0$, dissipates energy when it is supposed to. The numerical solution is thereafter evolved by an iteration scheme that is based on computing successive approximations of the solution to \eqref{eq:LagrSystemIntro} with initial data $L \circ P_{\Dx}((\bar{u}, \bar{\mu}, \bar{\nu}))$. This is explained thoroughly in Section~\ref{sec:numMethod} and \cite[Sec. 3.2]{Alpha2}.   

Despite the existence of several numerical methods for \eqref{eq:HS}, the convergence rate derived in \cite{AlphaRate} for $\alpha\in [0,1]$ and here for $\alpha\in W^{1,\infty}(\mathbb{R}, [0,1))$ are the first {\em robust convergence rates}, i.e., error rates that persist even if wave breaking takes place. The few results that exist elsewhere either break down at wave breaking or prevent it from taking place, see \cite[Sec. 1]{AlphaRate} for a brief discussion. To be more precise, we prove that the numerical method from \cite{Alpha2} satisfies 
\begin{equation}\label{eq:rateIntro}
	\sup_{t \in [0, T]} \! \|u(t) - u_{\Dx}(t)\|_{\infty} \leq \mathcal{O}(\Dx^{\gamma}),
\end{equation}
for some $\gamma > 0$, where $T > 0$ is a fixed final simulation time, $\Dx$ denotes the spatial discretization parameter, and $\{u_{\Dx}\}_{\Dx > 0}$ is the family of numerical approximations. It suffices to this end, to derive a convergence rate in $[L^{\infty}(\R)]^2$ for the numerical Lagrangian pair $(y_{\Dx}\!-\!\id, U_{\Dx})(t)$, since by \cite[Lem. 4.11]{Alpha1}, we have
\begin{equation*}
	\sup_{t \in [0, T]} \!\|u(t) - u_{\Dx}(t)\|_{\infty} \! \leq \!\sup_{t \in [0, T]} \!\|U(t) - U_{\Dx}(t)\|_{\infty} + \sqrt{\bar{\mu}(\R)}\! \!  \sup_{t \in [0, T]}\!\|y(t) - y_{\Dx}(t)\|_{\infty}^{\nicefrac{1}{2}}. 
\end{equation*}
This is however challenging,  due to the highly nonlinear relation between the Eulerian and Lagrangian variables. In particular, the singular part $\bar{\mu}_{\mathrm{sing}}$, which is supported on a set of measure zero in Eulerian coordinates, leads to a set $\mathcal{A}$ of positive measure in Lagrangian coordinates, on which we are unable to connect the differentiated Eulerian and Lagrangian variables. As a consequence, we only have a convergence rate for the initial Lagrangian energy density, $\bar{V}_{\Dx, \xi}$,  on $\R\setminus \mathcal{A}$. However, as the norms $\|y(t) - y_{\Dx}(t)\|_{\infty}$ and $\|U(t) - U_{\Dx}(t)\|_{\infty}$ heavily depend on the difference $V_{\xi}(t) - V_{\Dx, \xi}(t)$ on all of $\R$, cf. \eqref{eq:LagrSystemIntro}, we have to prove that the set $\mathcal{A}$ contributes at most with a certain order. To achieve this, one needs to postulate additional regularity on the Eulerian initial data. In particular, we require that there exist constants $C > 0$ and $\beta \in (0, 1]$ such that 
\begin{equation}\label{eq:extraCondIntro}
	\|\bar{u}_x(\cdot + h) - \bar{u}_x(\cdot)\|_2 \leq Ch^{\beta} \hspace{0.4cm} \text{for all } h \in (0, 2].
\end{equation}
The remaining argument then hinges on the fact that, by construction, the aforementioned projection operator, $P_{\Dx}$, preserves the mass of the singular part $\bar{\mu}_{\mathrm{sing}}$ locally. This property is combined with a novel pair of transformations, inspired by \cite[Def. 6]{NVWE}, that allows us to prove that the contribution from the set $\mathcal{A}$ is at most of order $\mathcal{O}(\Dx^{\nicefrac{\beta}{2}})$. 

The way in which we proceed to prove the convergence rate shares a lot of similarities with the one proposed in \cite{AlphaRate} for $\alpha\in [0,1]$ -- the main difference is the transformation used to deal with the problematic set $\mathcal{A}$. The mapping in \cite{AlphaRate}, which can be constructed explicitly, maps the points where the exact solution breaks initially to those points where the numerical solution breaks initially, and, when combined with the local preservation of $\bar{\mu}_{\mathrm{sing}}$, this enables one to prove that the set $\mathcal{A}$ contributes at most with order $\mathcal{O}(\Dx^{\nicefrac{\beta}{2}})$.  However, the author of \cite{AlphaRate} has to assume that the initial energy measure $\bar{\mu}$ has no singular continuous part. In this work, on the other hand, we introduce a pair of transformations which rescales both the projected and the exact initial data. In this way we are able to handle any finite, positive Radon measure $\bar\mu$, but this comes at the cost of only having implicit expressions for the transformations. In spite of that, we also obtain an improved convergence rate; instead of the $\mathcal{O}(\Dx^{\nicefrac{\beta}{8}})$-rate shown in \cite[Thm. 4.11]{AlphaRate} for $\alpha\in[0,1]$, we deduce that \eqref{eq:rateIntro} holds with $\gamma = \frac{1}{4}\min \{ \beta, \frac{1}{2} \}$ for $\alpha$-dissipative solutions where $\alpha \in W^{1, \infty}(\R, [0, 1)) \cup \{1\}$, which reduces to $\mathcal{O}(\Dx^{\nicefrac{\beta}{4}})$ when $\alpha$ is a constant. The transformations used in \cite{AlphaRate} and herein are compared in Section~\ref{sec:AlternativeMap}. 

Note that even if one starts with a purely absolutely continuous energy measure, a singular continuous part may form at some later time. Thus, in order to be able to analyze numerical methods where one maps back and forth between Eulerian and Lagrangian coordinates, like in \cite{NumericalConservative}, one has to be able to treat singular continuous measures. We hope that our approach also turns out valuable in such cases. 

This paper is organized in the following way. In Section~\ref{sec:Prelim} we outline how to construct $\alpha$-dissipative solutions via the generalized method of characteristics from \cite{LipschitzConservative} and \cite{AlphaHS} and we thereafter recall the numerical method from \cite{Alpha2}. Then, in Section~\ref{sec:rateMethod}, we proceed by proving the aforementioned convergence rate. Our rigorous analysis of the pair of transformations and the subsequent change of variables play a key role. Finally, in Section~\ref{sec:NumericalExp} we conduct several numerical experiments which support our theoretical result. In particular, the cusped wave profile analyzed in \cite[Ex. 5.2]{Alpha1} and  \cite[Ex. 3]{NumericalConservative} satisfies \eqref{eq:extraCondIntro} with $\beta = \nicefrac{1}{6}$, and hence converges with order $\mathcal{O}(\Dx^{\nicefrac{1}{24}})$.

\section{Preliminaries}\label{sec:Prelim} 
The numerical method for $\alpha$-dissipative solutions, to which we derive a convergence rate, is based on the generalized method of characteristics from \cite{LipschitzConservative} and \cite{AlphaHS}.  We therefore split this section into two parts: in the first part we remind the reader of this generalized method of characteristics and the definition of $\alpha$-dissipative solutions, which is followed by a presentation of the numerical method proposed in \cite{Alpha2}.

\subsection{The generalized method of characteristics and $\alpha$-dissipative solutions}
Let $\alpha \in W^{1, \infty}(\R, [0, 1))$ and denote by $\M^+(\R)$ the space of finite, positive Radon measures on $\R$. In order to define the set of admissible initial data, we have to recall some function spaces from \cite{LipschitzConservative} and \cite{AlphaHS}. 

Let us start by introducing the Banach space 
\begin{equation*}
	E:= \{f \in L^{\infty}(\R) \!\mid f' \in L^2(\R) \} \hspace{0.35 cm} \text{endowed with} \hspace{0.35 cm} \|f\|_E := \|f\|_{\infty} + \|f'\|_2, 
\end{equation*}
and let 
\begin{equation*}
	H_d^1(\R) := H^1(\R) \times \R^d, \hspace{0.35cm} d = 1, 2. 
\end{equation*}
Furthermore, introduce a partition of unity on $(-\infty, 1) \cup (-1, \infty) = \R$, i.e., a pair of functions $\phi^-$ and $\phi^+$ belonging to $C^{\infty}(\R)$, which satisfies
\begin{enumerate}[label=(\roman*)]
	\item $\phi^+ + \phi^- = 1 $, 
	\item $\mathrm{supp}(\phi^+) \subset (-1, \infty)$ and $\mathrm{supp}(\phi^-) \subset (-\infty, 1)$,
	\item $0 \leq \phi^{\pm} \leq 1$.  
\end{enumerate}
This pair is used to define the following linear, continuous and injective mappings 
\begin{align*}
	R_1\!: H_1^1(\R) &\rightarrow E, \hspace{0.5cm} (\bar{f}, a ) \mapsto f= \bar{f} + a \phi^+, \\
	R_2\!:H_2^1(\R) &\rightarrow E, \hspace{0.6cm} (\bar{f}, a, b) \mapsto f=\bar{f} + a\phi^+ + b\phi^-,
\end{align*}
which allow us to introduce the Banach spaces $E_1$ and $E_2$ as the images of $H_1^1(\R)$ and $H_2^1(\R)$ under the mappings $R_1$ and $R_2$, respectively, that is, 
\begin{align*}
	E_1 := R_1(H_1^1(\R)) \hspace{0.5cm} \text{and} \hspace{0.5cm} E_2:= R_2(H_2^1(\R)), 
\end{align*}
equipped with the norms
\begin{align*}
	\| f\|_{E_1}\!:&\!= \|\bar{f} + a \phi^+ \|_{E_1} = \Big(\|\bar{f}\|_{H^1(\R)}^2 + a^2 \Big)^{\nicefrac{1}{2}}\!, \\
	\| f\|_{E_2}\!:&\!= \|\bar{f} + a \phi^+ + b\phi^-\|_{E_2} = \Big(\|\bar{f}\|_{H^1(\R)}^2 + a^2 +b^2 \Big)^{\nicefrac{1}{2}}\! \!.
\end{align*}
It was shown in \cite[Sec. 2]{NonVanishingAsymptotes} that the spaces $E_1$ and $E_2$ are independent of the chosen partition of unity. In addition, $R_1$ is also well-defined when applied to functions belonging to $L_1^2(\R) = L^2(\R) \times \R$ and we therefore introduce
\begin{equation*}
	E_1^0\! := R_1(L_1^2(\R)) \hspace{0.25cm} \text{endowed with} \hspace{0.25cm} \|f\|_{E_1^0} := \Big(\|\bar{f}\|_2^2 + a^2\Big)^{\nicefrac{1}{2}} \! \!. 
\end{equation*}

With this in place, we can finally define the set of Eulerian coordinates $\D^{\alpha}$, which contains all the admissible initial data to \eqref{eq:reformulatedHS}. 

\begin{definition}\label{def:EulSet}
   The space $\D^{\alpha}$ consists of all triplets $(u, \mu, \nu)$ that satisfy 
    \begin{enumerate}[label=(\roman*)]
        \item $u \in E_2$,
        \item $\mu \leq \nu \in \M^{+}(\R)$, \label{def:EulSet2}
        \item $\mu_{\mathrm{ac}} \leq \nu_{\mathrm{ac}}$,
        \item $d\mu_{\mathrm{ac}} = u_x^2dx$, \label{def:EulSet4}
        \item $\mu( (-\infty, \cdot)) \in E_1^0$,
        \item $\nu((-\infty, \cdot))  \in E_1^0$,
        \item $\frac{d\mu}{d\nu}(x) > 0$, and $\frac{d\mu_{\mathrm{ac}}}{d\nu_{\mathrm{ac}}}(x) = 1$ if $u_x(x) < 0$. 
    \end{enumerate}
    Moreover, we define
   \begin{equation}\label{eq:D0}
    	\D_0^{\alpha} = \{(u, \mu, \nu) \in \D^{\alpha} \! \mid \mu = \nu \}. 
   \end{equation}
\end{definition}

By \cite[Thm. 1.16]{RealFolland}, there is a one-to-one relation between $\mu$ and the cumulative energy $F(x) = \mu((-\infty, x))$, and likewise between $\nu$ and $G(x) = \nu((-\infty, x))$. These functions are bounded, left-continuous, increasing, and satisfy 
\begin{equation*}
	\lim_{x \rightarrow -\infty} F(x) =0 = \lim_{x \rightarrow -\infty}G(x),
\end{equation*}
in addition to 
\begin{equation*}
	F_{\infty} = \lim_{x \rightarrow \infty}F(x) = \mu(\R) \hspace{0.35cm} \text{and} \hspace{0.35cm}  G_{\infty} = \lim_{x \rightarrow \infty} G(x) = \nu(\R). 
\end{equation*}

Furthermore, recall that any $\mu \in \mathcal{M}^+(\R)$ admits a decomposition of the form 
\begin{equation}\label{eq:basicDecomp}
	\mu = \mu_{\mathrm{ac}} + \mu_{\mathrm{sing}},
\end{equation}
where $\mu_{\mathrm{ac}}$ is absolutely continuous and $\mu_{\mathrm{sing}}$ is singular with respect to the Lebesgue measure, see e.g., \cite[Thm. 3.8]{RealFolland}. We also take advantage of the fact that $\mu_{\mathrm{sing}}$ can be decomposed further into two mutually singular measures
\begin{equation}\label{eq:refinedDecomp}
	\mu_{\mathrm{sing}} = \mu_{\mathrm{d}} + \mu_{\mathrm{sc}},
\end{equation}
where $\mu_{\mathrm{d}}$ is purely discrete and $\mu_{\mathrm{sc}}$ is singular continuous, see \cite[Thm. 9.7]{McDonaldAnalysis}.  A consequence of \eqref{eq:basicDecomp} and \eqref{eq:refinedDecomp} is that $F$ can be decomposed into three parts, 
\begin{equation}\label{eq:refinedDecompF}
	F(x) = F_{\mathrm{ac}}(x)  + F_{\mathrm{d}}(x) + F_{\mathrm{sc}}(x), 
\end{equation}
where $F_{\mathrm{ac}}(x) = \mu_{\mathrm{ac}}((-\infty, x))$ is absolutely continuous, $F_{\mathrm{d}}(x) = \mu_{\mathrm{d}}((-\infty, x))$ is a piecewise constant function, and $F_{\mathrm{sc}}(x) = \mu_{\mathrm{sc}}((-\infty, x))$ is continuous with $F'_{\mathrm{sc}} = 0$ a.e.. Furthermore, $F_{\mathrm{sing}}(x) = \mu_{\mathrm{sing}}((-\infty, x)) = F_{\mathrm{d}}(x) + F_{\mathrm{sc}}(x)$. The function $G(x) = \nu((-\infty, x))$ admits a similar decomposition.

Next, let us define the set of Lagrangian coordinates $\F^{\alpha}$. This requires us to introduce $E^4 = E_2 \times E_2 \times E_1 \times E_1$, which is a Banach space when equipped with 
\begin{equation}\label{eq:normB}
	\|(f_1, f_2, f_3, f_4)\|_{E^4} := \|f_1\|_{E_2} + \|f_2\|_{E_2} + \|f_3\|_{E_1} + \|f_4\|_{E_1}.
\end{equation}

\begin{definition}\label{def:LagSet} The space $\F^{\alpha}$ is composed of all quadruplets $X = (y, U, V, H)$ satisfying 
\begin{enumerate}[label=(\roman*)]
  \item $(y-\id, U, V, H) \in E^{4}\cap \big[W^{1, \infty}(\R) \big]^4 $,
  \item $ y_{\xi}, H_{\xi} \geq 0$ a.e. and there exists $c > 0$ such that $y_{\xi} + H_{\xi} \geq c $ a.e., \label{def:LagSet2}
  \item \label{def:LagSet3} $y_{\xi}V_{\xi}= U_{\xi}^2 $ a.e., 
  \item $0 \leq V_{\xi} \leq H_{\xi}$ a.e., 
  \item there exists $\kappa\! :\! \R\! \rightarrow \!(0, 1]$ such that $V_{\xi}(\xi) = \kappa(y(\xi))H_{\xi}(\xi)$ a.e., with \phantom{aa} $\kappa(y(\xi))=1$ whenever $U_{\xi}(\xi) < 0$. 
\end{enumerate}
Furthermore, we define  
\begin{equation}\label{eq:subsetsF}
	\F_0^{\alpha} := \{X \in \F^{\alpha}\!\mid y + H = \id \} \hspace{0.35cm} \text{and} \hspace{0.35cm} \F_0^{\alpha, 0} := \{X \in \F_0^{\alpha} \! \mid H = V \}. 
\end{equation}
\end{definition}

Typically, given some initial data $(\bar{u}, \bar{\mu}, \bar{\nu})$ belonging to $\D^{\alpha}$, one wants to solve \eqref{eq:reformulatedHS} combined with $\nu_t + (u\nu)_x = 0$. Moreover, to ensure that the correct amount of energy is dissipated, which is specified by $\alpha$, one uses a generalized method of characteristics. Hence, let us introduce the following mapping, which is well-defined by \cite[Prop. 2.1.5]{PhdThesisNordli}, to define the initial data in Lagrangian coordinates, $\bar X=(\bar y, \bar U, \bar V, \bar H)$.  
 
\begin{definition}\label{def:MapL}Let $L: \D^{\alpha} \rightarrow \F_0^{\alpha}$ be defined by $L((u, \mu, \nu)) = (y, U, V, H)$, where
\begin{subequations}
\begin{align}
    y(\xi) &= \sup \{x \in \R\! \mid x + \nu((-\infty, x)) < \xi \}, \label{eq:L_eq1}\\
    U(\xi) &= u (y(\xi)), \label{eq:L_eq2} \\
    H(\xi) &= \xi - y(\xi), \label{eq:L_eq3} \\
    V(\xi) &= \int_{-\infty}^{\xi} \frac{d\mu}{d\nu} (y(\eta)) H_{\xi}(\eta) d\eta \label{eq:L_eq4}. 
\end{align}
\end{subequations}
\end{definition}

In Lagrangian coordinates, the $\alpha$-dissipative solution with initial data $\bar{X} = (\bar{y}, \bar{U}, \bar{V}, \bar{H})\in \F^{\alpha}$ satisfies the following system of differential equations 
\begin{subequations}
\label{eq:LagrSystem}
\begin{align}
    y_t(t, \xi) &= U(t, \xi), \label{eq:ODE1} 
    \\
    U_t(t, \xi) &= \frac{1}{2}V(t, \xi) - \frac{1}{4}V_{\infty}(t), \label{eq:ODE2} 
    \\
    V(t, \xi) &= \int_{-\infty}^{\xi}\!\big(1 - \alpha(y(\tau(\eta), \eta)) \chi_{\{\omega \mid t \geq \tau(\omega) > 0\}}(\eta) \big) \bar{V}_{\xi}(\eta) d\eta, \label{eq:ODE3} 
    \\
    H_t(t, \xi) &= 0. \label{eq:ODE4}
\end{align}
\end{subequations}
Here $V_{\infty}(t) = \displaystyle \lim_{\xi \rightarrow \infty}V(t, \xi)$ represents the total Lagrangian energy and $\tau: \R \rightarrow [0, \infty]$ is the {\it wave breaking function} given by 
\begin{align}\label{eq:tau}
	\tau(\xi) = \begin{cases}
		0, & \text{ if }\bar{y}_{\xi}(\xi) = \bar{U}_{\xi}(\xi) = 0, \\
		- \frac{2\bar{y}_{\xi}(\xi)}{\bar{U}_{\xi}(\xi)}, & \text{ if} \bar{U}_{\xi}(\xi) < 0, \\
		\infty, & \text{ otherwise},
	\end{cases}
\end{align}
which tells us if and when the $\alpha$-dissipative solution experiences wave breaking along the characteristic labeled by $\xi \in \R$. 

It has been established in \cite[Lem. 2.2.1]{PhdThesisNordli} that \eqref{eq:LagrSystem}--\eqref{eq:tau} is globally well-posed for initial data belonging to $\F^{\alpha}$. Consequently, associated with \eqref{eq:LagrSystem}--\eqref{eq:tau}, we define the Lagrangian solution operator $S_t$, for which $\F^{\alpha}$ is an {\em invariant set}.

\begin{definition}\label{def:LagrOp}
	Define, for any $t\geq 0$, $S_t: \F^{\alpha} \rightarrow \F^{\alpha}$ by $X(t) = S_t(\bar{X})$, where $X(t)$ is the unique solution of \eqref{eq:LagrSystem}--\eqref{eq:tau} with initial data $X(0) = \bar{X} \in \F^{\alpha}$. 
\end{definition}

At last, in order to obtain the sought solution of \eqref{eq:reformulatedHS}, we have to transform the solution of \eqref{eq:LagrSystem}--\eqref{eq:tau} back to Eulerian coordinates, and the following mapping, which is well-defined by \cite[Lem. 2.1.7]{PhdThesisNordli}, is needed for this very purpose. 

\begin{definition}\label{def:MapM}
	Define $M\!: \F^{\alpha} \rightarrow \D^{\alpha}$ as $M((y, U, V, H)) = (u, \mu, \nu)$, where 
	\begin{subequations}
	\begin{align}
		u(x) &= U(\xi)  \text{ for any } \xi \in \R \text{ such that} x=y(\xi), \\
		\mu &= y_{\#}(V_{\xi}d\xi), \\
		\nu &= y_{\#}(H_{\xi}d\xi).
	\end{align}
	\end{subequations}
	Here $y_{\#}(V_{\xi}d\xi)$ is the pushfoward measure given by 
	\begin{equation*}
		\mu(A) = \int_{y^{-1}(A)} \!V_{\xi}(\xi)d\xi, \hspace{0.35cm} \text{for all Borel sets} A \subseteq \R,
	\end{equation*}
	and the measure $y_{\#}(H_{\xi}d\xi)$ is defined similarly.  
\end{definition}

By composing the mappings $L$ and $M$ with the Lagrangian solution operator $S_t$, one can construct $\alpha$-dissipative solutions, leading to the following definition.

\begin{definition}\label{def:AlphaSol}  Suppose $(\bar{u}, \bar{\mu}, \bar{\nu})\! \in \D^{\alpha}$, the $\alpha$-dissipative solution at time $t \geq 0$ is given by 
\begin{equation*}
	(u, \mu, \nu)(t) = T_t((\bar{u}, \bar{\mu}, \bar{\nu})) = M \circ S_t \circ L ((\bar{u}, \bar{\mu}, \bar{\nu})). 
\end{equation*}
\end{definition}

Note that there are three variables in Eulerian coordinates, while there are four in Lagrangian coordinates, implying that the mapping $M\! :\!\F^{\alpha} \rightarrow \D^{\alpha}$ cannot be injective. However, one can identify an equivalence relation on $\F^{\alpha}$, such that all elements belonging to the same equivalence class are mapped to the same element in Eulerian coordinates. To this end, let us define a group that acts on $\F^{\alpha}$ through {\em relabeling}. 

\begin{definition}\label{def:RelabelingGroup}
	Let $\mathcal{G}$ be the group of homeomorphisms $f:\R \rightarrow \R$ satisfying 	\begin{enumerate}[label=(\roman*)]
		\item $f- \id$, $f^{-1} -\id \in E_2$, 
		\item $f_{\xi}-1$, $(f^{-1})_{\xi}-1  \in L^\infty(\R)$. 
	\end{enumerate}
	Furthermore, define the group action $\bullet: \F^{\alpha} \times \mathcal{G} \rightarrow \F^{\alpha}$ by 
	\begin{equation}\label{eq:relabel}
		(X, f) \mapsto (y(f), U(f), V(f), H(f)) = X \bullet f. 
	\end{equation} 
\end{definition}
As mentioned, all elements belonging to the equivalence class
\begin{equation*}
	[X] := \{\hat{X} \in \F^{\alpha}\! \mid \text{there exists } f \in \mathcal{G} \text{ s.t. } \!\hat{X} = X \bullet f \},
\end{equation*}
for some $X \in \F^{\alpha}$, are mapped to the same element in Eulerian coordinates.

\begin{prop}[{\cite[Prop. 2.1.10]{PhdThesisNordli}}]
	Let $f \in \mathcal{G}$ and $X \in \F^{\alpha}$, then 
	\begin{equation*}
		M(X \bullet f) = M(X).
	\end{equation*} 
\end{prop}

The pair $(u, \mu)$ contains all the essential information about the solution in Eulerian coordinates, and $(y, U, V)$ therefore encodes the important information in Lagrangian coordinates. Yet, it is apparent from Definition~\ref{def:MapL}, that the choice of $\nu$ not only influences $H$, but also $(y, U, V)$. Despite this, the pair $(u,\mu)$ is independent of $\nu$.

\begin{lemma}[{\cite[Lem. 2.13]{LipschitzAlpha}}]
	Given $(\bar{u}_j, \bar{\mu}_j, \bar{\nu}_j)$ in $\D^{\alpha}$ for $j \in \{1, 2\}$ with $\bar{u}_1 = \bar{u}_2$ and $\bar{\mu}_1 = \bar{\mu}_2$, let $(u_j, \mu_j, \nu_j)(t) = T_t((\bar{u}_j, \bar{\mu}_j, \bar{\nu}_j))$ for $j \in \{1, 2\}$ and any $t \geq 0$,  then 
	\begin{equation*}
		u_1(t, \cdot) = u_2(t, \cdot) \hspace{0.5cm} \text{and} \hspace{0.5cm} \mu_1(t) = \mu_2(t). 
	\end{equation*}
\end{lemma} 
As a consequence, we will restrict our attention to initial data belonging to the subset $\D_0^{\alpha} \subset \D^{\alpha}$, defined in \eqref{eq:D0}, which in particular implies that $\bar{\mu} = \bar{\nu}$. 

\subsection{The numerical method}\label{sec:numMethod}

Let us proceed by describing the numerical method proposed in \cite{Alpha2}, which is closely related to  Definition~\ref{def:AlphaSol}, but differs in two ways: 
\begin{enumerate}[label=(\roman*)]
	\item Before mapping the Eulerian initial data to Lagrangian coordinates, one applies the piecewise linear projection operator introduced in \cite[Def. 3.2]{Alpha1}. The $\alpha$-dissipative solution associated with this projected initial data will remain piecewise linear for all $t\geq 0$. This makes it possible to implement the solution operator $T_t$ exactly in the case of $\alpha$ being a fixed constant, see \cite{Alpha1}. However, this is no longer possible when $\alpha \in W^{1, \infty}(\R, [0, 1))$. 
	\item If the $\alpha$-dissipative solution experiences wave breaking along the characteristic labeled by $\xi$, then the fraction of energy to be removed equals $\alpha(y(\tau(\xi), \xi))$. Moreover, since the time evolution of $y(t, \xi)$ depends on $V_{\xi}(t, \cdot)$ on all of $\R$,  cf. \eqref{eq:LagrSystem}, the position at which wave breaking takes place, $y(\tau(\xi), \xi)$, is influenced by all the other wave breaking occurrences happening prior to $t=\tau(\xi)$. Thus the time evolutions of $y(t, \xi)$ and $V_{\xi}(t, \cdot)$ are closely intertwined and \eqref{eq:LagrSystem} cannot be solved explicitly in general. To resolve this issue, one approximates the solution operator $S_t$ numerically by introducing an iteration scheme, which is based on computing successive approximations of the energy to be removed. 
\end{enumerate}

\subsubsection{The projection operator}
We start by recalling the projection operator introduced in \cite[Def. 3.2]{Alpha1}. To this end, let $\{x_j\}_{j \in \mathbb{Z}}$ represent a uniform mesh on $\R$ with $x_j = j\Dx$ for $j \in \mathbb{Z}$ and $\Dx > 0$ fixed. In addition, let $f_j = f(x_j)$ denote the values attained at these gridpoints for any function $f\!:\!\R \rightarrow \R$  and introduce the difference operator 
\begin{equation*}
	Df_{2j} = \frac{f_{2j+2} - f_{2j}}{2\Dx}.  
\end{equation*}
The projection operator $P_{\Dx}$, which is tailored to preserve the total energy $F_{\infty}$, the continuity of $u$, and Definition~\ref{def:EulSet}~\ref{def:EulSet4}, is then defined as follows. 

 \begin{definition}\label{def:ProjOperator}
    Define $P_{\Dx}\! \!:  \!\D_0^{\alpha}\! \rightarrow\! \D_0^{\alpha}$ by $P_{\Delta x} \!\left((u, F, G) \right)\! =\!\!(u_{\Delta x}, F_{\Delta x}, G_{\Delta x})$, where  
    \begin{align}
        u_{\Delta x}(x) &= \begin{cases}
        u_{2j} + \left(Du_{2j} \mp q_{2j} \right)(x-x_{2j}), & x_{2j} < x \leq x_{2j+1}, 
        \\ 
        u_{2j+2} + \left(Du_{2j} \pm q_{2j} \right) \left(x-x_{2j+2} \right) \!, & x_{2j+1} < x \leq x_{2j+2}, 
        \end{cases}
        \label{eq:up}
    \end{align}
    with $q_{2j}$ given by 
    \begin{equation}\label{eq:q2j}
    	q_{2j}:= \sqrt{DF_{\mathrm{ac}, 2j} - (Du_{2j})^2}, 
   \end{equation}
    and
    \begin{align}
        G_{\Delta x}(x) &= F_{\Delta x}(x) = F_{\Delta x, \mathrm{ac}}(x) + F_{\Delta x, \mathrm{sing}}(x), \hspace{0.5cm} \text{for all} x \in \R. 
        \label{eq:Fp}
    \end{align}
    Furthermore, the absolutely continuous part of $F_{\Dx}$ is defined by 
    \begin{align}\label{eq:Facp} 
    F_{\Delta x, \mathrm{ac}}(x) &= \begin{cases} F_{\mathrm{ac}, 2j} + \left(Du_{2j} \mp q_{2j}\right)^2 \left(x - x_{2j} \right)\!, & x_{2j} < x \leq x_{2j+1}, 
    \\ 
    \frac{1}{2}\left(F_{\mathrm{ac}, 2j+2} + F_{\mathrm{ac}, 2j}\right) \mp 2Du_{2j} q_{2j}\Delta x \\ \quad + \left(Du_{2j} \pm q_{2j} \right)^2(x-x_{2j+1}), & x_{2j+1} < x \leq x_{2j+2},
    \end{cases}
    \end{align}
    while the singular part is given by 
    \begin{equation}\label{eq:Fsp}
        F_{\Delta x, \mathrm{sing}}(x) = F_{2j+2} - F_{\mathrm{ac}, 2j+2}= F_{\mathrm{sing}, 2j+2}, \hspace{0.8cm}  x_{2j} < x \leq x_{2j+2}.
    \end{equation}
\end{definition}

Note that there is one degree of freedom for each interval $[x_{2j}, x_{2j+2}]$, that is, we have to determine which sign to use in \eqref{eq:up} and \eqref{eq:Facp}. This choice will not influence our upcoming analysis, but it affects the numerical accuracy in practical applications. For the numerical experiments in Section~\ref{sec:NumericalExp}, we therefore couple Definition~\ref{def:ProjOperator} with the sign-selection criterion from \cite[Sec. 3.1]{Alpha1}. This criterion is based on picking, for each $[x_{2j}, x_{2j+2}]$, the sign that minimizes the distance between $u(x_{2j+1})$ and $u_{\Dx}(x_{2j+1})$, which is given by $(-1)^{k_{2j}}$ for $[x_{2j}, x_{2j+1}]$, where 
\begin{equation*}
	k_{2j} := \argmin_{m \in \{0, 1\}} \left \{ \Big(\frac{u_{2j+1}-u_{2j}}{\Dx} - Du_{2j}\Big)\Dx + (-1)^{m+1}q_{2j}\Dx \right \} \!. 
\end{equation*}

\subsubsection{Implementation of $L$ and computing $\tau_{\Dx}$}
After applying $L$ from Definition~\ref{def:MapL} to the projected data, one obtains the numerical Lagrangian initial data
\begin{equation*}
	X_{\Dx} = (y_{\Dx}, U_{\Dx}, V_{\Dx}, H_{\Dx}) =  L \circ P_{\Dx}((u, \mu, \nu)).
\end{equation*}
Due to Definition~\ref{def:MapL} and Definition~\ref{def:ProjOperator}, this is a tuple of continuous and piecewise linear functions, whose nodes are located at the points $\{\xi_j\}_{j \in \mathbb{Z}}$. The value attained by $X_{\Dx}(\xi)$ at any $\xi \in \R$ can therefore be computed by linear interpolation based on the sequence $\{X_{\Dx}(\xi_j)\}_{j \in \mathbb{Z}}$. Here $\{\xi_j\}_{j \in \mathbb{Z}}$ represents a nonuniform discretization in Lagrangian coordinates, which is related to the uniform mesh $\{x_j\}_{j \in \mathbb{Z}}$ in Eulerian coordinates via
\begin{align}\label{eq:lagrPoints}
	\xi_{3j} &= x_{2j} + G_{\Dx}(x_{2j}),  \hspace{1.6cm} \xi_{3j+1} = x_{2j} + G_{\Dx}(x_{2j}+), \nonumber \\
	\xi_{3j+2} &= x_{2j+1} + G_{\Dx}(x_{2j+1}), \hspace{1.cm} \xi_{3j+3} = x_{2j+2} + G_{\Dx}(x_{2j+2}).
\end{align} 
Here $x_{2j}+$ denotes the limit from the right at $x_{2j}$. 

Furthermore, by inserting $X_{\Dx}$ into \eqref{eq:tau}, one obtains the numerical wave breaking function $\tau_{\Dx}\!: \R \rightarrow [0, \infty]$, which, for $\xi \in [\xi_{3j}, \xi_{3j+3})$, reads
\begin{align}\label{eq:numBreaking}
	\tau_{\Dx}(\xi) = \begin{cases}
	\tau_{3j+\frac{1}{2}}, & \xi \in [\xi_{3j}, \xi_{3j+1}], \\
	\tau_{3j+\frac{3}{2}}, & \xi \in (\xi_{3j+1}, \xi_{3j+2}], \\
	\tau_{3j+\frac{5}{2}}, & \xi \in (\xi_{3j+2}, \xi_{3j+3}),
	\end{cases}
\end{align}
where
\begin{equation*}
	\tau_{j+\frac{1}{2}} = \tau_{\Dx}\Big(\frac{1}{2}(\xi_{j+1}+ \xi_{j})\Big) \hspace{0.45cm} \text{for any } j \in \mathbb{Z}. 
\end{equation*}
In particular, $\xi_{3j} = \xi_{3j+1}$ whenever $G_{\Dx}$ does not have a jump discontinuity at $x=x_{2j}$ and, in addition, $\tau_{3j+\frac{1}{2}} = 0$.  By comparing \eqref{eq:tau} and \eqref{eq:numBreaking} we infer that the exact and numerical wave breaking times in general differ, and this complicates the upcoming error analysis. The interested reader is referred to \cite[Sec. 3.2.1--3.2.2]{Alpha1} for more details. 

\subsubsection{A brief recap of the time evolution}\label{Sec:OutlineSt}
The numerical time evolution is throughly described in \cite[Sec. 3.2.2]{Alpha2}, we therefore only give a brief summary. 

One starts by extracting a finite, increasing sequence of numerical breaking times, $\{\tau_k^*\}_{k=0}^N$ from $\{\tau_{j+\frac12}\}_{j\in \mathbb{Z}}$, which constitutes a non-uniform temporal discretization of $[0, T]$, for any finite $T\! > 0$. Thereafter, one evolves the differentiated Lagrangian variables, $X_{\Dx, \xi}$, between successive times in this sequence by an iteration scheme. This turns out to be much more efficient than solving \eqref{eq:LagrSystem} directly with initial data $\bar{X}_{\Dx}= L \circ P_{\Dx}((\bar{u}, \bar{\mu}, \bar{\nu}))$, which is possible as the exact Lagrangian solution, with initial data $\bar{X}_{\Dx}$,  remains piecewise linear for all $t \geq 0$. 

For the upcoming analysis, the particular values of the $\tau_k^*$'s do not matter. The important point is that $\{\tau_k^*\}_{k =0 }^N$ ensures that the iteration scheme, which we now define inductively, is a contraction and hence converges. 

Suppose
\begin{equation*}
	\tilde{X}_{\Dx, \xi}(t) = (\zeta_{\Dx, \xi}, U_{\Dx, \xi}, V_{\Dx, \xi}, H_{\Dx, \xi})(t), \hspace{0.35cm} \text{where} \hspace{0.35cm} \zeta_{\Dx}(\xi) = y_{\Dx}(\xi) - \xi, 
\end{equation*}
has been computed for $t \in [0, \tau_k^*]$ and denote by $\{\tilde{X}_{\Dx, \xi}^n(t, \cdot) \}_{n \in \mathbb{N}}$ the sequence whose time evolution, for $\xi \in \R \setminus \{\xi_j\}_{j \in \mathbb{Z}}$ and $t \in (\tau_k^*, \tau_{k+1}^*]$, is given by 
\vspace{-0.05cm}
\begin{subequations}\label{eq:iterationSeq}
\begin{align}
	\zeta_{\Dx, t \xi}^{n}(t, \xi) &=  U_{\Dx, \xi}^{n}(t, \xi), \\
	U_{\Dx, t \xi}^{n}(t, \xi) &= \frac{1}{2}V_{\Dx, \xi}^{n}(t, \xi), \\
	V_{\Dx, \xi}^{n}(t, \xi)  &= \big(1 - \beta_{\Dx}^{n}(t, \xi) \chi_{\{s\mid s\geq \tau_{\Dx}(\xi) > \tau_k^*\}}(t)\big)V_{\Dx, \xi}(\tau_k^*, \xi), \\
	H_{\Dx, t \xi}^{n}(t, \xi) &= 0, 
\end{align}
\end{subequations}
with $\tilde{X}_{\Dx, \xi}^{n}(\tau_k^*) = \tilde{X}_{\Dx, \xi}(\tau_k^*)$ for all $n\geq 1$, where
\begin{equation}\label{eq:beta1}
	\beta^1_{\Dx}(t, \xi) = 0 \hspace{0.65cm} \text{for } \xi \in \R, 
\end{equation}
and for $n\geq 2$, 
\begin{equation}\label{eq:beta2}
	\beta_{\Dx}^{n}(t, \xi) = \beta_{j+\frac{1}{2}}^{n}(t) \hspace{0.65cm} \text{for any} \xi \in (\xi_j, \xi_{j+1}],
\end{equation}
with 
\begin{align}\label{eq:beta3}
	\beta_{j+\frac{1}{2}}^{n}(t) &= \begin{cases}
		0, & \tau_{j+\frac{1}{2}} \notin (\tau_k^*, \tau_{k+1}^*], \\
		\alpha(y_{\Dx}^{n-1}(\tau_{k+1}^*, \xi_j)), &\tau_{j+\frac{1}{2}} \in (\tau_k^*, \tau_{k+1}^*].
\end{cases}
\end{align}
Since $\tilde{X}_{\Dx, \xi}^n(t,\cdot)$ is piecewise constant with possible jump discontinuities at $\{\xi_j\}_{j \in \mathbb{Z}}$, for all $t \in [\tau_k^*, \tau_{k+1}^*]$, it suffices to solve \eqref{eq:iterationSeq} at the midpoints  $\xi_{j+\frac{1}{2}} \!=\!\frac{1}{2}(\xi_{j} + \xi_{j+1})$. Moreover, it is evident from \eqref{eq:iterationSeq}--\eqref{eq:beta3} that we need to compute $y_{\Dx}^n(t, \cdot)$ and hence $X_{\Dx}^n(t, \cdot)$ at the times $\{\tau_k^*\}_{k=0}^N$. To this end, introduce for each $j \in \mathbb{Z}$
\begin{align*}
	\big(\zeta_{j+\frac{1}{2}\hspace{-0.05cm}, \hspace{0.05cm} \xi}^n, U_{j+\frac{1}{2}\hspace{-0.05cm}, \hspace{0.05cm} \xi}^n, V_{j+\frac{1}{2}\hspace{-0.05cm}, \hspace{0.05cm} \xi}^n, H_{j+\frac{1}{2}\hspace{-0.05cm}, \hspace{0.05cm} \xi}^n\big)(t) &= \tilde{X}_{j+\frac{1}{2}\hspace{-0.05cm}, \hspace{0.05cm} \xi}^n(t) 
	\\ &= \begin{cases}
	\tilde{X}_{\Dx\hspace{-0.025cm}, \hspace{0.05cm} \xi}^n\Big(t, \frac{1}{2}(\xi_{j+1} + \xi_j)\Big),  & \text{if } \xi_j \neq \xi_{j+1}, \\
	(0, 0, 0, 0) & \text{otherwise}, 
	\end{cases}
\end{align*}
and observe that $H_{\Dx}^{n}(t) = \bar{H}_{\Dx}$.  In order to compute $V_{\Dx}^{n}(t)$, $U_{\Dx}^{n}(t)$, and $y_{\Dx}^{n}(t)$, we use the recursive procedure from \cite[Sec. 3.2.2]{Alpha2}. It follows from \eqref{eq:ODE3} that$\lim_{\xi \rightarrow -\infty}V_{\Dx}^{n}(t, \xi) = 0$ and one therefore has
\begin{equation*}
	V_{\Dx}^{n}(t, \xi_j) = \sum_{m =-\infty}^{j-1} \! \!V_{m+\frac{1}{2}\hspace{-0.05cm}, \hspace{0.05cm}\xi}^{n}(t)(\xi_{m+1} - \xi_m) = V_{\Dx}^{n}(t, \xi_{j-1}) + V_{j-\frac{1}{2}\hspace{-0.05cm}, \hspace{0.05cm} \xi}^{n}(t)(\xi_j - \xi_{j-1}). 
\end{equation*}
The left asymptotes of $U_{\Dx}^{n}(t)$ and $y_{\Dx}^{n}(t)$ on the other hand, are time-dependent and evolve according to \eqref{eq:LagrSystem}. For brevity, let us introduce $f_{\pm \infty} = \lim_{\xi \rightarrow \pm \infty} f(\xi)$, then 
\begin{align*}
	U_{\Dx, -\infty}^{n}(t) &= U_{\Dx, -\infty}(\tau_k^*)  - \frac{1}{4}V_{\Dx, \infty}(\tau_k^*)(t-\tau_k^*) 
	\\ & \quad + \frac{1}{4}\sum_{m \in \mathbb{Z}} \!\beta_{m+\frac{1}{2}}^{n}(t)\chi_{\{l \mid t \geq \tau_{l+\frac{1}{2}} > \tau_k^* \}}(m)V_{m + \frac{1}{2}\hspace{-0.05cm}, \hspace{0.05cm}\xi}(\tau_k^*)(\xi_{m+1} - \xi_m) (t - \tau_{m+\frac{1}{2}}),
\end{align*}
while 
\begin{align*}
	\zeta_{\Dx, -\infty}^{n}(t)&= \zeta_{\Dx, -\infty}(\tau_k^*) + U_{\Dx, -\infty}(\tau_k^*)(t-\tau_k^*) - \frac{1}{8}V_{\Dx, -\infty}(\tau_k^*)(t- \tau_k^*)^2
	\\ & \quad \hspace{-0.2cm}+ \frac{1}{8}\sum_{m \in \mathbb{Z}} \! \beta_{m+\frac{1}{2}}^{n}(t)\chi_{\{l \mid t \geq \tau_{l+\frac{1}{2}} > \tau_k^* \}}(m)V_{m + \frac{1}{2}\hspace{-0.05cm}, \hspace{0.05cm} \xi}(\tau_k^*)(\xi_{m+1} - \xi_m) (t-\tau_{m+\frac{1}{2}})^2. 
\end{align*}
From these asymptotes we compute the functions $U_{\Dx}^{n}(t)$ and $\zeta_{\Dx}^{n}(t)$ at the gridpoints $\{\xi_j \}_{j \in \mathbb{Z}}$ and thereafter recover their value at any intermediate point $\xi$ by linear interpolation. In particular,  
\begin{align*}
	U_{\Dx}^{n}(t, \xi_j) &= U_{\Dx, -\infty}^{n}(t) + \sum_{m=-\infty}^{j-1}U_{m+\frac{1}{2}\hspace{-0.05cm}, \hspace{0.05cm}\xi}^n(t)(\xi_{m+1}-\xi_m) 
	\\ &= U_{\Dx}^{n}(t, \xi_{j-1}) + U_{j-\frac{1}{2}\hspace{-0.05cm}, \hspace{0.05cm} \xi}^n(t)(\xi_j - \xi_{j-1}),
\end{align*}
and
\begin{align*}
	\zeta_{\Dx}^{n}(t, \xi_j) &= \zeta_{\Dx, -\infty}^{n}(t) + \sum_{m=-\infty}^{j-1}\zeta_{m+\frac{1}{2}\hspace{-0.05cm}, \hspace{0.05cm} \xi}^n(t)(\xi_{m+1}-\xi_m) 
	\\ &= \zeta_{\Dx}^{n}(t, \xi_{j-1}) + \zeta_{j-\frac{1}{2}\hspace{-0.05cm}, \hspace{0.05cm} \xi}^n(t)(\xi_j - \xi_{j-1}).
\end{align*}
At last, one computes $y_{\Dx}^{n}(t, \cdot)$ from $y_{\Dx}^{n}(t, \xi_j) = \zeta_{\Dx}^{n}(t, \xi_j) + \xi_j$ for any $j \in \mathbb{Z}$.  

The iteration scheme is stopped at the first integer $n=M_{\mathrm{it}}^k$ satisfying
\begin{align}\label{eq:stopCond}
	\sup_{t \in [\tau_k^*, \tau_{k+1}^*]} \! \|y_{\Dx}^{M_{\mathrm{it}}^k}(t) - y_{\Dx}^{M_{\mathrm{it}}^k-1}(t)\|_{\infty} & \!\leq \! \sup_{t \in [\tau_k^*, \tau_{k+1}^*]} \! \| \{\zeta_{\Dx}^{M_{\mathrm{it}}^k}(t, \xi_j)\}_{j \in \mathbb{Z}} - \{\zeta_{\Dx}^{M_{\mathrm{it}}^k-1}(t, \xi_j)\}_{j \in \mathbb{Z}}\|_{\ell^{\infty}} \nonumber
	\\ & \leq \epsilon, 
\end{align}
for some user-determined threshold $\epsilon > 0$, where $\|\cdot \|_{\ell^{\infty}}$ denotes the usual sup-norm on $\ell^{\infty}$. As mentioned, the iteration scheme is contractive. Thus \eqref{eq:stopCond} will eventually be met for any $\epsilon > 0$, see \cite[Prop. 3.5]{Alpha2}. However, we will stick with the same choice as in \cite{Alpha2}, that is, 
\begin{equation*}
	\epsilon = \frac{1}{\|\alpha'\|_{\infty}}\Dx^2, 
\end{equation*}
in which case the number of iterates $M_{\mathrm{it}}^k$ in \eqref{eq:stopCond}, over any $(\tau_k^*, \tau_{k+1}^*]$, is uniformly bounded by $3$. Furthermore, set 
\begin{equation*}
	X_{\Dx}(t) = X_{\Dx}^{M_{\mathrm{it}}^k}(t) \hspace{0.5cm} \text{for all } t \in (\tau_k^*, \tau_{k+1}^*],
\end{equation*}
and let $S_{\Dx, t}: \F^{\alpha} \rightarrow \F^{\alpha}$ denote the numerical solution operator associated with the outlined evolution procedure, namely   
\begin{equation}\label{eq:numSol}
	S_{\Dx, t}(\bar{X}_{\Dx}) := X_{\Dx}(t), \hspace{0.5cm} \text{for any} t \in [0, T]. 
\end{equation}
The important point here is that $S_t$ from Definition~\ref{def:LagrOp} is not implemented exactly, but approximated by $S_{\Dx, t}$. 
As a consequence, we introduce two numerical errors on each time interval $(\tau_k^*, \tau_{k+1}^*]$, which we analyze in Section~\ref{sec:FirstAttempt}.

\subsubsection{Recovering the solution in Eulerian coordinates}
It remains to recover the numerical solution in Eulerian coordinates. One achieves this by applying $M$ from Definition~\ref{def:MapM} to $X_{\Dx}(t)$, namely 
\begin{equation}\label{eq:recoveringNum}
	(u_{\Dx}, \mu_{\Dx}, \nu_{\Dx})(t) = M(X_{\Dx}(t)). 
\end{equation}
As the components of $X_{\Dx}(t)$ are piecewise linear, so are the components of $(u_{\Dx}, F_{\Dx}, G_{\Dx})(t)$. In fact, their nodes are situated at the points $\{y_{\Dx}(t, \xi_j)\}_{j \in \mathbb{Z}}$ such that  \eqref{eq:recoveringNum} amounts to applying a piecewise linear reconstruction based on the values attained at these nodes. Furthermore, $u_{\Dx}(t)$ is continuous for all $t \geq 0$, while $F_{\Dx}(t)$ and $G_{\Dx}(t)$ are left-continuous and increasing. In addition, $F_{\Dx}(t) \neq G_{\Dx}(t)$ for $t>0$ in general. Again, see \cite[Sec. 3.2.4]{Alpha1} for more details. 

The outlined construction leads to the following definition of numerical $\alpha$-dissipative solutions. 

\begin{definition}\label{def:NumericalSol}
	Given $(\bar{u}, \bar{\mu}, \bar{\nu}) \! \in \D_0^{\alpha}$, $\Dx\! > \!0$, and $T>0$, define $T_{\Dx, t} = M \circ S_{\Dx, t} \circ L$. The numerical $\alpha$-dissipative solution, for any time $t \in [0, T]$, is then given by 
	\begin{equation*}
		(u_{\Dx}, \mu_{\Dx}, \nu_{\Dx})(t) = T_{\Dx, t} \circ P_{\Dx}((\bar{u}, \bar{\mu}, \bar{\nu})).
	\end{equation*}
\end{definition}

\section{Convergence rate for the numerical method}\label{sec:rateMethod}
Let  $T \geq 0$ be fixed and assume from now on that $\Dx \leq 1$. We aim at deriving a convergence rate in $C([0, T], L^{\infty}(\R))$ for the approximation family $\{u_{\Dx} \}_{\Dx > 0}$. By \cite[Lem. 4.11]{Alpha1}, it is enough to deduce an error estimate for $(y_{\Dx}-\id, U_{\Dx})(t)$ in $[L^{\infty}(\R)]^2$ for any $t \in [0, T]$,  because
\begin{equation}\label{eq:est_u}
	\|u(t) - u_{\Dx}(t)\|_{\infty} \leq \|U(t) - U_{\Dx}(t)\|_{\infty} + \bar{F}_{\infty}^{\nicefrac{1}{2}} \|y(t) - y_{\Dx}(t)\|_{\infty}^{\nicefrac{1}{2}}. 
\end{equation}

If we for a moment neglect the fact that we have to approximate $S_t$ by $S_{\Dx, t}$, cf. \eqref{eq:numSol}, and  integrate \eqref{eq:ODE2} with respect to time, we observe that in order to obtain a convergence rate for $\|U(t) - U_{\Dx}(t)\|_{\infty}$, we need a convergence rate for
\begin{align}\label{eq:trouble}
	\bigg|\frac{1}{2} \int_0^t \!\Big(\big(V(s, \xi) &- \frac{1}{2}V_{\infty}(s)\big) - \big(V_{\Dx}(s, \xi) - \frac{1}{2}V_{\Dx, \infty}(s) \big) \Big)ds \bigg| \nonumber \\
	& \hspace{-1.2cm}= \frac{1}{4}\Big |\!\int_0^t \! \Big(\int_{-\infty}^{\xi}\!(V_{\xi}- V_{\Dx, \xi})(s, \eta)d\eta - \int_{\xi}^{\infty}\!(V_{\xi}- V_{\Dx, \xi})(s, \eta)d\eta \Big)ds \Big |.
\end{align}
This is a delicate issue, because there might be a set $\mathcal{A}$ of positive measure in Lagrangian coordinates on which we have no initial rate for $(y_{\Dx, \xi}, U_{\Dx, \xi}, V_{\Dx, \xi}, H_{\Dx, \xi})$, which makes it difficult to obtain a convergence rate for \eqref{eq:trouble}. To resolve this issue, we introduce, in Section~\ref{sec:AdditionalChange}, a novel pair of maps, inspired by \cite[Def. 6]{NVWE}, that enables us to prove that the contribution from the set $\mathcal{A}$ to \eqref{eq:trouble} tends to $0$ with a certain rate when $\Dx \rightarrow 0$, provided $\bar{u}_{x}$ has more regularity than just belonging to $L^2(\R)$.  

\subsection{Initial convergence rates in Eulerian coordinates}\label{sec:InitialRatesEuler}
In order to have sufficiently strong initial convergence in Eulerian coordinates, one has to impose additional regularity on $\bar{u}_x$. In particular, we will use a regularity criterion which is inspired by a characterizing property of certain Besov spaces. 

 Let us start by recalling the following error estimates.  

\begin{prop}[{\cite[Prop. 4.1]{Alpha1}}] \label{prop:ratesPdx} Given $(\bar u, \bar \mu, \bar \nu) \in \D_0^{\alpha}$, let $(\bar u_{\Dx}, \bar \mu_{\Dx}, \bar \nu_{\Dx}) = P_{\Dx}((\bar u, \bar \mu, \bar \nu))$, then\begin{subequations}
	\begin{align}
	\|\bar u - \bar u_{\Dx}\|_{\infty} &\leq \big(1+\sqrt{2}\big) \bar F_{\mathrm{ac}, \infty}^{\nicefrac{1}{2}} \Dx^{\nicefrac{1}{2}}, \label{eq:u1}\\
	\|\bar u - \bar u_{\Dx}\|_2 &\leq \sqrt{2}\big(1+ \sqrt{2}\big) \bar F_{\mathrm{ac}, \infty}^{\nicefrac{1}{2}} \Dx, \\ \label{est:Fp}
	\|\bar F - \bar F_{\Dx}\|_p &\leq 2\bar F_{\infty}\Dx^{\nicefrac{1}{p}} \hspace{0.4cm} \text{for } p = 1, 2, \\
	 \|\bar G - \bar G_{\Dx}\|_p &\leq 2\bar G_{\infty}\Dx^{\nicefrac{1}{p}} \hspace{0.4cm} \text{for } p =1, 2. 
	\end{align}
	\end{subequations}
\end{prop}

Additionally, we also have $\bar u_{\Dx, x} \rightarrow \bar u_x$ in $L^2(\R)$ by \cite[Lem. 4.2]{Alpha1}. However,  as there is a close relation between $\bar u_{\Dx, x}$ and  the energy variable $\bar V_{\Dx, \xi}$, this convergence is too weak to guarantee a convergence rate for \eqref{eq:trouble} and hence for $(y_{\Dx}-\id, U_{\Dx})(t)$ in $[L^{\infty}(\R)]^2$. In fact,  it turns out that it is essential to have a bound of the form 
\begin{equation*}
\|\bar u_x - \bar u_{\Dx, x}\|_2 \leq \mathcal{O}(\Dx^{\gamma})\quad \text { for some } \quad \gamma > 0,  
\end{equation*}
because the solution operator $S_{\Dx, t}$ will otherwise completely deteriorate the initial convergence rate in \eqref{eq:u1} for $\alpha$-dissipative solutions due to the following reason. The characteristics along which wave breaking occurs at time $t$ are determined by the sets
\begin{equation}\label{eq:setsux}
	E_t := \Big\{\xi \!\mid \bar{u}_{x}(\bar y(\xi))) = -\frac{2}{t} \Big\} \hspace{0.35cm} \text{and} \hspace{0.35cm} E_{\Dx, t} :=  \Big\{\xi \!\mid \bar{u}_{\Dx, x}(\bar y_{\Dx}(\xi)) = -\frac{2}{t} \Big\}, 
\end{equation}
for $(\bar u,\bar F,\bar G)$ and $(\bar u_{\Dx}, \bar F_{\Dx}, \bar G_{\Dx})$, respectively.  Thus, to prevent the error between the exact and the numerical Lagrangian solution from becoming too large, one needs to be able to monitor the influence of $E_t$ and $E_{\Dx, t}$ for $t \in (0, T]$. This is possible if we have the stronger $L^2$-error bound on $\bar u_{\Dx, x}$. 

Note that $\bar u_{\Dx}$, given by \eqref{eq:up}, is piecewise linear such that we are approximating $\bar u_x \in L^2(\R)$ by a piecewise constant function and additional regularity has to be imposed in order to obtain a convergence rate in $L^2(\R)$. Motivated by the {\em Lipschitz spaces} $\mathrm{Lip}(\beta, L^2(\R))$, which coincide with the {\em Besov spaces} $ B_{\infty}^{\beta}(L^2(\R))$ for $\beta \in (0, 1)$, see e.g., \cite{NonlinearApprox, ConstructiveApprox},  we introduce, for $\beta \in (0, 1]$,  
\begin{equation}\label{eq:lipSpace}
	B_2^{\beta} := \big \{f \! \in L^2(\R)\! \mid \text{there exists } C\! > 0 \text{ s.t. } \|\delta_hf\|_2 \leq Ch^{\beta} \text{ for all } h \in (0, 2]\! \big \}. 
\end{equation}
Here $\delta_hf(x) := f(x+h) - f(x)$. Furthermore, denote by $|f|_{2, \beta}$ the smallest admissible constant in \eqref{eq:lipSpace}, i.e.,  
\begin{equation}\label{eq:semiNorm}
	|f|_{2, \beta} := \sup_{h \in (0, 2]}\! \!\Big( h^{-\beta} \|\delta_hf \|_2  \Big). 
\end{equation} 
Then the set $B_2^\beta$, which satisfies $B_{\infty}^\beta(L^2(\R))\subset B_2^{\beta}$, contains all the $L^2(\R)$-functions that optimally can be approximated  by piecewise constant functions on uniform meshes with an error of order $\mathcal{O}(\Dx^{\beta})$, see \cite[Sec. 3]{NonlinearApprox} and \cite[Chp. 12]{ConstructiveApprox}. However, the projection operator $P_{\Dx}$ is not constructed with the aim of approximating $\bar u_x$ by $\bar u_{\Dx, x}$ in the best possible way in $L^2(\mathbb{R})$, but instead with focus on preserving the relation $d\bar \mu_{\mathrm{ac}}= \bar u_x^2dx$. One therefore introduces $q_{2j}$, given by \eqref{eq:q2j}, which can be viewed as a correction term in \eqref{eq:up}, but it leads to a suboptimal convergence rate.

\begin{lemma}[{\cite[Lem. 3.2]{AlphaRate}}]\label{lem:rateux}
	Given $(\bar u, \bar \mu, \bar \nu) \in \D_0^{\alpha}$ with $\bar u_x \in B_2^{\beta}$ for some $\beta \in (0, 1]$,  let  $(\bar u_{\Dx}, \bar \mu_{\Dx}, \bar \nu_{\Dx}) = P_{\Dx}((\bar u, \bar \mu, \bar \nu))$. Then there exists a constant $\tilde{C}$, dependent on $\beta$, $\bar F_{\mathrm{ac}, \infty}$, and $|\bar u_x|_{2, \beta}$, such that
    	\begin{equation}\label{est:rateux}
        \|\bar u_x - \bar u_{\Delta x, x} \|_2 \leq \tilde{C}\Delta x^{\nicefrac{\beta}{2}}.
            \end{equation}
\end{lemma}

\begin{remark}\label{remark:beta}
Other more frequently encountered conditions are $\bar u_x \in \mathrm{BV}\cap L^1(\R)$ or $\bar u_x \in C_c^{0, \beta}(\R)$, i.e, $\bar u_x$ being $\beta$-H{\"o}lder continuous with compact support, which both lead to error estimates similar to \eqref{est:rateux}, because
\begin{equation*}
	BV\cap L^1(\R) \subset B_2^{\nicefrac{1}{2}} \hspace{0.55cm} \text{and} \hspace{0.55cm} 
	C_c^{0, \beta} \subset B_2^{\beta} \text{ for any } \beta \in (0, 1]. 
\end{equation*}
On the other hand, we only require $\bar u_x\in L^2(\R)$ and hence these two other conditions exclude a lot of interesting examples. One being the solution to the Cauchy problem \eqref{eq:HS} with the following initial data
 \begin{align*}
	\bar{u}(x) &= \begin{cases}
	1, &x < -1, \\
	\left|x\right|^{\nicefrac{2}{3}}\!, & -1\leq x \leq1, \\
	1, & 1 < x,
	\end{cases}
\end{align*}
for which wave breaking occurs for each $t\in [0,3]$ along a single characteristic, see \cite[Ex. 5.2]{Alpha1} and \cite[Ex. 3]{NumericalConservative}. 
Its derivative, given by 
\begin{align*}
	\bar{u}_{x}(x) &= \begin{cases}
	0, & x < -1, \\
	\frac{2}{3} \mathrm{sgn}(x)\left |x \right|^{-\nicefrac{1}{3}}\!, & -1 \leq x \leq 1, \\
	0, & 1 < x, \end{cases}
\end{align*} 
belongs neither to $C_c^{0, \beta}(\R)$ for any $\beta \in (0, 1]$ nor to $\mathrm{BV}(\R)$, but it does belong to $B_2^{\nicefrac{1}{6}}$. We will return to this example in Section~\ref{sec:NumericalExp}. 
\end{remark}

\subsection{A first step towards a convergence rate in Lagrangian coordinates}\label{sec:FirstAttempt}
The error estimates from Proposition~\ref{prop:ratesPdx} yield the following convergence rates in Lagrangian coordinates. 

\begin{lemma}[{\cite[Lem. 4.4]{Alpha1}}]\label{lem:initialRatesLagr}
	Suppose $(\bar u, \bar \mu, \bar \nu) \in \D_0^{\alpha}$, let $(\bar y, \bar U, \bar V, \bar H) = L((\bar u, \bar \mu, \bar \nu))$, and $(\bar y_{\Dx}, \bar U_{\Dx}, \bar V_{\Dx}, \bar H_{\Dx})= L \circ P_{\Dx}((\bar u, \bar \mu, \bar \nu))$, then 
	\begin{align*}
		\|\bar y - \bar y_{\Dx} \|_{\infty} & \leq 2\Dx, \\
		\|\bar U - \bar U_{\Dx} \|_{\infty} & \leq \big(1 + 2\sqrt{2} \big) \bar F_{\mathrm{ac}, \infty}^{\nicefrac{1}{2}} \Dx^{\nicefrac{1}{2}}, \\
		\|\bar H -\bar  H_{\Dx} \|_{\infty} &= \|\bar V - \bar V_{\Dx}\|_{\infty} \leq 2\Dx. 
	\end{align*}
\end{lemma}
Here we took advantage of the fact that $L(\D_0^{\alpha}) = \F_0^{\alpha, 0}$, with $\F_0^{\alpha, 0}$ defined in \eqref{eq:subsetsF}. 

Ideally, we would like to proceed by deriving convergence rates for the differentiated Lagrangian variables $\bar X_{\Dx, \xi} = (\bar y_{\Dx, \xi}, \bar U_{\Dx, \xi}, \bar V_{\Dx, \xi}, \bar H_{\Dx, \xi})$, but this is in general out of reach. Instead, we have to apply an additional change of variables in order to be able to handle \eqref{eq:trouble}. 

To better understand what type of transformation we are looking for, we start by outlining the approach we will take for further estimating \eqref{eq:est_u}. As mentioned, the solution operator $S_t$ from Definition~\ref{def:LagrOp} has to be approximated by the numerical operator $S_{\Dx, t}$,  which is based on evolving the numerical solution between successive times in $\{\tau_k^*\}_{k=0}^N$ by the iteration scheme described in Section~\ref{Sec:OutlineSt}. Consequently, one introduces two errors over each interval $[\tau_k^*, \tau_{k+1}^*]$ that we collectively refer to as the \emph{local error}: 
\begin{enumerate}[label=(\roman*)]
	\item the iteration scheme is terminated after a finite number of iterations, once \eqref{eq:stopCond} is satisfied with $\epsilon = \frac{1}{\|\alpha'\|_{\infty}}\Dx^2$, 
	\item the amount of energy to be removed at wave breaking is only approximated, cf.  \eqref{eq:beta1}--\eqref{eq:beta3}. 
\end{enumerate}

Therefore, introduce 
\begin{equation*}
	\widehat{X}_{\Dx}(t) = (\widehat{y}_{\Dx}, \widehat{U}_{\Dx}, \widehat{V}_{\Dx}, \widehat{H}_{\Dx})(t) = S_t(\bar{X}_{\Dx}).
\end{equation*}
It then follows from \eqref{eq:est_u} that we can write
\begin{align*}
	\|u(t) - u_{\Dx}(t)\|_{\infty} & \leq \|U(t) - \widehat{U}_{\Dx}(t)\|_{\infty} + \|\widehat{U}_{\Dx}(t) - U_{\Dx}(t)\|_{\infty}
	\\ & \qquad + \bar{G}_{\infty}^{\nicefrac{1}{2}}\|y(t) - \widehat{y}_{\Dx}(t)\|_{\infty}^{\nicefrac{1}{2}} + \bar{G}_{\infty}^{\nicefrac{1}{2}}\|\widehat{y}_{\Dx}(t) - y_{\Dx}(t)\|_{\infty}^{\nicefrac{1}{2}}, 
\end{align*}
where $\bar{F}_{\infty}=\bar{G}_{\infty}$, since $(\bar{u}, \bar{\mu}, \bar{\nu}) \in \D_0^{\alpha}$. A closer look reveals that the terms  $\|U(t) - \widehat{U}_{\Dx}(t)\|_{\infty}$ and $\|y(t) - \widehat{y}_{\Dx}(t)\|_{\infty}$ describe the {\em projection error}, i.e., the error introduced by the projection operator $P_{\Dx}$ and its influence on the time evolution in Lagrangian coordinates, while the other terms account for the accumulated effect of the aforementioned local errors. 

The local error has already been studied in detail in \cite{Alpha2} and a careful inspection of the estimates in \cite[(4.29) and Lem. 4.10]{Alpha2} reveals that there exists a constant $C$, dependent on $\|\alpha'\|_{\infty}$, $\|\bar u\|_\infty$,  $\bar G_{\infty}$, and  $T$, such that 
\begin{align}\label{eq:localErrLagr}
	 \|\widehat{y}_{\Dx}(t) - y_{\Dx}(t)\|_{\infty} &+ \|\widehat{U}_{\Dx}(t) - U_{\Dx}(t)\|_{\infty}  \leq C\Dx^{\nicefrac{1}{4}}, 
\end{align}
holds for all $t \in [0, T]$. Thus 
\begin{align}\label{eq:est_uNoLocal}
	\|u(t) - u_{\Dx}(t)\|_{\infty} &\leq \|U(t) - \widehat{U}_{\Dx}(t)\|_{\infty}  + \bar G_{\infty}^{\nicefrac{1}{2}}\|y(t) - \widehat{y}_{\Dx}(t)\|_{\infty}^{\nicefrac{1}{2}} \nonumber 
	\\ & \qquad + C\Dx^{\nicefrac{1}{4}} + \sqrt{C\bar G_{\infty}} \Dx^{\nicefrac{1}{8}},
\end{align}
for all $t\in [0,T]$. Or, in other words, the local error restricts the rate of convergence to be at best $\mathcal{O}(\Dx^{\nicefrac{1}{8}})$. However, it still remains to establish an error estimate for the projection error, i.e., for 
\begin{equation}\label{convrate:rest}
\|U(t) - \widehat{U}_{\Dx}(t)\|_{\infty} \quad \text{ and } \quad \|y(t) - \widehat{y}_{\Dx}(t)\|_{\infty}.
\end{equation} 

\vspace{0.1cm}
Integrating \eqref{eq:ODE1} with respect to time, immediately reveals that 
\begin{equation*}
	\|y(t) - \hat y_{\Dx}(t)\|_{\infty} \leq \|\bar{y} - \bar{y}_{\Dx}\|_{\infty} + \int_0^t \|U(s) - \hat U_{\Dx}(s)\|_{\infty}ds.
\end{equation*}
As Lemma~\ref{lem:initialRatesLagr} provides an upper bound for the first term on the right hand side and $t\leq T$, it suffices to derive a convergence rate for $\sup_{t\in [0,T]} \|U(t) - \hat U_{\Dx}(t)\|_{\infty}$.
Integrating \eqref{eq:ODE2} and using \eqref{eq:trouble} yields
\begin{align}
		\|U(t) - \hat U_{\Dx}(t)\|_{\infty} & \leq \|\bar{U} - \bar{U}_{\Dx}\|_{\infty} \nonumber
		\\ & \hspace{-0.15cm}+ \frac{1}{4} \bigg|\!\int_0^t\! \! \Big ( \int_{-\infty}^{\xi} \! \!\big(V_{\xi} - \hat V_{\Dx, \xi} \big)(s,\eta)d\eta -  \int_{\xi}^{\infty} \! \!\big(V_{\xi} - \hat V_{\Dx, \xi} \big)(s, \eta)d\eta \Big)ds \bigg|. \label{tow:conv1}
	\end{align}
Again, Lemma~\ref{lem:initialRatesLagr} provides us with an estimate for the first term on the right hand side and it is therefore enough to obtain a convergence rate for the last term on the right hand side, which coincides with \eqref{eq:trouble}. It will become apparent, in the next two sections, that the convergence order of \eqref{eq:trouble}, and hence also \eqref{eq:est_uNoLocal} and \eqref{convrate:rest}, heavily depends on the inequality from Lemma~\ref{lem:rateux}, and hence on the set $B_2^\beta$ that $\bar u_x$ belongs to. 

\subsection{The additional change of variables}\label{sec:AdditionalChange}
To establish a convergence rate for \eqref{eq:trouble}, we will apply an additional change of variables, which we study throughly in this section.

Let us start by observing that wave breaking occurs, cf. \eqref{eq:tau}, for the initial Lagrangian coordinate $\bar X$ on the following set
\begin{equation}\label{eq:S} 
	\mathcal{S} := \{\xi \in \R \! \mid \bar y_{\xi}(\xi) = 0 \}\! \! = \{\xi \in \R \! \mid \bar V_{\xi}(\xi) = 1 \}, 
\end{equation}
while it occurs for $\bar{X}_{\Dx}$, if we also recall \eqref{eq:numBreaking}, on the set
\begin{equation}\label{eq:Sdx}
	\mathcal{S}_{\Dx} := \{\xi \in \R \!\mid  \bar y_{\Dx, \xi}(\xi) = 0 \}\! \! = \{\xi \in \R \! \mid \bar V_{\Dx, \xi}(\xi) = 1 \} = \bigcup_{j \in \mathbb{Z}}[\xi_{3j}, \xi_{3j+1}]. 
\end{equation}
Furthermore, let $B=\bar y(\mathcal{S})$, then the proof of \cite[Thm. 27]{AlphaCH} reveals that 
\begin{equation}\label{eq:nuac}
	\bar \nu_{\mathrm{ac}} =\bar  \nu\lvert_{B^c} \hspace{0.5cm} \text{and} \hspace{0.5cm} \bar \nu_{\mathrm{sing}} = \bar \nu \lvert_{B}, 
\end{equation}
where $\bar \nu\lvert_B$ denotes the restriction of $\bar \nu$ to the set $B$, i.e., for any Borel set $A$, $\bar \nu\lvert_B (A) = \bar \nu(A \cap B)$. Hence,  
\begin{equation*}
	\mathrm{meas}(\mathcal{S}) = \bar \nu_{\mathrm{sing}}(\R).
\end{equation*}
Similarly, let $B_{\Dx}=\bar y_{\Dx}(\mathcal{S}_{\Dx})$, then 
\begin{equation}\label{eq:nudac}
	\bar \nu_{\Dx,\mathrm{ac}} =\bar  \nu_{\Dx}\lvert_{B_{\Dx}^c} \hspace{0.5cm} \text{and} \hspace{0.5cm} \bar \nu_{\Dx,\mathrm{sing}} = \bar \nu_{\Dx} \lvert_{B_{\Dx}}, 
\end{equation}
and, by Definition~\ref{def:ProjOperator}, we have 
\begin{equation*}
\mathrm{meas}(\mathcal{S}_{\Dx}) = \bar \nu_{\Dx, \mathrm{sing}}(\R)= \bar \nu_{\mathrm{sing}}(\R)=\mathrm{meas}(\mathcal{S}).
\end{equation*}
 
If 
\begin{equation}\label{SSDx}
\mathcal{S}^c= \mathcal{S}_{\Dx}^c,
\end{equation}
then the estimates for $\bar X_\xi-\bar X_{\Dx,\xi}$ on $\mathcal{S}^c$ can be played back to \eqref{est:rateux}, because on $\mathcal{S}^c$, the differentiated Lagrangian variables $\bar{X}_{\xi}$ can all be related to $\bar{u}_x$ and similarly for $\bar{X}_{\Dx, \xi}$ and $\bar{u}_{\Dx, x}$, see e.g., the proof of \cite[Prop. 2.1.5]{PhdThesisNordli}. Moreover, if \eqref{SSDx} holds, then $\bar{X}_{\xi} = (0, 0, 1, 1) = \bar{X}_{\Dx, \xi}$ a.e. on $\mathcal{S}$ and hence there is no contribution from $\bar X_\xi-\bar X_{\Dx, \xi}$ on $\mathcal{S}$. However,  \eqref{SSDx} does not hold in general, but there exists a rescaling of $\bar y$ and $\bar y_{\Dx}$ such that 
\begin{equation} \label{eq:mapCoincide}
	\bar y(\phi(r))=\bar y_{\Dx}(\psi(r)) \quad \text{ for all } r\in \mathbb{R}. 
\end{equation}
In analogy to the mappings $\mathcal{X}$ and $\mathcal{Y}$, defined in \cite[Def. 6]{NVWE} for the nonlinear variational wave equation, we introduce $\phi$ and $\psi$ as follows.

\begin{definition}\label{def:maps}
	Given $(\bar{u}, \bar{\mu}, \bar{\nu}) \in \D_0^{\alpha}$, let $(\bar{y}, \bar{U}, \bar{V}, \bar{H}) = L( (\bar{u}, \bar{\mu}, \bar{\nu}))$ and $(\bar{y}_{\Dx}, \bar{U}_{\Dx}, $  $\bar{V}_{\Dx}, \bar{H}_{\Dx}) = L \circ P_{\Dx}( (\bar{u}, \bar{\mu}, \bar{\nu}))$. Then the functions $\phi: \R\to \R$ and $\psi:\R\to \R$ are given by  
\begin{subequations}\label{eq:maps}
\begin{align}
	\phi(r)\! &= \sup \{\xi\! \mid \bar y(\xi') < \bar y_{\Dx}(2r - \xi')  \text{ for all } \xi'<\xi\}, \label{eq:map_chi}\\
	\psi(r)\! &= 2r - \phi(r). \label{eq:map_psi}
\end{align}
\end{subequations}
\end{definition}

Definition~\ref{def:maps} has a geometrical interpretation: as $\bar y$ and $\bar y_{\Dx}$ are continuous and increasing, the value attained by $\phi(r)$ corresponds, for each $r \in \R$, to the smallest intersection point between $\bar y(\xi)$ and $\bar y_{\Dx}(2r-\xi)$. Moreover, since $\bar G_{\Dx}(x_{2j}) = \bar G(x_{2j})$ by Definition~\ref{def:ProjOperator}, Definition~\ref{def:MapL} implies 
\begin{equation}\label{eq:coincideEnd}
\bar y(\xi_{3j}) = x_{2j} = \bar y_{\Dx}(\xi_{3j})\quad \text{ for all }j \in \mathbb{Z}, 
\end{equation}
which together with Definition~\ref{def:maps} yields  
\begin{equation}\label{eq:coincideGrid}
	[\phi(\xi_{3j}), \phi(\xi_{3j+3})) = [\xi_{3j}, \xi_{3j+3}) = [\psi(\xi_{3j}), \psi(\xi_{3j+3})) \hspace{0.75cm} \text{for } j \in \mathbb{Z}. 
\end{equation}

It therefore suffices to study the behavior of $\phi$ and $\psi$ on the isolated Lagrangian gridcell, $[\xi_{3j}, \xi_{3j+3})$. As a starting point, we look at two particular cases. 

As in \eqref{eq:nuac}, let $B=\bar y(\mathcal{S})$ and introduce 
\begin{equation}\label{def:SSDxj}
\mathcal{S}_{j}:=\mathcal{S}\cap [\xi_{3j}, \xi_{3j+3})\quad \text{ and } \quad \mathcal{S}_{\Dx, j}:= \mathcal{S}_{\Dx}\cap[\xi_{3j}, \xi_{3j+3})=[\xi_{3j},\xi_{3j+1}), 
\end{equation}
which, due to \eqref{eq:Fsp} and \eqref{eq:nuac}--\eqref{eq:nudac}, satisfy
\begin{equation}\label{size:SSDx}
\mathrm{meas}(\mathcal{S}_{j})=\bar \nu_{\mathrm{sing}}\big([x_{2j}, x_{2j+2}) \big)=\bar  \nu_{\Dx, \mathrm{sing}} \big([x_{2j}, x_{2j+2}) \big)= \mathrm{meas}(\mathcal{S}_{\Dx, {j}}).
\end{equation}

{\it Case 1:} If $B \cap [x_{2j}, x_{2j+2}) = \emptyset$, then $\mathcal{S}_{j}= \emptyset=\mathcal{S}_{\Dx, j}$ by Definition~\ref{def:ProjOperator} and \eqref{size:SSDx}, which in turn implies that $\bar y$ and $\bar y_{\Dx}$ are strictly increasing on $[\xi_{3j}, \xi_{3j+3})$. Thus, as depicted in Figure~\ref{fig:mapChiStrictly},  $\bar y(\xi)$ and $\bar y_{\Dx}(2r-\xi)$ have a unique intersection point for each $r \in [\xi_{3j}, \xi_{3j+3})$, given by $\phi(r)$, and the functions $\phi(r)$ and $\psi(r)$ are strictly increasing on $[\xi_{3j}, \xi_{3j+3})$.   

{\it Case 2:} If $B\cap [x_{2j}, x_{2j+2}) \neq \emptyset$ and $\bar \nu_{\mathrm{sing}}$ is purely discrete on $[x_{2j}, x_{2j+2})$, that is, $\bar \nu_{\mathrm{sing}}([x_{2j}, x_{2j+2}))=\bar \nu_{\mathrm{d}}([x_{2j}, x_{2j+2}))$, then $\bar y_{\Dx}$ is constant on the interval $[\xi_{3j}, \xi_{3j+1}]$, cf. \eqref{eq:Sdx}. Thus, by Definition~\ref{def:maps}, 
\begin{equation*}
\phi(r)=\xi_{3j} \quad \text{ and } \quad \psi(r)=\xi_{3j}+ 2(r-\xi_{3j}) \quad \text{ for all } r\in \big[\xi_{3j}, \tfrac12(\xi_{3j}+\xi_{3j+1}) \big].
\end{equation*}
Moreover, there exists a set $A\subset[\frac12(\xi_{3j}+ \xi_{3j+1}), \xi_{3j+3})$, which can be expressed as the union of at most countably many disjoint closed intervals, all having positive length, such that $\mathrm{meas}(A)=\frac12(\xi_{3j+1}-\xi_{3j})$ and
\begin{equation*}
\dot \phi(r)=2 \quad \text{ and } \quad \dot \psi(r)=0 \quad \text{ for a.e. } r\in {A}.
\end{equation*}
This is illustrated in Figure~\ref{fig:mapChiConst}.

\begin{figure}
	\captionsetup{width=.95\linewidth}
	\includegraphics{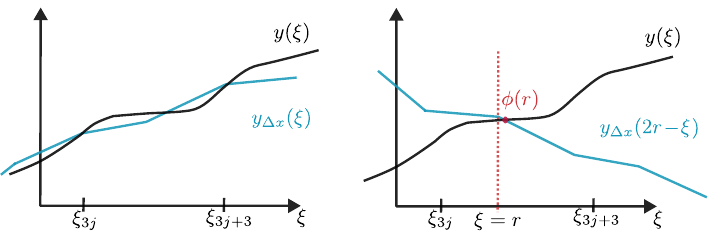}
	\caption{The left picture visualizes $\bar y(\xi)$ and $\bar y_{\Dx}(\xi)$. Since $\bar y$ and $\bar y_{\Dx}$ are strictly increasing on $[\xi_{3j}, \xi_{3j+3})$, $\bar y(\xi)$ and $\bar y_{\Dx}(2r-\xi)$ intersect in the unique point $\phi(r)$, marked by a red dot in the right picture. }
	\label{fig:mapChiStrictly} 
\end{figure}
\begin{figure}
	\captionsetup{width=.95\linewidth}
	\includegraphics{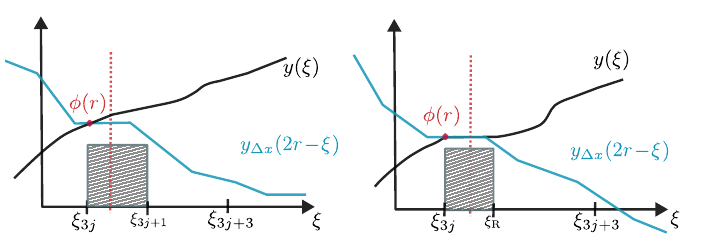}
	\caption{These plots illustrate what happens when $\bar y_{\Dx}(\xi)$ is constant on the interval $[\xi_{3j}, \xi_{3j+1}]$. Then either $\bar y(\xi)$ and $\bar y_{\Dx}(2r - \xi)$ intersect in a unique point, as in the left plot, or they coincide on a whole subinterval $[\xi_{3j}, \xi_{\mathrm{R}}] \subseteq [\xi_{3j}, \xi_{3j+1}]$ as in the right plot.}
	\label{fig:mapChiConst}
\end{figure} 

In general, neither $\phi$ nor $\psi$ will be relabeling functions in the sense of Definition~\ref{def:RelabelingGroup}, because we can only expect them to be increasing and not strictly increasing. Furthermore, it can be shown as in the proof of the well-posedness of \cite[Def. 6]{NVWE} that the pair $(\phi,\psi)$ satisfies the following properties. 

\begin{prop}\label{prop:propertiesmaps}
Let $\phi$ and $\psi$ be given by Definition~\ref{def:maps}, then 
\begin{enumerate}[label=(\roman*)]
	\item $\phi$ and $\psi$ are increasing, 
	\item $\phi-\id$, $\psi-\id \in W^{1, \infty}(\R)$,
	\item $0\leq \dot \phi\leq 2$ a.e. and $0\leq \dot \psi\leq 2$ a.e., 
	\item $\phi(r)+\psi(r)=2r$ for all $r\in \mathbb{R}$.
\end{enumerate} 
\end{prop}

To get a complete understanding of how  $\phi$ and $\psi$ behave on the isolated Lagrangian gridcell $[\xi_{3j}, \xi_{3j+3})$, and thereby on all of $\R$, we will need to take a slight detour via the function $\bar{\mathcal{Y}}(r)$ given by 
\begin{equation}\label{eq:Y}
\bar{ \mathcal{Y}}(r)= \bar y(\phi(r))=\bar y_{\Dx}(\psi(r)) \quad \text{ for all } r\in \R.
\end{equation}
In particular, we are interested in showing that $\bar{\mathcal{Y}}(r)$ is related to the measure $\bar \nu+\bar \nu_{\Dx}$ in much the same way as $\bar y$ is related to $\bar \nu$ in Definition~\ref{def:MapL}. As a consequence, $\bar{\mathcal{Y}}(r)$ will share a lot of properties with $\bar y$.

\begin{lemma}\label{lem:equivalentY}
 The function $\bar{\mathcal{Y}}(r)$ from \eqref{eq:Y} satisfies
\begin{equation}
		\bar{\mathcal{Y}}(r) = \sup \big\{x \in \R\!\mid x + \frac12 (\bar G + \bar G_{\Dx})(x) < r \big\}.  \label{eq:equivalentY} 
	\end{equation}
\end{lemma}

\begin{proof}
	Let
	\begin{equation*}
		\mathcal{Z}(r) := \sup \{x \in \R\! \mid 2x + (\bar G + \bar G_{\Dx})(x) < 2r \},
	\end{equation*}
	which, by definition, satisfies 
	\begin{equation}\label{prop:Z}
	2 \mathcal{Z}(r) + (\bar G + \bar G_{\Dx})(\mathcal{Z}(r))\leq 2r \leq 2\mathcal{Z}(r)  + (\bar G+\bar G_{\Dx})(\mathcal{Z}(r)+).
	\end{equation}
	Since the function $2x+(\bar G+\bar G_{\Dx})(x)$ is strictly increasing, the claim follows if we can show that $\bar{\mathcal{Y}}(r)$ also satisfies \eqref{prop:Z}.

	Definition~\ref{def:MapL} implies
	\begin{equation*}
		\bar y(\xi) + \bar G(\bar y(\xi)) \leq \bar y(\xi) + \bar H(\xi) \leq \bar y(\xi) + \bar G(\bar y(\xi)+),
	\end{equation*}
	which in turn leads to 
	\begin{equation}\label{eq:ineqG}
		\bar G(\bar y(\xi)) \leq \bar H(\xi) \leq \bar G(\bar y(\xi)+).
	\end{equation}
	In a similar way, we obtain
	\begin{equation}\label{eq:ineqGdx}
		\bar G_{\Dx}(\bar y_{\Dx}(\xi)) \leq \bar H_{\Dx}(\xi) \leq \bar G_{\Dx}(\bar y_{\Dx}(\xi)+). 
	\end{equation}
	Since both $\bar X$ and $\bar X_{\Dx}$ belong to $\F_0^{\alpha, 0}$, Propositon~\ref{prop:propertiesmaps} implies 
	\begin{align*}
		\bar y(\phi(r)) + \bar y_{\Dx}(\psi(r)) + \bar H(\phi(r)) + \bar H_{\Dx}(\psi(r)) = \phi(r) + \psi(r) = 2r, 
	\end{align*}
	which, after being combined with \eqref{eq:Y} and \eqref{eq:ineqG}--\eqref{eq:ineqGdx}, yields	
	\begin{align*}
		2 \bar{\mathcal{Y}}(r) + (\bar G + \bar G_{\Dx})(\bar {\mathcal{Y}}(r))\leq 2r \leq 2\bar{\mathcal{Y}}(r)  + (\bar G+\bar G_{\Dx})(\bar{\mathcal{Y}}(r)+). \hspace{0.65cm}\qedhere
	\end{align*}
\end{proof}

Next, let 
\begin{equation}\label{eq:setB}
\mathcal{B}:=\{r\in \R\mid \bar{\mathcal{Y}}_r(r)=0 \}
\end{equation}
and introduce $\mathcal{C}=\bar{\mathcal{Y}}(\mathcal{B})$, then, in a similar spirit to \eqref{eq:nuac} and \eqref{eq:nudac}, we have 
\begin{equation*}
\bar \nu_{\mathrm{ac}}+\bar \nu_{\Dx, \mathrm{ac}}=(\bar \nu+\bar \nu_{\Dx})_{\mathrm{ac}}= (\bar \nu+\bar\nu_{\Dx})\vert_{\mathcal{C}^c}= \bar \nu\vert_{\mathcal{C}^c}+ \bar \nu_{\Dx}\vert_{\mathcal{C}^c},
\end{equation*}
and
\begin{equation}\label{eq:sumSingular}
\bar \nu_{\mathrm{sing}}+\bar \nu_{\Dx, \mathrm{sing}}=(\bar \nu+\bar \nu_{\Dx})_{\mathrm{sing}}= (\bar \nu+\bar\nu_{\Dx})\vert_{\mathcal{C}}= \bar\nu\vert_{\mathcal{C}}+ \bar\nu_{\Dx}\vert_{\mathcal{C}}.
\end{equation}
As a consequence, one finds that 
\begin{equation*}
2\mathrm{meas}(\mathcal{B})= \bar \nu_{\mathrm{sing}}(\R)+ \bar \nu_{\Dx, \mathrm{sing}}(\R)=2\bar \nu_{\mathrm{sing}}(\R).
\end{equation*}

To finally get a better understanding of the pair $(\phi, \psi)$, we will apply the following result, which is a special case of \cite[Thm. 3.59]{LeoniBook} and ensures that the chain rule can be applied to compute the derivative of $\bar{\mathcal{Y}}(r)$, given by \eqref{eq:Y}. 

\begin{theorem}\label{thm:chainRule}
Suppose $f$, $g:\R\rightarrow\R$ are Lipschitz continuous, then 
\begin{equation*}
		(f \circ g)' (x)= f'(g(x))g'(x) \hspace{0.35cm} \text{for a.e. } x\in \R , 
	\end{equation*}
	where $f'(g(x))g'(x)$ is interpreted as zero whenever $g'(x) = 0$, even if $f$ is not differentiable at $g(x)$. 
\end{theorem}

A consequence of Lemma~\ref{lem:equivalentY} and Theorem~\ref{thm:chainRule}, is that we can give a complete description of the behavior of $\phi$ and $\psi$ on any Lagrangian gridcell $[\xi_{3j}, \xi_{3j+3})$. To this end, recall $\mathcal{S}_j$ and $\mathcal{S}_{\Dx,j}$ given by \eqref{def:SSDxj}.

{\it Case 1:} If $B \cap [x_{2j}, x_{2j+2}) = \emptyset$, then as before, $\mathcal{S}_{j}= \emptyset=\mathcal{S}_{\Dx, j}$ ensures that $\bar y$ and $\bar y_{\Dx}$ are strictly increasing on $[\xi_{3j}, \xi_{3j+3})$ and hence $\bar y(\xi)$ and $\bar y_{\Dx}(2r-\xi)$ have a unique intersection point for each $r \in [\xi_{3j}, \xi_{3j+3})$, given by $\phi(r)$. The functions $\phi(r)$ and $\psi(r)$ are therefore strictly increasing on $[\xi_{3j}, \xi_{3j+3})$. 

{\it Case 2:} If $B\cap [x_{2j}, x_{2j+2})\not=\emptyset$, then $\bar y_{\Dx}$ is constant on the interval $[\xi_{3j}, \xi_{3j+1}]$. Thus, by Definition~\ref{def:maps}
\begin{equation*}
\phi(r)=\xi_{3j} \quad \text{ and } \quad \psi(r)= \xi_{3j}+2(r-\xi_{3j}) \quad \text{ for all } r \in \mathcal{B}_{\Dx,j},
\end{equation*}
which implies 
\begin{equation}\label{PhiPB}
\dot \phi(r)=0 \quad \text{ and } \quad \dot \psi(r)=2 \quad \text{ for all } r \in \mathcal{B}_{\Dx,j},
\end{equation}
where 
\begin{equation}\label{def:BDxj}
	\mathcal{B}_{\Dx,j}= \big[\xi_{3j}, \tfrac12 (\xi_{3j}+\xi_{3j+1}) \big].
\end{equation}
Furthermore, one has 
\begin{equation}\label{psi:BS}
\psi(\mathcal{B}_{\Dx,j})=\mathcal{S}_{\Dx,j},
\end{equation}
and $\bar y_{\Dx, \xi}(\xi)> 0$ for a.e.  $\xi \in [\xi_{3j}, \xi_{3j+3})\backslash \mathcal{S}_{\Dx,j}$.

Let $\mathcal{B}_j= \big[\tfrac12(\xi_{3j}+\xi_{3j+1}), \xi_{3j+3} \big]\cap \mathcal B$. By the chain rule, it now follows that 
\begin{equation}\label{PsiPB}
\dot \phi(r)=2 \quad \text{ and } \quad \dot \psi(r)=0 \quad \text{ for all }r\in \mathcal{B}_j,
\end{equation}
and 
\begin{equation}\label{phi:BS}
\phi(\mathcal{B}_j)= \mathcal{S}_{j}.
\end{equation}
To summarize, 
\begin{equation*}
	\mathcal{B}_{\Dx,j} = \big \{r \in [\xi_{3j}, \xi_{3j+3}) \!\mid \dot{\phi}(r) = 0 \big\} \hspace{0.3cm} \text{and} \hspace{0.3cm}\mathcal{B}_j = \big\{r \in [\xi_{3j}, \xi_{3j+3}) \!\mid \dot{\psi}(r) = 0 \big\}. 
\end{equation*}
Furthermore,  by definition, 
\begin{equation*}
\mathcal{B}_j\cap \mathcal{B}_{\Dx,j}\subset \big\{\tfrac{1}{2}(\xi_{3j} + \xi_{3j+1})\big\}\quad \text{ and } \quad \mathcal{B}_j\cup \mathcal{B}_{\Dx,j}= \mathcal{B}\cap [\xi_{3j}, \xi_{3j+3}),
\end{equation*}
which implies, thanks to \eqref{eq:sumSingular}, \eqref{psi:BS}, and \eqref{phi:BS},
\begin{align}\nonumber
(\bar \nu+\bar \nu_{\Dx})_{\mathrm{sing}}\big([x_{2j}, x_{2j+2}) \big)& = 2\mathrm{meas} (\mathcal{B}\cap [\xi_{3j}, \xi_{3j+3}))\\\label{prop:set}
& =2( \mathrm{meas} (\mathcal{B}_j)+ \mathrm{meas}(\mathcal{B}_{\Dx,j})).
\end{align}

Finally, we want to show that the size of $\mathcal{B}_j$ and $\mathcal{B}_{\Dx,j}$ coincide for any $j\in \mathbb{Z}$.  As we will see later, this property yields crucial cancellations in \eqref{eq:trouble} and hence makes it possible to derive a convergence rate.   
\begin{lemma}\label{lem:coincidingLengths}
For any $j\in \mathbb{Z}$, 
\begin{equation*}
\mathrm{meas}(\mathcal{B}_j)=\mathrm{meas}(\mathcal{B}_{\Dx,j})=\frac12 \bar \nu_{\mathrm{sing}}([x_{2j}, x_{2j+2})).
\end{equation*}
\end{lemma}

\begin{proof}
If $\bar \nu_{\mathrm{sing}}([x_{2j}, x_{2j+2}))=\bar \nu_{\Dx,\mathrm{sing}}([x_{2j}, x_{2j+2}))=0$, the claim follows immediately from \eqref{prop:set}. 

For $\bar \nu_{\mathrm{sing}}([x_{2j}, x_{2j+2}))=\bar \nu_{\Dx,\mathrm{sing}}([x_{2j}, x_{2j+2}))\not =0$, we have that $\bar V_{\Dx, \xi}(\xi)=1$ for all $\xi \in \mathcal{S}_{\Dx, j}$ by \eqref{eq:Sdx} and \eqref{def:SSDxj}. Furthermore, $\dot \psi(r)=2$ for all $r \in \mathcal{B}_{\Dx,j}$, which combined with \eqref{eq:nudac} and \eqref{psi:BS} yields 
\begin{align*}
\bar \nu_{\mathrm{sing}}([x_{2j}, x_{2j+2}))& =\bar  \nu_{\mathrm{sing}, \Dx} ([x_{2j}, x_{2j+2})) =  \int_{\mathcal{S}_{\Dx,j}} 1 d\eta\\
& = \int_{\psi(\mathcal{B}_{\Dx,j})} 1 d\eta
 = \int_{\mathcal{B}_{\Dx,j}} 2 d\eta
= 2 \mathrm{meas} (\mathcal{B}_{\Dx,j}),
\end{align*}
and recalling \eqref{prop:set} therefore finishes the proof. 
\end{proof}

Next, recall the composition operator $\bullet$ from \eqref{eq:relabel} and introduce 
\begin{equation}\label{def:cX}
\bar{\mathcal{X}}(r)=  (\bar{\mathcal{Y}}(r), \bar{\mathcal{U}}(r), \bar{\mathcal{V}}(r), \bar{\mathcal{H}}(r)) =  \bar X \bullet \phi(r),
\end{equation}
and 
\begin{equation}\label{def:cXdx}
\bar {\mathcal{X}}_{\Dx}(r)= (\bar {\mathcal{Y}}_{\Dx}(r), \bar {\mathcal{U}}_{\Dx}(r), \bar {\mathcal{V}}_{\Dx}(r), \bar {\mathcal{H}}_{\Dx}(r)) = \bar X_{\Dx} \bullet \psi(r),
\end{equation}
which satisfy, cf. \eqref{eq:Y}
\begin{equation}\label{prop:YYD}
\bar{\mathcal{Y}}(r)= \bar{\mathcal{Y}}_{\Dx}(r) \quad \text{ for all } r\in \R. 
\end{equation}

Note that $\bar{\mathcal{X}}$ and $\bar {\mathcal{X}}_{\Dx}$, in general, do not belong to $\F^{\alpha}$, given by Definition~\ref{def:LagSet}, since $\phi$ and $\psi$ do not belong to $\mathcal{G}$, the set of relabeling functions given by Definition~\ref{def:RelabelingGroup}. To be more specific, there are two conditions which are not satisfied:
\begin{itemize}
\item Since $\phi-\id$ and $\psi-\id$ belong to $W^{1, \infty}(\mathbb{R})$, but not to $E_2$, the same holds for $\bar{\mathcal{Y}}-\id $ and $\bar {\mathcal{Y}}_{\Dx}- \id$. 
\item Since $\phi$ and $\psi$ are only increasing, but not strictly increasing, $\bar{\mathcal{Y}}_r+ \bar {\mathcal{H}}_r \geq 0$ and $\bar{\mathcal{Y}}_{\Dx, r} +\bar{\mathcal{H}}_{\Dx,r}\geq 0$, by Theorem~\ref{thm:chainRule}, hence violating Definition~\ref{def:LagSet}~\ref{def:LagSet2}. 
\end{itemize}

We claim that the tuples $\bar{\mathcal{X}}$ and $\bar{\mathcal{X}}_{\Dx}$ belong to the space $\mathfrak{F}^{\alpha}$, which is defined next. 

\begin{definition}
The space $\mathfrak{F}^{\alpha}$ is composed of all quadruplets $\mathcal{X}=(\mathcal{Y}, \mathcal{U}, \mathcal{V}, \mathcal{H})$ satisfying
\begin{enumerate}[label=(\roman*)]
\item $(\mathcal{Y}-\id, \mathcal{U}, \mathcal{V}, \mathcal{H})\in [W^{1, \infty}(\mathbb{R})]^{4}$,
\item $\mathcal{Y}_r, \mathcal{H}_r\geq 0$ a.e. and $\mathcal{Y}_r+\mathcal{H}_r\geq 0$ a.e.,
\item $\mathcal{Y}_r\mathcal{V}_r= \mathcal{U}_r^2$ a.e.,  
\item $0\leq \mathcal{V}_r\leq \mathcal{H}_r$ a.e.,
\item There exists $\kappa:\mathbb{R}\to[0,1)$ such that $\mathcal{V}_r(r)= \kappa(\mathcal{Y}(r))\mathcal{H}_r (r)$ for a.e. $r\in \mathbb{R}$ with $\kappa(\mathcal{Y}(r))=1$ whenever $\mathcal{U}_r(r)<0$.
\end{enumerate}
\end{definition}

\begin{lemma}
The tuplets $\bar{\mathcal{X}}$ and $\bar{\mathcal{X}}_{\Dx}$ belong to $\mathfrak{F}^{\alpha}$. 
\end{lemma}

The proof is an immediate consequence of Theorem~\ref{thm:chainRule} and the following theorem, that is a special case of \cite[Thm. 4]{ChainRule}, which enable us to perform a change of variables when integrating.

\begin{theorem}\label{thm:changeVar}
Let $f\in L^1(\R)$
and assume that $g: \R\rightarrow \R$ is Lipschitz continuous and increasing. Then 
\begin{enumerate}
		\item $(f\circ g)g' \in L^1(\R)$ and  
		\item the change of variables formula
		 \begin{equation*}
		\int_{g(\kappa)}^{g(\gamma)}f(x)dx = \int_{\kappa}^{\gamma} f(g(s))g'(s)ds, 
	\end{equation*}
	holds for any $\gamma, \kappa \in  \R$.
	\end{enumerate}
\end{theorem}

Denote by $g(\bar {\mathcal{X}})$ the function from \cite[Def. 4.3]{AlphaHS}, which describes the energy loss at wave breaking continuously in time, in contrast to $\bar {\mathcal{V}}_{r}$ which might drop abruptly at wave breaking. For $\mathcal{X}\in\mathfrak{F}^{\alpha}$, it is defined by 
 \begin{align}\label{eq:g}
        g(\mathcal{X})(r) \!&:= \begin{cases}
        \big(1 - \alpha(\mathcal{Y}(r)) \big)\mathcal{V}_{r}(r), & r \in \Omega_d(\mathcal{X}), \\
        \mathcal{V}_{r}(r), & r \in \Omega_c(\mathcal{X}),
     \end{cases}
    \end{align}
where 
\begin{equation}\label{eq:omegas}
    \Omega_d (\mathcal{X}) := \{ r \in \R \! \mid \mathcal{U}_{r}(r) < 0 \} \hspace{0.35cm} \text{and} \hspace{0.35cm}     \Omega_c(\mathcal{X}) := \{r \in \R \! \mid \mathcal{U}_{r}(r) \geq 0 \}.
\end{equation}
Note that $\Omega_d(\mathcal{X})$ contains all the points at which the $\alpha$-dissipative solution eventually will experience wave breaking, whereas $\Omega_c(\mathcal{X})$ contains the points where there will be no wave breaking in the future. 

\begin{lemma}\label{lem:boundHrUrgr} Suppose $(\bar{u}, \bar{\mu}, \bar{\nu}) \in \D^{\alpha}_0$. Let $\bar{X} = L((\bar{u}, \bar{\mu}, \bar{\nu}))$ and  \newline $\bar{X}_{\Dx} = L \circ P_{\Dx}((\bar u, \bar \mu, \bar \nu))$. Moreover, introduce $\bar{\mathcal{X}}= \bar X\bullet \phi$ and $\bar{\mathcal{X}}_{\Dx}= \bar X_{\Dx} \bullet \psi$ with $\phi$ and $\psi$ given by Definition~\ref{def:maps}, then the following estimates hold 
\begin{align}\label{help:1}
\|\bar{\mathcal{H}}_r - \bar{\mathcal{H}}_{\Dx,r}\|_{L^1(\mathcal{B}^c)}& \leq 2\bar F_{\mathrm{ac},\infty}^{\nicefrac{1}{2}}\|\bar u_x - \bar u_{\Dx, x}\|_2,\\ \label{help:2}
\|\bar {\mathcal{U}}_r-\bar{\mathcal{U}}_{\Dx,r}\|_{L^2(\mathcal{B}^c)}& \leq 2\|\bar u_x-\bar u_{\Dx,x}\|_2,
\\ \label{help:3}
\|g(\bar {\mathcal{X}}) - g(\bar{\mathcal{X}}_{\Dx}) \|_{L^1(\mathcal{B}^c)}& \leq 4 \bar F_{\mathrm{ac}, \infty}^{\nicefrac{1}{2}} \|\bar u_x - \bar u_{\Dx, x}\|_2.
\end{align}
\end{lemma}

\begin{proof}
A closer look at the proof of the well-posedness of $L$, see \cite[Prop. 2.1.5]{PhdThesisNordli}, reveals that 
\begin{align}\label{rep:Xxi}
\bar X_\xi(\xi) &= (\bar{y}_{\xi}(\xi), \bar{U}_{\xi}(\xi), \bar{V}_{\xi}(\xi), \bar{H}_{\xi}(\xi)) \nonumber 
\\ &= \begin{cases}(\bar y_\xi(\xi), \bar u_x(\bar y(\xi))\bar y_\xi(\xi), \bar u_x^2(\bar y(\xi))\bar y_\xi(\xi), \bar u_x^2(\bar y(\xi))\bar y_\xi(\xi)), &  \xi \in \mathcal{S}^c\!,\\
(0,0,1,1), & \xi \in \mathcal{S},
\end{cases}
\end{align}
and analogously, 
\begin{equation*}
\bar X_{\Dx,\xi}(\xi)= \begin{cases}(\bar y_{\Dx,\xi}(\xi), \bar u_{\Dx,x}(\bar y_{\Dx}(\xi))\bar y_{\Dx,\xi}(\xi), \\
 \quad \bar u_{\Dx,x}^2(\bar y_{\Dx}(\xi))\bar y_{\Dx,\xi}(\xi), \bar u_{\Dx,x}^2(\bar y_{\Dx}(\xi))\bar y_{\Dx,\xi}(\xi)), &   \xi \in \mathcal{S}_{\Dx}^c,\\
(0,0,1,1), &  \xi \in \mathcal{S}_{\Dx},
\end{cases}
\end{equation*}
where $\mathcal{S}$ and $\mathcal{S}_{\Dx}$ are given by \eqref{eq:S} and \eqref{eq:Sdx}, respectively. Thus, recalling \eqref{eq:Y}, \eqref{PhiPB}, and \eqref{psi:BS}--\eqref{phi:BS} and applying Theorem~\ref{thm:chainRule}, we have 
\begin{equation}\label{rep:Xr}
\bar{\mathcal{X}}_r(r)= \begin{cases}
(\bar{\mathcal{Y}}_r(r), \bar u_x(\bar{\mathcal{Y}}(r))\bar{\mathcal{Y}}_r(r), \bar u_x^2 (\bar{\mathcal{Y}}(r))\bar{\mathcal{Y}}_r(r), \bar u_x^2 (\bar{\mathcal{Y}}(r))\bar{\mathcal{Y}}_r(r)) , & r\in \mathcal{B}^c,\\
(0, 0, *,*), &  r\in \mathcal{B},
\end{cases}
\end{equation}
and by  \eqref{prop:YYD}, we also have
\begin{equation*}
\bar{\mathcal{X}}_{\Dx,r}(r)= \begin{cases}
(\bar{\mathcal{Y}}_{r}(r), \bar u_{\Dx, x}(\bar{\mathcal{Y}}(r))\bar{\mathcal{Y}}_{r}(r),\\
 \quad  \bar u_{\Dx, x}^2 (\bar{\mathcal{Y}}(r))\bar{\mathcal{Y}}_{r}(r), \bar u_{\Dx,x}^2 (\bar{\mathcal{Y}}(r))\bar{\mathcal{Y}}_{r}(r)) , &  r\in \mathcal{B}^c,\\
(0, 0, *,*), &  r\in \mathcal{B},\\
\end{cases}
\end{equation*}
where $*$ either equals $0$ or $2$, 
which implies
\begin{align*}
	\|\bar{\mathcal{H}}_r - \bar{\mathcal{H}}_{\Dx,r}\|_{L^1(\mathcal{B}^c)} 
	&=  \int_{\mathcal{B}^c} \! \!|\bar u_x^2(\bar{\mathcal{Y}}) - \bar u_{\Dx, x}^2(\bar{\mathcal{Y}})|\bar{\mathcal{Y}}_r (r)dr. 
\end{align*}
Applying the change of variables $x=\bar{\mathcal{Y}}(r)$, which is possible since all assumptions of Theorem~\ref{thm:changeVar} are satisfied, and thereafter the Cauchy--Schwarz inequality, one ends up with 
\begin{align*}
	  \int_{\mathcal{B}^c} \! \!|\bar u_x^2(\bar{\mathcal{Y}}) - \bar u_{\Dx, x}^2(\bar{\mathcal{Y}})|\bar{\mathcal{Y}}_r (r)dr &\leq \int_{\R} |\bar u_x^2(x) - \bar u_{\Dx, x}^2(x)|dx 
	 \\ & \leq \|\bar u_x + \bar u_{\Dx, x}\|_2 \|\bar u_x - \bar u_{\Dx, x}\|_2 
	 \\ & \leq 2\bar F_{\mathrm{ac},\infty}^{\nicefrac{1}{2}}\|\bar u_x - \bar u_{\Dx, x}\|_2. 
\end{align*}

Since the proof of the second inequality \eqref{help:2} is the same, up to some slight modifications, we do not include the details here.

Finally we turn our attention towards \eqref{help:3}.  Based on \eqref{eq:omegas}, introduce 
\begin{equation}\label{def:Omegamn}
\bar{\Omega}_{m, n}=\mathcal{B}^c\! \! \cap \Omega_m(\bar{\mathcal{X}})\cap \Omega_n(\bar{\mathcal{X}}_{\Dx}), \quad n, m\in \{c,d\},
\end{equation}
and recall that $g$ is given by \eqref{eq:g}. Then 
	\begin{align} \nonumber
		\|g(\bar{\mathcal{X}}) - g(\bar{\mathcal{X}}_{\Dx})\|_{L^1(\mathcal{B}^c)} &= \int_{\bar{\Omega}_{c, c}} \!\big|\bar{\mathcal{V}}_r - \bar{\mathcal{V}}_{\Dx, r} \big|(r)dr
		\\ & \qquad + \int_{\bar{\Omega}_{d, c}}\!\big| \big(1-\alpha(\bar{\mathcal{Y}})\big)\bar{\mathcal{V}}_r - \bar{\mathcal{V}}_{\Dx, r} \big|(r)dr \nonumber
		\\ & \qquad + \int_{\bar{\Omega}_{c, d}}\!\big| \bar{\mathcal{V}}_r - \big(1-\alpha(\bar{\mathcal{Y}}) \big)\bar{\mathcal{V}}_{\Dx, r} \big|(r)dr \nonumber
		\\ & \qquad + \int_{\bar{\Omega}_{d, d}}\!\big(1-\alpha(\bar{\mathcal{Y}})\big)\big|\bar{\mathcal{V}}_r - \bar{\mathcal{V}}_{\Dx, r} \big|(r)dr \nonumber\\
		& \leq   \|\bar{\mathcal{H}}_r - \bar{\mathcal{H}}_{\Dx,r}\|_{L^1(\mathcal{B}^c)} \label{eq:g1}
		\\ & \qquad + \int_{\bar{\Omega}_{d, c}}\alpha(\bar{\mathcal{Y}})\bar{\mathcal{H}}_{r}(r)dr \label{eq:g2}
		\\ & \qquad + \int_{\bar{\Omega}_{c, d}}\alpha(\bar{\mathcal{Y}})\bar{\mathcal{H}}_{\Dx, r}(r)dr, \label{eq:g3}
	\end{align}
	where we used \eqref{prop:YYD}. 
	
	As \eqref{help:1} provides an upper bound on the first term, it remains to examine \eqref{eq:g2} and \eqref{eq:g3}. Here we focus on \eqref{eq:g2}, since \eqref{eq:g3} has a similar structure.
	
Observing that for almost every $r\in \mathcal{B}^c$ the signs of $\bar{\mathcal{U}}_r(r)$ and $\bar{\mathcal{U}}_{\Dx,r}(r)$ coincide with the signs of $\bar u_x(\bar{\mathcal{Y}})(r)$ and $\bar u_{\Dx,x}(\bar{\mathcal{Y}})(r)$, respectively, we find, by Theorem~\ref{thm:changeVar} and \eqref{rep:Xr}, 
	 \begin{align*}
	 	 \int_{\bar{\Omega}_{d, c}}\alpha(\bar{\mathcal{Y}})\bar{\mathcal{H}}_r(r)dr &\leq \int_{\bar{\Omega}_{d, c}}\bar u_x^2(\bar{\mathcal{Y}(}r))\bar{\mathcal{Y}}_r(r)dr
		 \\ & \leq \int_{\bar{\Omega}_{d, c}}\bar u_x(\bar{\mathcal{Y}}(r))\big(\bar u_x(\bar{\mathcal{Y}}(r)) - \bar u_{\Dx, x}(\bar{\mathcal{Y}}(r))\big)\bar{\mathcal{Y}}_r(r)dr
		  \\ &\leq \bar F_{\mathrm{ac}, \infty}^{\nicefrac{1}{2}} \|\bar u_x - \bar u_{\Dx, x}\|_2. 
	\end{align*}
	Deriving a similar estimate for \eqref{eq:g3}, which is left to the interested reader, finishes the proof.
\end{proof}

\subsection{The convergence rate for the critical term \eqref{eq:trouble}}

Recall \eqref{eq:relabel}, that $\hat{X}_{\Dx}(t) = S_t(\bar{X}_{\Dx})$, and introduce 
\begin{align}\label{def:mX}
\mathcal{X}(t,r)&= (\mathcal{Y}(t,r), \mathcal{U}(t,r), \mathcal{V}(t,r), \mathcal{H}(t,r)) =  X(t) \bullet \phi(r) \end{align}
and 
\begin{align}\label{def:mXDx}
\hat {\mathcal{X}}_{\Dx}(t,r)&= (\hat {\mathcal{Y}}_{\Dx}(t,r), \hat {\mathcal{U}}_{\Dx}(t,r), \hat {\mathcal{V}}_{\Dx}(t,r), \hat {\mathcal{H}}_{\Dx}(t,r)) = \hat X_{\Dx}(t) \bullet \psi(r).
\end{align}
As in the last section, one can show that $\mathcal{X}(t, \cdot)$ and $\hat{\mathcal{X}}_{\Dx}(t, \cdot)$ belong to $\mathfrak{F}^{\alpha}$ for every $t\geq 0$ and by Theorem~\ref{thm:chainRule}, \eqref{PhiPB}, and \eqref{PsiPB}, one has in addition that for every $j\in\mathbb{Z}$ \vspace{-0.1cm}
\begin{equation}\label{cond:WB}
(\mathcal{Y}_r+ \mathcal{H}_r)(t,r)=0 \quad \text{and} \quad  \hat{\mathcal{V}}_{\Dx,r}(t,r)= \hat{\mathcal{H}}_{\Dx,r}(t,r)=2\quad \text{for } r\in \mathcal{B}_{\Dx,j},
\end{equation}
and \vspace{-0.1cm}
\begin{equation}\label{cond:WBD}
 (\hat{\mathcal{Y}}_{\Dx,r}+ \hat{\mathcal{H}}_{\Dx,r})(t,r)=0 \quad \text{and}\quad \mathcal{V}_{r}(t,r)= \mathcal{H}_{r}(t,r)=2\quad \text{for }r\in \mathcal{B}_j.
 \end{equation}

Moreover, the solution operator $S_t$ respects the mappings $\phi$ and $\psi$ in the following sense.

\begin{lemma} \label{lem:commutation}
Suppose $(\bar{u}, \bar{\mu}, \bar{\nu}) \in \D_0^{\alpha}$ and $t \in [0, T]$. Let $\bar{X} = L((\bar{u}, \bar{\mu}, \bar{\nu}))$ and $\bar{X}_{\Dx} = L \circ P_{\Dx}((\bar{u}, \bar{\mu}, \bar{\nu}))$. Then 
	\begin{equation}\label{eq:identitiesCom}
		S_t(\bar{X} \bullet \phi) = S_t(\bar{X}) \bullet \phi \hspace{0.35cm} \text{and} \hspace{0.35cm} 		S_t(\bar{X}_{\Dx} \bullet \psi) = S_t(\bar{X}_{\Dx}) \bullet \psi. 
	\end{equation}
\end{lemma}

The proof, up to some slight modifications, coincides with the one of \cite[Prop. 3,7]{AlphaHS}, which we do not repeat here. Note that after using \eqref{def:cX}, \eqref{def:cXdx}, \eqref{def:mX}, and \eqref{def:mXDx}, Lemma~\ref{lem:commutation} rewrites as 
\begin{equation}\label{eq:identitiesCom2}
\mathcal{X}(t,r)= S_t(\bar {\mathcal{X}})(r)  \quad \text{ and }\quad  \hat {\mathcal{X}}_{\Dx}(t,r)= S_t(\bar {\mathcal{X}}_{\Dx})(r),
\end{equation}
that is, $\mathcal{X}(t)$ and $\hat{\mathcal{X}}_{\Dx}(t)$ solve \eqref{eq:LagrSystem} with the wave breaking functions
\begin{equation}\label{Tau}
\T(r)= \begin{cases} \tau(\phi(r)), & \quad r \in \mathcal{B}^c\\
0, & \quad otherwise,
\end{cases}
\end{equation}
and 
\begin{equation}\label{Taux}
\T_{\Dx}(r)= \begin{cases} \tau_{\Dx}(\psi(r)), & \quad r \in \mathcal{B}^c\\
0, & \quad otherwise,
\end{cases}
\end{equation}
respectively.

We are now ready to take the next step towards a convergence rate for \eqref{eq:trouble}, by establishing the following theorem, which eventually allows us to take advantage of the coordinates $\mathcal{X}(t,r)$ and $\hat{\mathcal{X}}_{\Dx}(t,r)$. 

\begin{theorem}\label{thm:Aprior1}
	Suppose $(\bar{u}, \bar{\mu}, \bar{\nu}) \in \D^{\alpha}_0$ and $t\in [0,T]$. Let $\bar{X} = L((\bar{u}, \bar{\mu}, \bar{\nu}))$ and  $\bar{X}_{\Dx} = L \circ P_{\Dx}((\bar u, \bar \mu, \bar \nu))$. Moreover, introduce $X(t) = S_t (\bar{X})$, $\hat X_{\Dx}(t) = S_t(\bar{X}_{\Dx})$, $\mathcal{X}(t)= X(t)\bullet \phi$, and $\hat{\mathcal{X}}_{\Dx}(t)= \hat X_{\Dx}(t) \bullet \psi$ with $\phi$ and $\psi$ from Definition~\ref{def:maps}, then
	\begin{align*}
	\bigg|\!\int_0^t\! \! \Big ( \int_{-\infty}^{\xi} \! \!\big(V_{\xi}&- \hat V_{\Dx, \xi}\big)(s, \eta) d\eta -  \int_{\xi}^{\infty} \! \!\big(V_{\xi} - \hat V_{\Dx, \xi}\big)(s, \eta) d\eta \Big)ds \bigg| \\
	& \leq 4t\Dx + 4t \|\delta_{2\Dx}\bar{u}_x^2\|_1 + \int_0^t  \|\mathcal{V}_r(s) -\hat{\mathcal{V}}_{\Dx,r} (s)\|_{L^1(\mathcal{B}^c)}ds,
	\end{align*}
where $\delta_{2\Dx}f(x)= f(x+2\Dx)-f(x)$.
\end{theorem}

\begin{proof}
To this end, note that for any $j\in \mathbb{Z}$ one has, 
by Theorem~\ref{thm:changeVar}, \eqref{eq:coincideGrid}, \eqref{cond:WB}, \eqref{cond:WBD}, and Lemma~\ref{lem:coincidingLengths},
\begin{align}\nonumber
\int_{\xi_{3j}}^{\xi_{3j+3}} (V_\xi- \hat V_{\Dx, \xi})(s, \eta) d\eta & = \int_{\xi_{3j}}^{\xi_{3j+3}} (\mathcal{V}_r - \hat {\mathcal{V}}_{\Dx,r})(s,r) dr\\ \nonumber
& =- \int_{\mathcal{B}_{\Dx,j}} \hat{\mathcal{V}}_{\Dx,r}(s,r)dr + \int_{\mathcal{B}_{j}}\mathcal{V}_r(s,r) dr\\ \nonumber
&\qquad + \int_{[\xi_{3j}, \xi_{3j+3})\cap \mathcal{B}^c}(\mathcal{V}_r- \hat {\mathcal{V}}_{\Dx,r})(s,r) dr\\ \nonumber
& = -2\mathrm{meas}(\mathcal{B}_{\Dx,j})+ 2\mathrm{meas}(\mathcal{B}_{j})\\ \nonumber
& \qquad + \int_{[\xi_{3j}, \xi_{3j+3})\cap \mathcal{B}^c}(\mathcal{V}_r- \hat {\mathcal{V}}_{\Dx,r})(s,r) dr\\ \label{eq:absPart12}
& = \int_{[\xi_{3j}, \xi_{3j+3})\cap \mathcal{B}^c}(\mathcal{V}_r - \hat {\mathcal{V}}_{\Dx,r})(s,r) dr.
\end{align}
The above argument shows that for any $j \in \mathbb{Z}$ one has
\begin{align}\label{eq:absPart}
\biggl| \int_0^t\Big( \int_{-\infty}^{\xi_{3j}}(V_{\xi} - \hat V_{\Dx, \xi})(s, \eta)d\eta &-\int_{\xi_{3j+3}}^{\infty}(V_{\xi}-\hat V_{\Dx, \xi})(s, \eta)d\eta \Big)ds \biggr| \nonumber\\ 
&\quad  \leq  \int_0^t\int_{\mathcal{B}^c}\! \! \left|\mathcal{V}_r - \hat{\mathcal{V}}_{\Dx,r} \right|\!(s,r)drds,
\end{align}
and, as there exists, for any $\xi\in \R$, some $N \in \mathbb{Z}$ such that $\xi \in [\xi_{3N}, \xi_{3N+3})$, one obtains
\begin{align}\label{eq:VxiDiff}
		\biggl | \int_0^t \biggl (\int_{-\infty}^{\xi}(V_{\xi}&- \hat V_{\Dx, \xi})(s, \eta)d\eta  - \int_{\xi}^{\infty}(V_{\xi} - \hat V_{\Dx, \xi})(s, \eta)d\eta \biggr)ds \biggr| \nonumber 
		\\ & \leq \int_0^t\int_{\xi_{3N}}^{\xi_{3N+3}} \!|V_{\xi} - \hat V_{\Dx, \xi}|(s, \eta)d\eta ds   +  \int_0^t\int_{\mathcal{B}^c}\! \left|\mathcal{V}_r - \hat{\mathcal{V}}_{\Dx,r} \right| \!(s,r)drds. 
	\end{align}
It therefore remains to estimate the first term in \eqref{eq:VxiDiff}. The main idea, which is taken from the proof of \cite[Thm. 4.6]{AlphaRate}, is to decompose the integration domain in \eqref{eq:VxiDiff} by using the sets $\mathcal{S}$ and $\mathcal{S}_{\Dx}$ from \eqref{eq:S} and \eqref{eq:Sdx}, respectively. To this end, observe that according to \eqref{eq:LagrSystem}, we have, for all $t \in [0, T]$, 
\begin{align*}
	V_{\xi}(t, \xi) &= 1 \hspace{0.85cm} \text{for a.e. } \xi \in \mathcal{S}, \\
	\hat V_{\Dx, \xi}(t, \xi) &= 1 \hspace{0.85cm} \text{for a.e. } \xi \in \mathcal{S}_{\Dx},
\end{align*}
such that 
\begin{subequations}\label{eq:rest}
\begin{align}
	\int_0^t\!\int_{\xi_{3N}}^{\xi_{3N+3}} \!\!|V_{\xi} - \hat V_{\Dx, \xi}|(s, \eta)d\eta ds &\leq \! \int_0^t \!\int_{\mathcal{S}^c \cap \mathcal{S}_{\Dx,N}} \! \!\!\big(1-V_{\xi}(s, \eta) \big) d\eta ds \label{eq:rest1} \\
	&  \hspace{-0.45cm} + \int_0^t \!\int_{\mathcal{S}_N \cap \mathcal{S}_{\Dx}^c} \! \! \! \big(1 - \hat V_{\Dx, \xi}(s, \eta) \big)d\eta ds \label{eq:rest2}\\
	& \hspace{-0.45cm} + \int_0^t \! \int_{\mathcal{S}^c \cap[\xi_{3N+1}, \xi_{3N+3})} \! \!|V_{\xi} - \hat V_{\Dx, \xi}|(s, \eta)d\eta ds, \label{eq:rest3}
\end{align}
\end{subequations}
by \eqref{def:SSDxj}.

The definition of $V_\xi(t,\xi)$, see \eqref{eq:ODE3}, combined with $X(0,\xi)=\bar X(\xi)\in \F_0^\alpha$ implies, for any $\xi \in \mathcal{S}^c$,
\begin{align*}
0\leq 1- V_\xi(t,\xi)& = \begin{cases} 1- \bar V_\xi(\xi), & t<\tau(\xi),\\ 1-(1-\alpha(y(\tau(\xi),\xi)))\bar V_\xi(\xi), & \tau(\xi)\leq t, \end{cases}\\
&  =\begin{cases}  \bar y_\xi(\xi) , & t<\tau(\xi),\\ \bar y_\xi(\xi)+ \alpha(y(\tau(\xi), \xi))\bar V_\xi(\xi), & \tau(\xi)\leq t, \end{cases}
\end{align*}
and, in particular, 
\begin{equation*}
0\leq 1-V_\xi(t,\xi)\leq \bar y_\xi(\xi)+ \bar V_\xi(\xi) \quad \text{ for }\xi \in \mathcal{S}^c\cap \mathcal{S}_{\Dx,N},
\end{equation*}
which in turn yields 
\begin{align}\label{semiresult}
	\int_0^t \!\int_{\mathcal{S}^c \cap \mathcal{S}_{\Dx,N}}\!\! \big(1-V_{\xi}(s, \eta) \big) d\eta ds 
	 \leq t\int_{\mathcal{S}^c \cap \mathcal{S}_{\Dx,N}} \! \! \big(\bar y_\xi(\eta)+ \bar V_\xi(\eta) \big) d\eta.
	\end{align}
Since $\bar y(\xi_{3j})=x_{2j}$ for all $j\in \mathbb{Z}$, cf. \eqref{eq:coincideEnd}, it follows that 
\begin{equation}\label{eq:resultRest11}
	 \int_{\mathcal{S}^c \cap \mathcal{S}_{\Dx,N}}\bar{y}_{\xi}(\eta)d\eta \leq \int_{\xi_{3N}}^{\xi_{3N+3}}\bar{y}_{\xi}(\eta)d\eta = 2\Dx. 
\end{equation}
For the second integral, combining \eqref{rep:Xxi} and Theorem~\ref{thm:changeVar} leads to the following estimate 
\begin{align}\label{eq:resultRest21}
	\int_{\mathcal{S}^c \cap \mathcal{S}_{\Dx,N}} \! \bar{V}_{\xi}(\eta)d\eta &\leq \int_{\mathcal{S}^c\cap[\xi_{3N}, \xi_{3N+3})}\bar{u}_x^2(\bar{y}(\eta))\bar{y}_{\xi}(\eta)d\eta \nonumber
	\\ & \leq \int_{x_{2j}}^{x_{2j+2}}\bar{u}_x^2(z)dz \nonumber
	\\ & =  \Big( \int_{-\infty}^{x_{2j+2}}\bar{u}_x^2(z)dz - \int_{-\infty}^{x_{2j}}\bar{u}_x^2(z)dz \Big) \nonumber
	\\ & \leq  \|\delta_{2\Dx}\bar{u}_x^2 \|_1.
\end{align}
After combining \eqref{semiresult}--\eqref{eq:resultRest21} we end up with 
\begin{align}\label{eq:VxiDiff4}
	\int_0^t \!\int_{\mathcal{S}^c \cap \mathcal{S}_{\Dx,N}}\!\! \big(1-V_{\xi}(s, \eta)\big)d\eta ds 
	 \leq 2t\Dx +  t\|\delta_{2\Dx}\bar{u}_x^2 \|_1 .
	\end{align}

To estimate \eqref{eq:rest2} from above, we can proceed in the same way, because  
\begin{equation}\label{eq:CoincideFac}
\int_{-\infty}^{x_{2j}}\bar{u}_{\Dx, x}^2(z)dz = \bar{F}_{\Dx, \mathrm{ac}}(x_{2j}) = \bar{F}_{\mathrm{ac}}(x_{2j}) = \int_{-\infty}^{x_{2j}}\bar{u}_x^2(z)dz \hspace{0.4cm} \text{for all } j \in \mathbb{Z}, 
\end{equation} 
such that 
\begin{align}\label{eq:VxiDiff5}
	\int_0^t \!\int_{\mathcal{S}_N \cap \mathcal{S}_{\Dx}^c}\!\! \big(1-\hat V_{\Dx, \xi}(s, \eta) \big) d\eta ds 
	 \leq 2t\Dx +  t\|\delta_{2\Dx}\bar{u}_x^2 \|_1 .
	\end{align}

It remains to bound \eqref{eq:rest3}. Here we use  that $V_{\xi}(s, \xi) \leq \bar{V}_{\xi}(\xi)$  and $V_{\Dx, \xi}(s, \xi) \leq \bar{V}_{\Dx, \xi}(\xi)$ by \eqref{eq:LagrSystem} together with \eqref{eq:nuac}, \eqref{eq:nudac}, \eqref{eq:coincideEnd}, and \eqref{eq:CoincideFac}
 to obtain 
\begin{align*}
	\int_0^t \! \int_{\mathcal{S}^c \cap[\xi_{3N+1}, \xi_{3N+3})}&  \! \!|V_{\xi} - \hat V_{\Dx, \xi}|(s, \eta)d\eta ds \nonumber\\
	& \qquad \leq t \!\left( \int_{\mathcal{S}^c \cap  [\xi_{3N+1}, \xi_{3N+3})} \! \bar V_\xi (\eta) d\eta + \int_{ [\xi_{3N+1}, \xi_{3N+3})} \! \bar V_{\Dx, \xi}(\eta) d\eta\right) \nonumber \\
	& \qquad \leq t \! \left( \int_{x_{2j}}^{x_{2j+2}}( \bar u_x^2 + \bar u_{\Dx, x}^2 )(x) dx\right) \nonumber \\ 
	& \qquad \leq 2t \|\delta_{2\Dx}\bar{u}_x^2 \|_1.
\end{align*}
Combining the above estimate with \eqref{eq:VxiDiff}, \eqref{eq:rest}, \eqref{eq:VxiDiff4}, and \eqref{eq:VxiDiff5} finishes the proof. 
\end{proof}

\begin{corollary}\label{cor:Aprior1}
Suppose $(\bar{u}, \bar{\mu}, \bar{\nu}) \in \D^{\alpha}_0$ and $t\in [0,T]$. Let $\bar{\mathcal{X}} = (L((\bar{u}, \bar{\mu}, \bar{\nu})))\bullet \phi$ and $\bar{\mathcal{X}}_{\Dx} = (L \circ P_{\Dx}((\bar u, \bar \mu, \bar \nu)))\bullet \psi$, where $\phi$ and $\psi$ are given by Definition~\ref{def:maps}. Moreover, introduce $\mathcal{X}(t)= S_t (\bar{\mathcal{X}})$ and $\hat{\mathcal{X}}_{\Dx}(t)= S_t(\bar{\mathcal{X}}_{\Dx})$, then 
\begin{align*}
\! \! \bigg|  \int_{-\infty}^{r} \! \!\big(\mathcal{V}_{r}- \hat{\mathcal{V}}_{\Dx, r}\big)(t, \eta) d\eta & -  \int_{r}^{\infty} \! \!\big(\mathcal{V}_{r} - \hat{\mathcal{V}}_{\Dx, r}\big)(t, \eta) d\eta \bigg|\\
&  \leq \| \mathcal{V}_r(t)- \hat{\mathcal{V}}_{\Dx,r}(t)\|_{L^1(\mathcal{B}^c)}+ 2\bar \nu_{\mathrm{sing}} \big([x_{2j}, x_{2j+2})\big). 
\end{align*}
\end{corollary}

\begin{proof}
By \eqref{eq:absPart12}, we have 
\begin{align*}\nonumber
\bigg|  \int_{-\infty}^{\xi_{3j}}(\mathcal{V}_{r} &- \hat{\mathcal{V}}_{\Dx, r})(t, \eta)d\eta -\int_{\xi_{3j+3}}^{\infty}(\mathcal{V}_r-\hat{\mathcal{V}}_{\Dx, r} )(t, \eta)d\eta \bigg| \\
 &  \hspace{-0.2cm}\leq \|\mathcal{V}_r(t) - \hat{\mathcal{V}}_{\Dx,r} (t)\|_{L^1(\mathcal{B}^c \cap(-\infty, \xi_{3j}))} + \|\mathcal{V}_r(t) - \hat{\mathcal{V}}_{\Dx,r} (t)\|_{L^1(\mathcal{B}^c \cap[\xi_{3j+3}, \infty))} ,
\end{align*}
while, using \eqref{cond:WB}, \eqref{cond:WBD}, and Lemma~\ref{lem:coincidingLengths} yields 
\begin{align*}
\bigg|  \int_{\xi_{3j}}^r(\mathcal{V}_{r} &- \hat{\mathcal{V}}_{\Dx,r})(t, \eta)d\eta
-\int_r^{\xi_{3j+3}}(\mathcal{V}_r-\hat{\mathcal{V}}_{\Dx, r})(t, \eta)d\eta \bigg| \\
& \leq \|\mathcal{V}_r(t) - \hat{\mathcal{V}}_{\Dx,r} (t)\|_{L^1(\mathcal{B}^c \cap [\xi_{3j}, \xi_{3j+3}))}+ 2(\mathrm{meas}(\mathcal{B}_j)+\mathrm{meas}(\mathcal{B}_{\Dx,j}))\\
&\leq \|\mathcal{V}_r (t)- \hat{\mathcal{V}}_{\Dx,r} (t)\|_{L^1(\mathcal{B}^c \cap [\xi_{3j}, \xi_{3j+3}))}+2\bar\nu_{\mathrm{sing}}\big([x_{2j}, x_{2j+2})\big). \qedhere
\end{align*}
 \end{proof}

To further improve Theorem~\ref{thm:Aprior1}, we need the following proposition, which provides upper bounds on the size of the sets along which wave breaking has already occurred. 

\begin{prop}\label{prop:measureBound}
	Suppose $(\bar{u}, \bar{\mu}, \bar{\nu}) \in \D^{\alpha}_0$ and $t\in [0,T]$. Let $\bar{\mathcal{X}} = (L((\bar{u}, \bar{\mu}, \bar{\nu})))\bullet \phi$ and $\bar{\mathcal{X}}_{\Dx} = (L \circ P_{\Dx}((\bar u, \bar \mu, \bar \nu)))\bullet \psi$, where $\phi$ and $\psi$ are given by Definition~\ref{def:maps}. Moreover, introduce $\mathcal{X}(t)= S_t (\bar{\mathcal{X}})$ and $\hat{\mathcal{X}}_{\Dx}(t)= S_t(\bar{\mathcal{X}}_{\Dx})$, then \vspace{-0.1cm}
	\begin{equation*}
		\mathrm{meas}\big( \{r \! \mid \T(r) \leq t \} \big)\! \leq \!\Big(1 + \frac{1}{4}t^2 \Big)\bar{G}_{\infty},
	\end{equation*}
	and \vspace{-0.175cm}
	\begin{equation*}
	\mathrm{meas}\big( \{r \! \mid  \T_{\Dx}(r) \leq t \} \big) \! \leq \! \Big(1 + \frac{1}{4}t^2 \Big)\bar{G}_{\infty}.
\end{equation*}
\end{prop}

\begin{proof}
Our approach is based on the proof of \cite[Cor. 2.4]{AlphaHS}. Moreover, we only prove the upper bound on $\mathrm{meas}(\{ r\mid \T(r)\leq t\})$, since the other one is shown in exactly the same way. 
	
	Since  $\bar{X}$ and $\bar{X}_{\Dx}$ both belong to $\F_0^{\alpha, 0}$, we have, using \eqref{def:cX}, \eqref{def:cXdx}, and Proposition~\ref{prop:propertiesmaps}, 
	\begin{equation*}
		0 < (\bar{\mathcal{Y}}_r + \bar{\mathcal{H}}_r + \bar{\mathcal{Y}}_{\Dx,r} + \bar{\mathcal{H}}_{\Dx,r} \big) (r)= \big(\dot \phi + \dot\psi \big)(r) = 2.
	\end{equation*} 
	Furthermore, by \eqref{eq:LagrSystem} and \eqref{eq:identitiesCom2}, one has, for $r \in \R$ with $\T(r) \leq t$, 
	\begin{align} \nonumber
		0 &= \mathcal{Y}_r(\T(r),r) = \bar{\mathcal{Y}}_r (r)+ \T(r)\bar{\mathcal{U}}_{r}(r) + \frac{1}{4}\T^2(r)\bar{\mathcal{H}}_r(r), \\ \label{rel:TbU}
		0 &= \mathcal{U}_r(\T(r), r) = \bar{\mathcal{U}}_r(r) + \frac{1}{2}\T(r)\bar{\mathcal{H}}_r(r). 
	\end{align}
	Combining these, reveals that \vspace{-0.1cm}
	\begin{equation}\label{rel:TbY}
	\bar{\mathcal{Y}}_r(r) = \frac{1}{4}\T^2(r)\bar{\mathcal{H}}_r(r), 
	\end{equation}
	and we therefore get, using that $\bar{\mathcal{Y}}_r= \bar{\mathcal{Y}}_{\Dx,r}$, cf. \eqref{prop:YYD}, and \eqref{def:cXdx},
	\begin{align*}
		\mathrm{meas} \big(\{r \! \mid  \T(r) \leq t \} \big) & =  \frac{1}{2}\int_{\{r \mid \T(r) \leq t \}}\!\Big(\bar{\mathcal{Y}}_r + \bar{\mathcal{H}}_r + \bar{\mathcal{Y}}_{\Dx,r} + \bar{\mathcal{H}}_{\Dx,r} \Big) (r) dr 
		\\ &\leq \frac{1}{2} \int_{\R} \Big( \big(1+\frac{1}{2}t^2\big)\bar{\mathcal{H}}_r +  \bar{\mathcal{H}}_{\Dx, r}\Big)(r)dr.
	\end{align*} 
	Since $\bar{H}_{\Dx, \infty}   = \bar{G}_{\Dx, \infty} = \bar{G}_{\infty} = \bar{H}_{\infty}$ by Definition~\ref{def:MapL} and Definition~\ref{def:ProjOperator}, Theorem~\ref{thm:changeVar}, \eqref{def:cX}, and \eqref{def:cXdx} imply 
	\begin{equation*}
	\mathrm{meas}\big(\{r \! \mid \T(r) \leq t \} \big) \leq \frac{1}{2} \Big(1+\frac{1}{2}t^2 \Big)\bar{G}_{\infty} + \frac{1}{2}\bar{G}_{\Dx, \infty} =  \Big(1+\frac{1}{4}t^2 \Big)\bar{G}_{\infty}. \qedhere
	\end{equation*} 
\end{proof}

\begin{theorem}\label{thm:Aprior2}
	Suppose $(\bar{u}, \bar{\mu}, \bar{\nu}) \in \D^{\alpha}_0$ and $t\in [0,T]$. Let $\bar{\mathcal{X}} = (L((\bar{u}, \bar{\mu}, \bar{\nu})))\bullet \phi$ and $\bar{\mathcal{X}}_{\Dx} = (L \circ P_{\Dx}((\bar u, \bar \mu, \bar \nu)))\bullet \psi$, where $\phi$ and $\psi$ are given by Definition~\ref{def:maps}. Moreover, introduce $\mathcal{X}(t)= S_t (\bar{\mathcal{X}})$ and $\hat{\mathcal{X}}_{\Dx}(t)= S_t(\bar{\mathcal{X}}_{\Dx})$, then \vspace{-0.1cm}
\begin{align*}
		 \int_0^t & \|\mathcal{V}_r (s)-\hat{\mathcal{V}}_{\Dx,r}(s)\|_{L^1(\mathcal{B}^c)}ds \\	
& \qquad \leq \bigg(T \Big(4+T \|\alpha'\|_\infty  \Big(\|\bar u \|_\infty+\frac14T  \bar F_{\infty}\Big) \! \Big)\|\bar{\mathcal{H}}_r-\bar{\mathcal{H}}_{\Dx,r}\|_{L^1(\mathcal{B}^c)}\\ \nonumber
& \qquad \qquad  + T\| g(\bar{\mathcal{X}})- g(\bar{\mathcal{X}}_{\Dx})\|_{L^1 (\mathcal{B}^c)}\\ \nonumber
& \qquad \qquad + 4\Big(1+ T\|\alpha'\|_\infty \Big(\|\bar u\|_\infty + (1+ \sqrt{2})\bar F_\infty^{\nicefrac{1}{2}} \Dx^{\nicefrac{1}{2}}+ \frac14 T \bar F_\infty \Big)\! \Big) \\ \nonumber
& \qquad \qquad \qquad \times \! \Big(1+ \frac12 T \Big) \bar G_\infty^{\nicefrac{1}{2}}\|\bar{\mathcal{U}}_r-\bar{\mathcal{U}}_{\Dx,r}\|_{L^2(\mathcal{B}^c)} \\ \nonumber
& \qquad \qquad +  T^2\|\alpha'\|_\infty \bar F_\infty \|\bar u - \bar u_{\Dx}\|_\infty + 4T\|\alpha'\|_\infty \bar G_\infty \Dx \bigg) e^{\frac14 T^2 \|\alpha'\|_\infty  \bar G_\infty}.		 
 \end{align*}
\end{theorem}

\begin{proof}
Recall \eqref{def:Omegamn}, which together with Fubini's theorem, allows us to write 
\begin{subequations}
	\begin{align}
	 \int_0^t\int_{\mathcal{B}^c}\! \! \bigl|\mathcal{V}_r - \hat{\mathcal{V}}_{\Dx,r} \bigr|(s,r)drds &=  \int_{\bar \Omega_{c, c}}\int_0^t   \!\bigl|\mathcal{V}_r - \hat{\mathcal{V}}_{\Dx,r} \bigr|(s,r)dsdr  \nonumber 
	 \\  & \quad +  \int_{\bar \Omega_{d, c}}\int_0^t  \! \bigl|\mathcal{V}_r - \hat{\mathcal{V}}_{\Dx,r} \bigr|(s,r)dsdr \label{eq:Vxi2}
	 \\ & \quad + \int_{\bar \Omega_{c, d}}\int_0^t \! \bigl|\mathcal{V}_r - \hat{\mathcal{V}}_{\Dx,r} \bigr|(s,r)dsdr  \label{eq:Vxi3}
	 \\ & \quad +  \int_{\bar \Omega_{d, d}}\int_0^t  \!\bigl|\mathcal{V}_r - \hat{\mathcal{V}}_{\Dx,r} \bigr|(s,r)dsdr.  \label{eq:Vxi4}
	 \end{align}
	 \end{subequations}
	 
No wave breaking takes place in the future for $r\in \bar \Omega_{c,c}$, and since $(\bar u,\bar \mu, \bar\nu)\in \D_0^{\alpha}$, we have, by \eqref{eq:LagrSystem},
 \begin{equation}\label{eq:none}
	 	\int_0^t \int_{\bar \Omega_{c,c}} \!\!\bigl|\mathcal{V}_r - \hat{\mathcal{V}}_{\Dx,r} \bigr|(s,r)ds dr = t \|\bar{\mathcal{H}}_r- \bar{\mathcal{H}}_{\Dx,r}\|_{L^1(\bar \Omega_{c,c})}.  
\end{equation}

Next, let us have a closer look at \eqref{eq:Vxi2}. Wave breaking will occur for any $r \in \Omega_{d}(\bar{\mathcal{X}})$ at some later time $t^{\star}$ and $r\in \Omega_c(\mathcal{X}(s))$ for $s\geq t^\star$, but $t^{\star}$ might not lie inside $[0, t]$. According to \eqref{eq:LagrSystem}, \eqref{eq:identitiesCom2}, and \eqref{Tau}, we thus have 
\begin{align*}
		\int_{\bar \Omega_{d, c}}\int_0^t  |&\mathcal{V}_r - \hat{\mathcal{V}}_{\Dx,r}|(s,r)dsdr  
		\\ &= \int_{\bar \Omega_{d, c} }\int_0^{\min(t, \T(r))}  |\bar{\mathcal{H}}_r - \bar{\mathcal{H}}_{\Dx,r}|(r)dsdr 
		\\ & \qquad \hspace{0.1cm} +  \int_{\bar \Omega_{d, c}}\int_{\min(t, \T(r))}^t   |\mathcal{V}_r (s,r)- \bar{\mathcal{V}}_{\Dx,r}(r)| dsdr
		\\ & \leq  t\|\bar{\mathcal{H}}_r- \bar{\mathcal{H}}_{\Dx,r}\|_{L^1(\bar\Omega_{d,c})}
		\\ & \qquad \hspace{0.1cm}+ \int_{\bar \Omega_{d, c}} \int_{\min(t,\T(r))}^t  \! \! |\!\left(1-\alpha(\mathcal{Y}(\T(r),r))\right)\bar{\mathcal{V}}_r(r) - \bar{\mathcal{V}}_{\Dx,r}(r)| dsdr. 
	\end{align*}
As $\alpha\in W^{1, \infty}(\R, [0,1))$, we obtain 
\begin{align}\nonumber
	 \vert \alpha\left(\mathcal{Y}(\T(r),r)\right)-\alpha(\bar{\mathcal{Y}}(r))\vert\bar{\mathcal{V}}_r(r) &= \bigg \vert  \int_0^{\T(r)}\frac{d}{ds}\big(\alpha(\mathcal{Y}(s,r)) \big)ds \bigg\vert \bar{\mathcal{V}}_r(r) \nonumber
	\\ &\leq  \int_0^{\T(r)} \big\vert \alpha'(\mathcal{Y}(s, r))\mathcal{U}(s, r)\big\vert ds \bar{\mathcal{V}}_r(r) \nonumber\\ \nonumber
	& \leq 2\|\alpha'\|_\infty \max_{s\in [0,T]}\|\mathcal{U}(s)\|_\infty  \frac12\T(r) \bar{\mathcal{V}}_r(r)\\ \nonumber
	& \leq  2\|\alpha'\|_\infty \max_{s\in [0,T]}\|\mathcal{U}(s)\|_\infty \big(\mathcal{U}_r(\T(r),r)-\bar{\mathcal{U}}_r(r) \big)\\ \label{eq:estalpha}
	& \leq 2\|\alpha'\|_\infty \max_{s\in [0,T]}\|\mathcal{U}(s)\|_\infty \big(\bar{\mathcal{U}}_{\Dx,r}(r)-\bar{\mathcal{U}}_r(r) \big),
\end{align}
where we, in the last step, used that $\bar {\mathcal{U}}_r(r)<\mathcal{U}_r(\T(r),r)=0\leq \bar{\mathcal{U}}_{\Dx,r}(r)$. Furthermore, one has, for all $(b,r)\in [0,T]\times \R$, 
\begin{equation*}
\vert \mathcal{U}(b,r)\vert \leq \|\bar{\mathcal{U}}\|_\infty + \frac14 T \bar{\mathcal{V}}_\infty = \|\bar u\|_\infty + \frac14 T \bar F_\infty.
\end{equation*}
Recalling \eqref{eq:g}, we therefore end up with 
\begin{align} \nonumber
 \int_{\bar \Omega_{d, c}} \int_0^t |&\mathcal{V}_r - \hat{\mathcal{V}}_{\Dx,r}|(s,r)dsdr\\ \nonumber
& \leq    t\| \bar{\mathcal{H}}_r-\bar{\mathcal{H}}_{\Dx,r}\|_{L^1(\bar \Omega_{d,c})} \\ \nonumber
&  \quad \hspace{0.1cm}+\int_{\bar \Omega_{d, c}}\int_{\min(t,\T(r))}^t  \! \! |\!\left(1-\alpha(\bar{\mathcal{Y}}(r))\right)\!\bar{\mathcal{V}}_r(r) - \bar{\mathcal{V}}_{\Dx,r}(r)| dsdr\\ \nonumber
&  \quad \hspace{0.1cm}+  \int_{\bar \Omega_{d, c}}\int_{\min(t,\T(r))}^t  \left|\alpha(\mathcal{Y}(\T(r),r))-\alpha(\bar{\mathcal{Y}}(r))\right\vert \bar{\mathcal{V}}_r(r)dsdr\\ \nonumber
& \leq t \|\bar{\mathcal{H}}_r- \bar{\mathcal{H}}_{\Dx,r}\|_{L^1(\bar \Omega_{d,c})}+ t \|g(\bar{\mathcal{X}})- g(\bar {\mathcal{X}}_{\Dx})\|_{L^1(\bar\Omega_{d,c})}\\\label{eq:onlyOne}
& \quad \hspace{0.1cm} + 2t \|\alpha'\|_\infty \Big(\|\bar u\|_\infty + \frac14 T\bar F_\infty \Big)\|\bar{\mathcal{U}}_r-\bar{\mathcal{U}}_{\Dx,r}\|_{L^1(\bar \Omega_{d,c}\cap \Omega_c(\mathcal{X}(t)))}.
\end{align}

For \eqref{eq:Vxi3} we can follow the same lines. Since \eqref{eq:LagrSystem} and \eqref{eq:u1} imply for all $(b,r)\in [0,T]\times \R$, 
\begin{equation*}
\vert \mathcal{U}_{\Dx}(b,r)\vert \leq \|\bar u_{\Dx}\|_\infty + \frac14 T \bar F_{\Dx,\infty} \leq \|\bar u\|_\infty +  (1+\sqrt{2})\bar{F}_{\mathrm{ac}, \infty}^{\nicefrac{1}{2}}\Dx^{\nicefrac{1}{2}} + \frac{1}{4}T\bar{F}_{\infty},
\end{equation*}
we end up with 
\begin{align} \label{eq:onlyTwo}
 \int_{\bar \Omega_{c, d}}& \int_0^t |\mathcal{V}_r - \hat{\mathcal{V}}_{\Dx,r}|(s,r)dsdr \nonumber \\ \nonumber
& \leq t \|\bar{\mathcal{H}}_r- \bar{\mathcal{H}}_{\Dx,r}\|_{L^1(\bar \Omega_{c,d})}+ t \|g(\bar{\mathcal{X}})- g(\bar {\mathcal{X}}_{\Dx})\|_{L^1(\bar \Omega_{c,d})} \\ 
& \quad  + 2t \|\alpha'\|_\infty \Big(\|\bar u\|_\infty +(1+\sqrt{2})\bar{F}_{\mathrm{ac}, \infty}^{\nicefrac{1}{2}}\Dx^{\nicefrac{1}{2}}+ \frac14 T\bar F_\infty \Big) \nonumber
\\ & \qquad \qquad \qquad \quad \times \|\bar{\mathcal{U}}_r-\bar{\mathcal{U}}_{\Dx,r}\|_{L^1(\bar \Omega_{c,d}\cap \Omega_c(\hat{\mathcal{X}}_{\Dx}(t)))}.
\end{align}

It remains to consider \eqref{eq:Vxi4}, which can be further decomposed as follows
\begin{subequations}\label{eq:bothLater}
\begin{align} \nonumber
	\int_{\bar{\Omega}_{d, d}} \int_0^t  |\mathcal{V}_r-&  \hat{\mathcal{V}}_{\Dx, r}|(s,r)ds dr \\
	&=  \int_{\bar{\Omega}_{d, d}} \int_0^{\min(t,\T(r), \T_{\Dx}(r))} \! \big|\bar{\mathcal{H}}_r -  \bar{\mathcal{H}}_{\Dx,r} \big|(r)ds dr \label{eq:bothLater20}
	\\ & \quad + \int_{\bar{\Omega}_{d, d}}\int_{\min(t,\T(r), \T_{\Dx}(r))}^{\min(t, \max(\T(r), \T_{\Dx}(r)))} \! |\mathcal{V}_r- \hat{\mathcal{V}}_{\Dx, r}|(s,r)dsdr \label{eq:bothLater21}\\
	& \quad + \int_{\bar{\Omega}_{d, d}}\int_{\min(t, \max(\T(r), \T_{\Dx}(r)))}^t \! |\mathcal{V}_r- \hat{\mathcal{V}}_{\Dx, r}|(s,r)dsdr \label{eq:bothLater22}.
\end{align}
\end{subequations}
Here we need to have a closer look at the terms \eqref{eq:bothLater21} and \eqref{eq:bothLater22}. To this end, assume, without loss of generality, that $0<\T(r)<\T_{\Dx}(r)<t$. The remaining possibilities can be treated, up to slight modifications, using the same idea. 

To estimate the time integral of \eqref{eq:bothLater21}, note that $\mathcal{U}_r(\T(r), r) = 0 = \mathcal{U}_{\Dx, r}(\T_{\Dx}(r), r)$, and it therefore follows from \eqref{eq:LagrSystem} that 
\begin{align} \nonumber
(\T_{\Dx}(r)-\T(r)) \bar{\mathcal{V}}_{\Dx,r}(r) &= 2(\hat{\mathcal{U}}_{\Dx,r}(\T_{\Dx}(r),r)- \hat{\mathcal{U}}_{\Dx,r}(\T(r),r))\\ \nonumber
& = 2 (\mathcal{U}_r- \hat{\mathcal{U}}_{\Dx,r})(\T(r), r)\\ \label{rel:TUV}
& = 2(\bar{\mathcal{U}}_r- \bar{\mathcal{U}}_{\Dx,r})(r) + \T(r)(\bar{\mathcal{V}}_r- \bar{\mathcal{V}}_{\Dx,r})(r), 
\end{align}
and hence 
\begin{align}\label{eq:bothLater21Est}
\int_{\T(r)}^{\T_{\Dx}(r)} |\mathcal{V}_r- \hat{\mathcal{V}}_{\Dx, r}|(s,r)ds& = (\T_{\Dx}(r)-\T(r)) (1- \alpha(\mathcal{Y}(\T(r), r)))\vert \bar{\mathcal{V}}_r- \bar{\mathcal{V}}_{\Dx,r}\vert (r) \nonumber \\
& \qquad + (\T_{\Dx}(r)-\T(r)) \alpha(\mathcal{Y}(\T(r),r)) \bar{\mathcal{V}}_{\Dx,r}(r) \nonumber \\
& \leq 2 \big\vert \bar{\mathcal{U}}_r- \bar{\mathcal{U}}_{\Dx,r} \big\vert (r)+ 2t\big \vert \bar{\mathcal{V}}_r- \bar{\mathcal{V}}_{\Dx,r} \big\vert(r). 
\end{align}

For \eqref{eq:bothLater22}, we again start by first considering the time integral. Observe that for $r\in [\xi_{3j}, \xi_{3j+3})$, using \eqref{eq:LagrSystem}, \eqref{prop:YYD},  \eqref{eq:identitiesCom2}, and Corollary~\ref{cor:Aprior1}, yields
\begin{align*}
\vert \mathcal{Y}(\T(r),&\hspace{0.0025cm} r)  - \hat {\mathcal{Y}}(\T_{\Dx}(r), r)\vert 
\\ & = \bigg\vert \int_0^{\T(r)} \mathcal{U}(s,r)ds - \int_0^{\T_{\Dx}(r)} \hat{\mathcal{U}}_{\Dx}(s,r)ds \bigg\vert \\
& \leq \big\vert \T(r)\bar{\mathcal{U}}(r)- \T_{\Dx}(r)\bar{\mathcal{U}}_{\Dx}(r) \big\vert \\
& \quad + \bigg \vert \int_0^{\T(r)}\! \int_0^s \mathcal{U}_t(l,r) dlds- \int_0^{\T_{\Dx}(r)}  \!\int_0^s \hat{\mathcal{U}}_{\Dx, t}(l, r) dl ds \bigg\vert \\
& \leq \big(\T_{\Dx}(r) - \T(r) \big) \big\vert \bar{\mathcal{U}}(r) \big\vert +  \T_{\Dx}(r) \big\vert \bar{\mathcal{U}}(r)- \bar{\mathcal{U}}_{\Dx}(r) \big\vert 
\\ & \quad + \bigg \vert \int_{\T(r)}^{\T_{\Dx}(r)} \int_0^s \hat{\mathcal{U}}_{\Dx, t}(l, r)dl ds \bigg \vert \\ 
& \quad +\bigg \vert \int_0^{\T(r)}\int_0^s \big(\mathcal{U}_t - \hat{\mathcal{U}}_{\Dx, t}\big)(l, r) dl ds \bigg \vert \\ 
& \leq (\T_{\Dx}(r)- \T(r)) \|\bar u \|_\infty + t\|\bar u- \bar u_{\Dx}\|_\infty\\
& \quad + \frac18 (\T_{\Dx}^2(r)- \T^2(r))\bar{\mathcal{V}}_\infty  \\
& \quad + \frac14 \bigg\vert \! \int_0^{\T(r)} \! \! \! \int_0^s \! \left(\int_{-\infty}^r \! (\mathcal{V}_r- \hat{\mathcal{V}}_{\Dx,r})(l, \eta) d\eta - \! \int_{r}^\infty \!(\mathcal{V}_r- \hat{\mathcal{V}}_{\Dx,r})(l, \eta) d\eta \right) \! dl ds \bigg\vert\\
& \leq (\T_{\Dx}(r)- \T(r)) \|\bar u \|_\infty + t\|\bar u- \bar u_{\Dx}\|_\infty\\
& \quad  + \frac  14 t(\T_{\Dx}(r)- \T(r))\bar F_{\infty} + \frac14  \int_0^{t}\int_0^s \| \mathcal{V}_r(l)- \hat{\mathcal{V}}_{\Dx,r}(l)\|_{L^1(\mathcal{B}^c)}dlds\\
& \quad + \frac14 \T(r)^2 \bar \nu_{\mathrm{sing}}\big([x_{2j}, x_{2j+2})\big).
\end{align*}
Thus,  for $r\in [\xi_{3j}, \xi_{3j+3})$,  we have
\begin{align*}
\int_{\T_{\Dx}(r)}^t \!  |\mathcal{V}_r &- \hat{\mathcal{V}}_{\Dx, r}|(s,r)ds
\\ & \leq \big(t- \T_{\Dx}(r) \big) \big(1- \alpha(\mathcal{Y}(\T(r), r)) \big)\vert \bar{\mathcal{V}}_r-\bar{\mathcal{V}}_{\Dx,r}\vert(r)\\
& \quad + \big(t-\T_{\Dx}(r) \big) \big\vert \alpha(\mathcal{Y}(\T(r),r))- \alpha(\hat{\mathcal{Y}}_{\Dx}(\T_{\Dx}(r),r)) \big\vert \bar{\mathcal{V}}_{\Dx,r}(r)\\
& \leq t \vert \bar{\mathcal{V}}_r-\bar{\mathcal{V}}_{\Dx,r}\vert(r)\\
& \quad + t\|\alpha'\|_\infty \big\vert \mathcal{Y}(\T(r),r)-\hat{ \mathcal{Y}}_{\Dx}(\T_\Dx(r),r) \big\vert \bar{\mathcal{V}}_{\Dx,r}(r)\\
& \leq t \vert \bar{\mathcal{V}}_r-\bar{\mathcal{V}}_{\Dx,r}\vert(r) + t^2\|\alpha'\|_\infty \|\bar u- \bar u_{\Dx}\|_\infty\bar{\mathcal{V}}_{\Dx,r}(r) \\
& \quad +t \|\alpha'\|_\infty \Big(\|\bar u \|_\infty+\frac14t  \bar F_{\infty} \Big)(\T_{\Dx}(r)- \T(r))\bar{\mathcal{V}}_{\Dx,r}(r)\\
& \quad +  \frac14 t \|\alpha'\|_\infty \T(r)^2 \bar \nu_{\mathrm{sing}}\big([x_{2j}, x_{2j+2})\big)\bar{\mathcal{V}}_{\Dx,r}(r)
\\ & \quad + \frac14t \|\alpha'\|_\infty  \int_0^{t}\int_0^s \| \mathcal{V}_r(l)- \hat{\mathcal{V}}_{\Dx,r}(l)\|_{L^1(\mathcal{B}^c)}dl ds \bar{\mathcal{V}}_{\Dx,r}(r). 
\end{align*}
Furthermore, $\T(r)<\T_{\Dx}(r)$, which implies, by following the same argument as that leading to \eqref{rel:TbY},
\begin{equation*}
\frac14 \T^2(r)\bar{\mathcal{V}}_{\Dx,r}(r)\leq \frac14 \T_{\Dx,r}^2(r)\bar{\mathcal{V}}_{\Dx,r}(r)=\bar{\mathcal{Y}}_{\Dx,r}(r),
\end{equation*}
which in turn, when combined with \eqref{rel:TUV}, for $r\in [\xi_{3j}, \xi_{3j+3})$, gives
\begin{align}\label{eq:bothLater22Est}
\int_{\T_{\Dx}(r)}^t \! \big|\mathcal{V}_r &- \hat{\mathcal{V}}_{\Dx, r} \big|(s,r)ds \nonumber 
\\ & \leq t\Big(1+t\|\alpha'\|_\infty \big(\|\bar u \|_\infty+\frac14t  \bar F_{\infty} \big) \Big) \big\vert \bar{\mathcal{V}}_r-\bar{\mathcal{V}}_{\Dx,r} \big\vert(r) \nonumber \\
& \quad + 2 t \|\alpha'\|_\infty \Big(\|\bar u\|_\infty + \frac14 t\bar F_\infty \Big) \big\vert\bar{\mathcal{U}}_r- \bar{\mathcal{U}}_{\Dx,r} \big\vert (r) \nonumber \\
& \quad + t\|\alpha'\|_\infty\Big( t\|\bar u- \bar u_{\Dx}\|_\infty \bar{\mathcal{V}}_{\Dx,r}(r) +  \bar\nu_{\mathrm{sing}} \big([x_{2j}, x_{2j+2}) \big) \bar{\mathcal{Y}}_{\Dx,r}(r) \Big) \nonumber
\\ 
& \quad  + \frac14 t \|\alpha'\|_{\infty} \int_0^{t} \int_0^s \| \mathcal{V}_r(l)- \hat{\mathcal{V}}_{\Dx,r}(l)\|_{L^1(\mathcal{B}^c)}dl ds\bar{\mathcal{V}}_{\Dx,r}(r).
\end{align}
By combining \eqref{eq:bothLater20}, \eqref{eq:bothLater21Est}, \eqref{eq:bothLater22Est}, and similar estimates for the remaining cases,  we therefore end up with 
\begin{align}\label{eq:both}
\int_{\bar{\Omega}_{d,d}}\int_0^t \! \big\vert \mathcal{V}_r&-\hat{\mathcal{V}}_{\Dx,r} \big\vert (s,r) ds dr \nonumber 
\\& \leq t \Big(4+t\|\alpha'\|_\infty \big(\|\bar u \|_\infty+\frac14t  \bar F_{\infty} \big)\Big) \|\bar{\mathcal{H}}_r-\bar{\mathcal{H}}_{\Dx,r}\|_{L^1(\bar{\Omega}_{d,d})} \nonumber\\ \nonumber
& \quad + 2\Big(1+ t\|\alpha'\|_\infty \big(\|\bar u\|_\infty + \frac14 t \bar F_\infty\big)\Big)\\ \nonumber
& \qquad \qquad \times \|\bar{\mathcal{U}}_r-\bar{\mathcal{U}}_{\Dx,r}\|_{L^1(\bar{\Omega}_{d,d}\cap (\Omega_c(\mathcal{X}(t))\cup\Omega_c(\mathcal{X}_{\Dx}(t))))} \\ \nonumber
& \quad + t^2\|\alpha'\|_\infty \bar F_\infty \|\bar u - \bar u_{\Dx}\|_\infty + 4 t\|\alpha'\|_\infty \bar G_\infty \Dx \\ 
& \quad + \frac14 t \|\alpha'\|_\infty  \bar G_\infty \int_0^{t} \int_0^s\| \mathcal{V}_r(l)- \hat{\mathcal{V}}_{\Dx,r}(l)\|_{L^1(\mathcal{B}^c)}dlds,
\end{align}
where we used the following inequality, which holds as $\bar{\mathcal{Y}}(\xi_{3j}) = x_{2j} = \bar{\mathcal{Y}}_{\Dx}(\xi_{3j})$ for $j \in \mathbb{Z}$, cf. \eqref{eq:coincideEnd}--\eqref{eq:coincideGrid}, \eqref{eq:Y}, and \eqref{prop:YYD}, 
\begin{align*}
\sum_{j\in \mathbb{Z}}\int_{\bar{\Omega}_{d,d}\cap [\xi_{3j}, \xi_{3j+3})}  \! \! &\bar\nu_{\mathrm{sing}}\big([x_{2j}, x_{2j+2})\big)\big(\bar{\mathcal{Y}}_{r}+ \bar{\mathcal{Y}}_{\Dx,r} \big)(r) dr \\
&\leq \sum_{j\in \mathbb{Z}} \bar\nu_{\mathrm{sing}}\big([x_{2j}, x_{2j+2})\big) \int_{\xi_{3j}}^{\xi_{3j+3}}(\bar{\mathcal{Y}}_{r}+ \bar{\mathcal{Y}}_{\Dx,r})(r) dr \\
& =\sum_{j\in \mathbb{Z}} \bar\nu_{\mathrm{sing}} \big([x_{2j}, x_{2j+2}) \big) 2(x_{2j+2}-x_{2j})\leq 4\bar G_\infty \Dx .
\end{align*}
Moreover, by the Cauchy--Schwarz inequality and Proposition~\ref{prop:measureBound}, we have 
\begin{align}\label{eq:L2U}
	 \|\bar{\mathcal{U}}_r-\bar{\mathcal{U}}_{\Dx,r}&\|_{L^1(\{r \mid 0 <\T(r)\leq t \text{ or } 0 <\T_{\Dx}(r)\leq t\})} \nonumber
	\\ & \leq \sqrt{2\big(1 + \frac{1}{4}t^2 \big)}\bar{G}_{\infty}^{\nicefrac{1}{2}}\|\bar{\mathcal{U}}_r-\bar{\mathcal{U}}_{\Dx,r}\|_{L^2(\mathcal{B}^c)}, 
\end{align}
hence, when finally combining \eqref{eq:none}, \eqref{eq:onlyOne}, \eqref{eq:onlyTwo},  \eqref{eq:both}, and \eqref{eq:L2U}, we obtain 
\begin{align} \nonumber
\int_0^t & \|\mathcal{V}_r(s)- \hat{\mathcal{V}}_{\Dx,r}(s)\|_{L^1(\mathcal{B}^c)} ds \\ \nonumber
& \qquad \leq T\Big(4+ T\|\alpha'\|_\infty  \Big(\|\bar u \|_\infty+\frac14T  \bar F_{\infty} \Big) \! \Big) \|\bar{\mathcal{H}}_r-\bar{\mathcal{H}}_{\Dx,r}\|_{L^1(\mathcal{B}^c)}\\ \nonumber
& \qquad \quad  + T\| g(\bar{\mathcal{X}})- g(\bar{\mathcal{X}}_{\Dx})\|_{L^1 (\mathcal{B}^c)}\\ \nonumber
& \qquad \quad + 4 \Big(1+ T\|\alpha'\|_\infty \Big(\|\bar u\|_\infty + (1+ \sqrt{2})\bar F_\infty^{\nicefrac{1}{2}} \Dx^{\nicefrac{1}{2}}+ \frac14 T \bar F_\infty\Big) \! \Big)\\ \nonumber
& \qquad \qquad \quad \times \! \Big(1+\frac{1}{2}T\Big)\bar{G}_\infty^{\nicefrac{1}{2}}\|\bar{\mathcal{U}}_r-\bar{\mathcal{U}}_{\Dx,r}\|_{L^2(\mathcal{B}^c)} \\ \nonumber
& \qquad \quad + T^2\|\alpha'\|_\infty \bar F_\infty \|\bar u - \bar u_{\Dx}\|_\infty+ 4 T\|\alpha'\|_\infty\bar G_\infty \Dx \\ \nonumber
& \qquad \quad + \frac14 T \|\alpha'\|_\infty  \bar G_\infty \! \int_0^{t} \int_0^s\| \mathcal{V}_r(l)-\hat{ \mathcal{V}}_{\Dx,r}(l)\|_{L^1(\mathcal{B}^c)}dlds.
\end{align}
Applying Gronwall's inequality proves the statement. 
\end{proof}

\subsection{Convergence rate in Lagrangian and Eulerian coordinates} At last, we can obtain the 
 convergence rates for \eqref{convrate:rest}, which subsequently enable us to deduce a  convergence rate in Eulerian coordinates. 

\begin{theorem}\label{thm:ConvRateLagr}
Suppose $(\bar u, \bar \mu, \bar \nu)\in \D_0^\alpha$ with $\bar u_x\in B_2^\beta$ for some $\beta\in (0,1]$. Let $\bar{X} = L((\bar{u}, \bar{\mu}, \bar{\nu}))$ and  $\bar{X}_{\Dx} = L \circ P_{\Dx}((\bar u, \bar \mu, \bar \nu))$. Moreover, introduce $X(t) = S_t (\bar{X})$ and $\hat X_{\Dx}(t) = S_t(\bar{X}_{\Dx})$, then there exists a constant $\tilde{D}$, dependent on $T$, $\|\alpha '\|_\infty$, $\|\bar u\|_\infty$,  $\bar G_\infty$, $\beta$, and $|\bar u_x|_{2,\beta}$,  such that 
\begin{equation*}
\sup_{t\in [0,T]} \|y(t) - \hat y_{\Dx}(t)\|_{\infty}+\sup_{t\in [0,T]} \|U(t) - \hat U_{\Dx}(t)\|_{\infty}\leq \tilde{D}(\Dx^{\nicefrac{1}{2}}+ \Dx^{\nicefrac{\beta}{2}}).
\end{equation*}
\end{theorem}

\begin{proof}
Combining \eqref{tow:conv1}, Theorem~\ref{thm:Aprior1}, and Theorem~\ref{thm:Aprior2} yields, for all $t\in [0,T]$, 
\begin{align*}
		\|U(t) &- \hat U_{\Dx}(t)\|_{\infty} 
		\\ & \leq \|\bar{U} - \bar{U}_{\Dx}\|_{\infty} \nonumber
		\\ & \quad+ \frac{1}{4} \bigg|\int_0^t \Big ( \int_{-\infty}^{\xi} \big(V_{\xi} - \hat V_{\Dx, \xi} \big)(s,\eta)d\eta -  \int_{\xi}^{\infty} \big(V_{\xi} - \hat V_{\Dx, \xi} \big)(s, \eta)d\eta \Big)ds \bigg|\\
		& \leq \|\bar{U} - \bar{U}_{\Dx}\|_{\infty} + T( \Dx+ \|\delta_{2\Dx}\bar u_x^2\|_1) 
		\\ & \quad +\frac14  \int_0^T \|\mathcal{V}_r(s)- \hat{\mathcal{V}}_{\Dx,r}(s)\|_{L^1(\mathcal{B}^c)}ds\\
	& \leq  \|\bar{U} - \bar{U}_{\Dx}\|_{\infty} + T\big( \Dx+ \|\delta_{2\Dx}\bar u_x^2\|_1\big) \\
	&\quad+\frac14 \bigg(T \Big(4+T\|\alpha'\|_\infty  \Big(\|\bar u \|_\infty+\frac14T  \bar F_{\infty} \Big)\! \Big) \|\bar{\mathcal{H}}_r-\bar{\mathcal{H}}_{\Dx,r}\|_{L^1(\mathcal{B}^c)}\\ \nonumber
& \qquad \qquad + T\| g(\bar{\mathcal{X}})- g(\bar{\mathcal{X}}_{\Dx})\|_{L^1 (\mathcal{B}^c)}\\ \nonumber
&\qquad \qquad+ 4\Big(1+ T\|\alpha'\|_\infty \Big(\|\bar u\|_\infty + (1+ \sqrt{2})\bar F_\infty^{\nicefrac{1}{2}} \Dx^{\nicefrac{1}{2}}+ \frac14 T \bar F_\infty\Big) \! \Big)\\ \nonumber
& \qquad \qquad \qquad \qquad \times \! \Big(1+ \frac12 T\Big) \bar G_\infty^{\nicefrac{1}{2}}  \|\bar{\mathcal{U}}_r-\bar{\mathcal{U}}_{\Dx,r}\|_{L^2(\mathcal{B}^c)} \\ \nonumber
& \qquad \qquad+  T^2\|\alpha'\|_\infty \bar F_\infty \|\bar u - \bar u_{\Dx}\|_\infty + 4T\|\alpha'\|_\infty\bar G_\infty \Dx \bigg) e^{\frac14 T^2\|\alpha'\|_\infty  \bar G_\infty}.
\end{align*}
Furthermore, Proposition~\ref{prop:ratesPdx}, Lemma~\ref{lem:initialRatesLagr}, Lemma~\ref{lem:boundHrUrgr} and Lemma~\ref{lem:rateux} imply that there exist constants $\hat C$ and $C^{\star}$, only dependent on $T$, $\|\alpha '\|_\infty$, $\|\bar u\|_\infty$,  $\bar G_\infty$, $\beta$, and $|\bar u_x|_{2,\beta}$, such that 
\begin{align*}
\sup_{t\in [0,T]} \|U(t) - \hat U_{\Dx}(t)\|_{\infty}& \leq \hat C \big(\Dx^{\nicefrac{1}{2}}+ \|\delta_{2\Dx} \bar u_x^2\|_1 + \|\bar u_x- \bar u_{\Dx,x}\|_2 \big)\\[-0.9ex] 
& \leq  C^{\star} (\Dx^{\nicefrac{1}{2}}+ \Dx^{\nicefrac{\beta}{2}}).
\end{align*}

To finish the proof,  observe that 
\begin{equation*}
\sup_{t\in [0,T]} \|y(t) - \hat y_{\Dx}(t)\|_{\infty}\leq \|\bar{y} - \bar{y}_{\Dx}\|_{\infty}+ T \sup_{t\in [0,T]} \|U(t) - \hat U_{\Dx}(t)\|_{\infty},
\end{equation*}
which in turn, by Lemma~\ref{lem:initialRatesLagr}, implies 
\begin{equation*}
\sup_{t\in [0,T]} \|y(t) - \hat y_{\Dx}(t)\|_{\infty}\leq (2+C^{\star}T) (\Dx^{1/2}+ \Dx^{\beta/2}). \qedhere
\end{equation*}	
\end{proof}

It has been vital in the above proof to have a convergence rate for $\|\bar u_x-\bar u_{\Dx,x}\|_2$, but also that $\bar u_{x}\in B_2^\beta$. As already emphasized in Section~\ref{sec:InitialRatesEuler}, this is natural, since $\bar u_x$ contains all the information about future wave breaking, i.e., wave breaking occurring at $t>0$ for the exact solution, cf. \eqref{eq:setsux}. The same role is played by $\bar u_{\Dx,x}$ for the numerical solution. Thus, to prevent the energy discrepancy between the exact and the numerical solution from becoming too large within finite time, one needs to control $\|\bar u_x-\bar u_{\Dx,x}\|_2$ and $\|\delta_{2\Dx} \bar u_x^2\|_1$, which is possible when $\bar u_x\in B_2^\beta$. 

Our main theorem is now an immediate consequence of Theorem~\ref{thm:ConvRateLagr} and \eqref{eq:est_uNoLocal}. 

\begin{theorem}\label{thm:ConvRateEuler}
Suppose $(\bar u, \bar \mu, \bar \nu)\in \D_0^\alpha$ with $\bar u_x\in B_2^\beta$ for some $\beta\in (0,1]$. Let $(u,\mu,\nu)(t)=T_t((\bar u, \bar\mu, \bar \nu))$ and $(u_{\Dx}, \mu_{\Dx}, \nu_{\Dx})(t)=T_{\Dx,t}\circ P_{\Dx}((\bar u,\bar \mu,\bar \nu))$, then there exists a constant $D$, dependent on $T$, $\|\alpha '\|_\infty$, $\|\bar u\|_\infty$,  $\bar G_\infty$, $\beta$, and $|\bar u_x|_{2,\beta}$, such that 
\begin{equation*}
\sup_{t\in [0,T]} \|u(t) -u_{\Dx}(t)\|_{\infty}\leq D(\Dx^{\nicefrac{1}{8}}+ \Dx^{\nicefrac{\beta}{4}}).
\end{equation*}
\end{theorem}

\begin{remark}
When $\bar{u}_x \in B_2^{\beta}$ for $\frac{1}{2} <\beta \leq 1$, the local error dominates the error in \eqref{eq:est_uNoLocal}, cf. \eqref{eq:localErrLagr} and Theorem~\ref{thm:ConvRateLagr}. On the other hand, we do not expect that one can achieve anything better than $\mathcal{O}(\Dx^{\nicefrac{1}{4}})$ for the projection error by our approach, since $B_\infty^\beta(L^2(\R))$, which satisfies $B_\infty^\beta(L^2(\R))\subset B_2^{\beta}$, for $\beta>1$, only consists of constant functions, see \cite[Sec. 3]{ConstructiveApprox}. 

In the case of {\em multipeakons}, i.e., piecewise linear solutions, which are of particular importance as their piecewise linear structure is naturally preserved by the HS equation, it holds that $\bar{u}_x \in B_2^{\nicefrac{1}{2}}$, see \cite[Ex. 6.1]{AlphaRate}. Hence, for such data, there is a balance between the projection error and the local error.  
\end{remark}

\begin{remark}\label{rem:convrate}
	In the particular case of $\alpha \in [0, 1)$, the numerical solution $X_{\Dx}(t)$ coincides with $\hat X_{\Dx}(t)$ and hence the local error vanishes, which corresponds to setting $C=0$ in \eqref{eq:localErrLagr}. Thus, \eqref{eq:est_uNoLocal} and Theorem~\ref{thm:ConvRateLagr}, imply
\begin{equation}\label{rem:ConvRateEuler}
		\sup_{t \in [0,T ]} \|u(t) - u_{\Dx}(t)\|_{\infty} \leq \bar{D} \Dx^{\nicefrac{\beta}{4}}, 
	\end{equation}
	for a constant $\bar{D}$ dependent on $T$,  $\bar{G}_{\infty}$, $\|\bar u\|_\infty $, $\beta$, and $|\bar{u}_x|_{2, \beta}$. 
\end{remark}

\begin{remark}
	The concept of $\alpha$-dissipative solutions was introduced in \cite{AlphaHS} for $\alpha \in W^{1, \infty}(\R, [0, 1))$. The value $1$ is not permitted, because in that case the semigroup property of the solution operator $S_t$, i.e, $S_{s+t}(X)= S_t\circ S_s(X)$ for $s, t\geq 0$, might not hold, see \cite[Ex. 2.4.4]{PhdThesisNordli}. However, our numerical method and hence our proof, do not rely on the semigroup property and therefore Theorem~\ref{thm:ConvRateEuler} still holds if we extend our class of solutions to $\alpha \in W^{1, \infty}(\R, [0, 1])$, and \eqref{rem:ConvRateEuler} holds for $\alpha=1$. 
\end{remark}

\subsection{An alternative mapping when $\bar \nu_{\mathrm{sc}} = 0$}\label{sec:AlternativeMap}
An alternative way of deriving a convergence rate is presented in \cite{AlphaRate} for the case $\alpha\in [0,1]$ and $\bar \nu_{\mathrm{sc}}=0$. Instead of introducing the two mappings $\phi$ and $\psi$, which are Lipschitz continuous and satisfy $\bar y(\phi(r))=\bar y_{\Dx}(\psi(r))$, the author of \cite{AlphaRate} constructs a function $f$, defined in \cite[(4.24)]{AlphaRate}, which is discontinuous as it maps intervals along which $\bar y$ is constant to intervals where $\bar y_{\Dx}$ is constant and $\bar y(\xi)\not = \bar y_{\Dx}(f(\xi))$ in general, see Figure~\ref{fig:comparisonMaps} for one particular example. Furthermore, since both $\mathcal{S}$ and $\mathcal{S}_{\Dx}$, given by \eqref{eq:S} and \eqref{eq:Sdx}, respectively, can be written as an at most countable union of disjoint closed intervals whenever $\bar \nu_{\mathrm{sc}} = 0$, this ensures that $f(\mathcal{S}) = \mathcal{S}_{\Dx}\setminus E$, where $E \subset \R$ is a set of measure zero. As a consequence, one infers that 
\begin{align*}
	V_{\xi}(t, \xi) - V_{\Dx, \xi}(t, f(\xi))=0 \hspace{0.45cm} \text{for a.e. } \xi \in \mathcal{S},
\end{align*}
which allows one to derive an analogue of Theorem~\ref{thm:Aprior1}, cf. the proof of \cite[Thm. 4.6]{AlphaRate} and subsequently a convergence rate, see \cite[Thm. 4.11]{AlphaRate}. However, this alternative construction is limited to the case $\bar\nu_{\mathrm{sc}} = 0$, since one will otherwise encounter issues with the measurability of, e.g., $\bar{V}_{\Dx, \xi}(f(\xi))$. On the other hand, a major benefit of the mapping $f$ is the fact that it preserves the piecewise linear structure of $\bar{y}_{\Dx}$ and that $\bar{V}_{\Dx, \xi}(f)$ is piecewise constant. This is not necessarily true for the mappings $\phi$ and $\psi$ if $\bar{y}$ is not piecewise linear. 

\begin{figure}
	\includegraphics[width=0.85\linewidth]{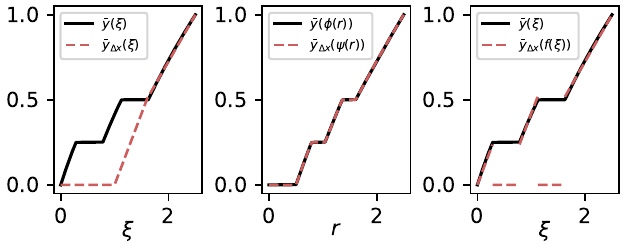}
	\captionsetup{width=.95\linewidth}
	\vspace{-0.25cm}
	\caption{The very left plot depicts $\bar y$ and $\bar y_{\Dx}$, while $\bar y(\phi)$ and $\bar y_{\Dx}(\psi)$, with $\phi$ and $\psi$ from Definition~\ref{def:maps}, are shown in the middle, whereas $\bar y$ and $\bar y_{\Dx}(f)$, with $f$ defined in \cite[Sec. 4]{AlphaRate}, are displayed in the rightmost plot.}
	\label{fig:comparisonMaps}
\end{figure}

Moreover, comparing the convergence rate from  \cite[Thm. 4.11]{AlphaRate} with that in Remark~\ref{rem:convrate}, or more precisely \eqref{rem:ConvRateEuler}, reveals that there is one additional benefit of using the mappings $\phi$ and $\psi$ instead of $f$ -- one obtains an improved convergence rate. Instead of having the $\mathcal{O}(\Dx^{\nicefrac{\beta}{8}})$-error estimate from \cite[Thm. 4.11]{AlphaRate}, we now obtain the improved $\mathcal{O}(\Dx^{\nicefrac{\beta}{4}})$-rate in \eqref{rem:ConvRateEuler}. The reason is that we now have, thanks to Lemma~\ref{lem:rateux} and Lemma~\ref{lem:boundHrUrgr}, $\|\bar{\mathcal{U}}_r - \bar{\mathcal{U}}_{\Dx, r}\|_{L^2(\mathcal{B}^c)} \leq \mathcal{O}(\Dx^{\nicefrac{\beta}{2}})$, as opposed to $\|\bar{U}_\xi - \bar{U}_{\Dx, \xi}(f)\|_{L^2(\mathcal{S}^c)} \leq \mathcal{O}(\Dx^{\nicefrac{\beta}{4}})$ in \cite[Lem. 4.10]{AlphaRate}.

\section{A numerical study of convergence rates}\label{sec:NumericalExp}
To numerically validate the theoretical convergence rate from Theorem~\ref{thm:ConvRateEuler}, we investigate three different examples: two multipeakon examples, i.e., piecewise linear solutions, and the case of a cusped initial wave profile. These examples highlight distinct challenges for the numerical method. 

A characteristic feature of the multipeakon examples is that a finite amount of energy concentrates at each of the finitely many wave breaking occurrences, leaving the cumulative energy, $F(t)$, discontinuous in space. For the cusp example, on the other hand, wave breaking happens, for each $t \in [0, 3]$, in the form of infinitesimal energy concentrations, such that $F(t)$ is absolutely continuous for all $t\geq 0$. 

We present plots showing the time evolution of the wave profile and the cumulative energy for each of the examples. Since the exact solution is unavailable for the cusped initial data, Example~\ref{ex:plainCusp}, we instead compute a numerical solution with $\Dx = 10^{-5}$ and use that as a reference solution, denoted by $(u_{\mathrm{ref}}, F_{\mathrm{ref}})$. Furthermore, we also compute a sequence of numerical solutions, $\{(u_{\Dx_k}, F_{\Dx_k}) \}_{k}$, for each example,  where 
\begin{equation}\label{eq:meshFam}
	\Dx_k = 4^{-k}, \hspace{0.35cm} k \geq 1, 
\end{equation}
and, based on that sequence, we compute the relative errors 
\begin{equation}\label{eq:relativeErrs}
	\mathrm{Err}_k(T) := \sup_{t \in [0, T]} \frac{\|u(t) - u_{\Dx_k}(t)\|_{\infty}}{\|u(t)\|_{\infty}},
\end{equation}
with $u(t)$ being replaced by $u_{\mathrm{ref}}(t)$ for the cusp example. These errors are graphically visualized along with an optimally fitted least square (LS) regression line, which represents the numerically computed convergence order. Additionally, we also include tables that display the corresponding experimental orders of convergence (EOCs), computed via 
\begin{equation*}
	\mathrm{EOC}_k(T) := \frac{\ln\big(\nicefrac{\mathrm{Err}_{k-1}(T)}{\mathrm{Err}_k(T)}\big)}{\ln\big(\nicefrac{\Dx_{k-1}}{\Dx_k})}. 
\end{equation*}

\begin{example}\label{ex:Multipeakon}
Let us start by investigating the Cauchy problem with the following initial data
\begin{align*}
	\bar{u}(x) &= \begin{cases}
		3, & x \leq 0, \\
		3-x, & 0 < x \leq 1, \\
		4-2x, & 1 < x \leq 2, \\
		0, & 2 < x, 
		\end{cases}  \\
	d\bar{\mu} &= d\bar{\nu} = d\delta_0 + \bar{u}_{x}^2dx, 
\end{align*}
and any $\alpha \in W^{1, \infty}(\R, [0, 1))$ satisfying $\alpha(\frac{11}{4}) = \frac{3}{4}$ and $\alpha(\frac{35}{8}) =\tfrac{9}{10}$. The exact solution experiences wave breaking at $(t, x) = (0, 0)$ and along all the characteristics starting in $(1, 2]$ and $(0, 1]$  at $\tau_1 =1$ and $\tau_2=2$, respectively. Since the projected initial data, $\bar{u}_{\Dx}$, coincides with $\bar u$ for \eqref{eq:meshFam},
 the numerical method performs extremely well as suggested by Figure~\ref{fig:comparisonMult}, where the exact solution $(u, F)(t)$, computed in Appendix~\ref{sec:Multipeakon}, is compared to two numerical solutions $(u_{\Dx_j}, F_{\Dx_j})(t)$ for $j \in \{c, f \}$ with $\Dx_c = 4^{-2}$ and $\Dx_f=4^{-4}$ at the times $t=0, 1$, and $2$. In fact, computing the relative errors $\mathrm{Err}_k(T)$, given by \eqref{eq:relativeErrs}, reveals that the numerical solution, computed with $\alpha_1$ below, coincides with the exact solution, up to machine precision, for $\Dx = \frac{1}{4}$ when $T < 1$ and for all $T \geq 0$ when $\Dx = 4^{-2}$. Hence, round-off errors become more and more dominant as we refine the mesh, and the corresponding EOCs are therefore not included for this example. 
 
 \begin{figure}
	\includegraphics{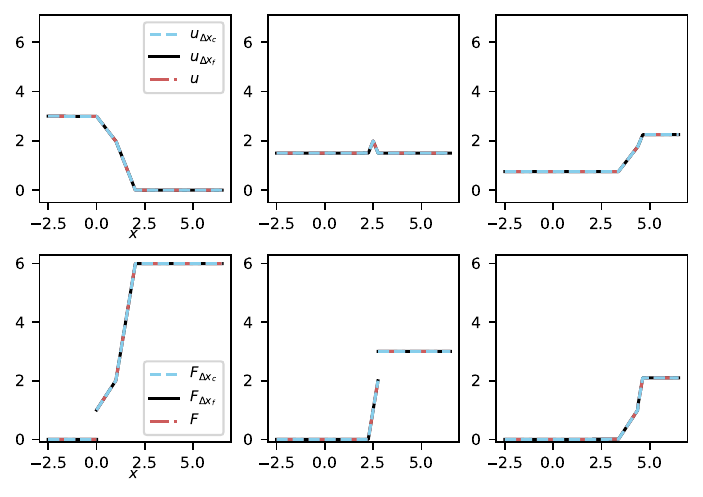}
	\captionsetup{width=.975\linewidth}
	\vspace{-0.4cm}
	\caption{
	A comparison of $u$ (top row) and $F$ (bottom row), both with dotted red lines, to that of $u_{\Dx_j}$ (top row) and $F_{\Dx_j}$ (bottom row) for $\Dx_c = 4^{-2}$ (dashed blue) and $\Dx_f = 4^{-4}$ (solid black) for Example~\ref{ex:Multipeakon}. The solutions are compared from left to right at $t=0$, $1$, and $2$.}
	\label{fig:comparisonMult}
\end{figure}

In addition, we compute numerical $\alpha$-dissipative solutions with two different $\alpha\in W^{1, \infty}(\R, [0,1))$ satisfying $\alpha(\frac{11}{4}) = \frac{3}{4}$ and $\alpha(\frac{35}{8}) =\tfrac{9}{10}$. In particular, we use 
\begin{align}\label{eq:alpha1}
	\alpha_1(x) &= \begin{cases}
		\frac{3}{11}x, & 0 \leq x \leq \frac{11}{4}, \\
		\frac{6}{65}x + \frac{129}{260}, & \frac{11}{4} < x \leq \frac{35}{8}, \\
		\frac{9}{10}, & \frac{35}{8} < x, \end{cases} 
\end{align}
and  
\begin{align}\label{eq:alpha2}
	\alpha_2(x) &= \begin{cases}
	-1 + e^{\frac{4}{11}\ln(\frac{7}{4})x}, & 0 \leq x \leq \frac{11}{4}, \\
	\frac{48}{65}x^2 -\frac{336}{65}x +\frac{2439}{260}, &\frac{11}{4} \leq x \leq \frac{35}{8}, \\ 
	\frac{9}{10} & \frac{35}{8} < x,
	\end{cases}
\end{align}
which are compared in Figure~\ref{fig:differentAlphas} and which both lead to the same exact $\alpha$-dissipative solution. Yet, a notable difference is that $\|\alpha_1'\|_{\infty} = \frac{3}{11} \approx 0.273$, whereas $\|\alpha_2'\|_{\infty} = \frac{84}{65} \approx 1.292$. In spite of this difference, the computed relative errors are identical, for any $k\geq 1$, for these two numerical solutions, indicating that the computed family, $\{u_{\Dx_k}\}_{k}$, is indifferent with respect to $\|\alpha'\|_{\infty}$ for this example.

In general, one cannot expect to achieve close to machine precision for $\Dx \sim 0.1$, which corresponds to $k=1$ or $k=2$ in \eqref{eq:meshFam}. If we for instance had chosen a different family of mesh parameters, say $\widetilde{\Dx}_k = 3^{-k}$, or if the initial wave breaking had occurred at $x \notin \{x_{2j}\}_{j \in \mathbb{Z}}$, then the errors would become much worse. 

\begin{figure}
	\makebox[\linewidth]{
    \centering
    \begin{subfigure}[t]{2.55in}
    \includegraphics[scale=0.85]{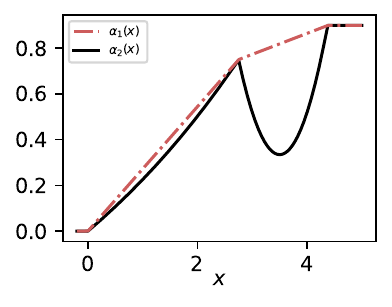}
    \vspace{-0.25cm}
    \captionsetup{width=.925\linewidth}
    \caption{A visualization of $\alpha_1$ and $\alpha_2$ from \eqref{eq:alpha1} and \eqref{eq:alpha2}, respectively, in Example~\ref{ex:Multipeakon}. Note that $\alpha_1(\nicefrac{11}{4}) = \alpha_2(\nicefrac{11}{4})$ and $\alpha_1(\nicefrac{35}{8}) = \alpha_2(\nicefrac{35}{8})$ such that they lead to the same $\alpha$-dissipative solution. }
    \label{fig:differentAlphas}
    \end{subfigure}
    ~
    \begin{subfigure}[t]{2.25in}
    \includegraphics[scale=0.85]{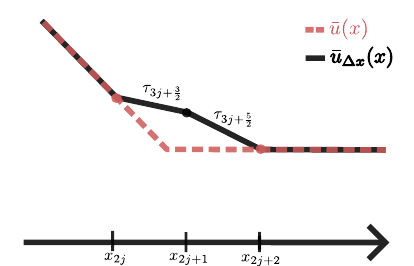}
    \vspace{-0.65cm}
    \captionsetup{width=.95\linewidth}
    \caption{The mesh does not perfectly align with the nodes of the multipeakon in Example~\ref{ex:complexMultipeakon}, such that the projected wave profile $\bar{u}_{\Dx}$ may contain some linear segments with different breaking times.}
    \label{fig:mismatch}
    \end{subfigure}
    }
\end{figure}

\end{example}

The next example illustrates that if the nodes of the initial wave profile do not perfectly align with the mesh, then the errors become much worse.

\begin{example}\label{ex:complexMultipeakon}
Next, let us study the Cauchy problem for the multipeakon initial data from \cite[Ex. 5.1]{Alpha2}, which is given by  
\begin{align*}
	\bar{u}(x) &= \begin{cases}
		3, & x \leq 0, \\
		3-x & 0 < x \leq 1, \\
		2, & 1 < x \leq \frac{400}{361}, \\
		\frac{58}{19} - \frac{19}{20}x, & \frac{400}{361} \leq x \leq \frac{800}{361}, \\
		 \frac{18}{19}, & \frac{800}{361} < x \leq \frac{200}{81}, \\
		\frac{542}{171} - \frac{9}{10}x, &  \frac{200}{81} < x \leq \frac{100}{27}, \\
		-\frac{28}{171}, & \frac{100}{27} < x, 
		\end{cases}  \\
	d\bar{\mu} &= d\bar{\nu} = \bar{u}_{x}^2dx, \\
	\alpha(x) &= \begin{cases}
		0, & x \leq \frac{1434}{361}, \\
		-\frac{478}{127} + \frac{361}{381}x, & \frac{1434}{361} < x \leq \frac{6879}{1444}, \\
		\frac{1}{1705}(361x -441), &\frac{6879}{1444} < x \leq 5, \\
		\frac{4}{5}, &5 < x. 
		\end{cases}
\end{align*}

The exact solution $(u, F)(t)$ experiences wave breaking at $\tau_1= 2$, $\tau_2 =\tfrac{40}{19}$, and $\tau_3=\tfrac{20}{9}$, and, as in the previous example, $F(\tau_j)$ has a jump discontinuity at the spatial point  where wave breaking occurs, see \cite[App. A]{Alpha2} for further details. Yet, despite its resemblance to the previous example, the computed relative errors in Table~\ref{tab:EOCComplexMult} are much larger, due to the following two reasons. First and foremost, there is a mismatch between the gridpoints and the nodes of the multipeakon, such that the projected data $\bar{u}_{\Dx}$ contains a small number of linear segments with slopes that differ from those of the exact data, which is illustrated in Figure~\ref{fig:mismatch}. As a consequence, there are a few additional numerical breaking times, given by \eqref{eq:numBreaking}, which do not coincide with $\tau_j$  for $j=1$, $2$, and $3$. However, there are even-numbered gridpoints situated at $x=0$ and $x=1$ when $\Dx$ is given by \eqref{eq:meshFam}, hence only the $4$ rightmost nodes can give rise to additional numerical breaking times in this case. In addition, also the local errors discussed in Section~\ref{sec:FirstAttempt} contribute to the relative errors in Table~\ref{tab:EOCComplexMult}. In particular, as observed in \cite[Ex. 5.1]{Alpha2}, if 
\begin{equation*}
	\Dx \geq 1.75\! \cdot \! 10^{-2}
\end{equation*}
then the extracted sequence $\{\tau_k^*\}_{k}$ does not contain the breaking time $\tau_2=\tfrac{40}{19}$, such that one only approximates the amount of energy to be removed at this time, cf. \eqref{eq:iterationSeq}--\eqref{eq:beta3}. However, due to the mismatch between the gridpoints and the $4$ rightmost nodes, there might still be local errors even for $\Dx \leq 1.75\! \cdot \! 10^{-2}$. This is confirmed by Figure~\ref{fig:complexMult}, where the exact solution $(u, F)(t)$ is compared to two approximations, $(u_{\Dx_j}, F_{\Dx_j})(t)$ for $j \in \{c, d\}$ with $\Dx_c = 4^{-1}$ and $\Dx_f =4^{-4}$ at the times $t=0$, $\tfrac{40}{19}$, and $3$.  The finer approximation, $(u_{\Dx_f}, F_{\Dx_f})(t)$, computed with $\Dx_f < 1.75\! \cdot \! 10^{-2}$, agrees well with the exact solution, except that we observe a large energy discrepancy at $t=\tfrac{40}{19}$. This is due to the aforementioned mismatch, which implies that the wave breaking $\tau_2=\tfrac{40}{19}$ is not present in the list $\{\tau_k^*\}$ and this breaking time is therefore delayed according to the iteration scheme, given by \eqref{eq:iterationSeq}--\eqref{eq:beta3}. 

Nevertheless, as $\bar{u}_{x}$ belongs to $BV(\R) \cap L^1(\R)\subset B_2^{\nicefrac{1}{2}} $, cf. Remark~\ref{remark:beta}, Theorem~\ref{thm:ConvRateEuler} ensures that
\begin{equation}\label{eq:multipeakonRate}
	\sup_{t \in [0, T]} \|u(t) - u_{\Dx}(t)\|_{\infty} \leq \mathcal{O}(\Dx^{\nicefrac{1}{8}}).
\end{equation} 
The computed EOCs in Table~\ref{tab:EOCComplexMult} and  the LS convergence order in Figure~\ref{fig:errorEx2}, suggest that this error estimate is rather pessimistic. On the other hand, we do expect the numerical method to perform well on multipeakon data, since the projection operator $P_{\Dx}$ from Definition~\ref{def:ProjOperator} is piecewise linear. 

\begin{figure}
	\includegraphics{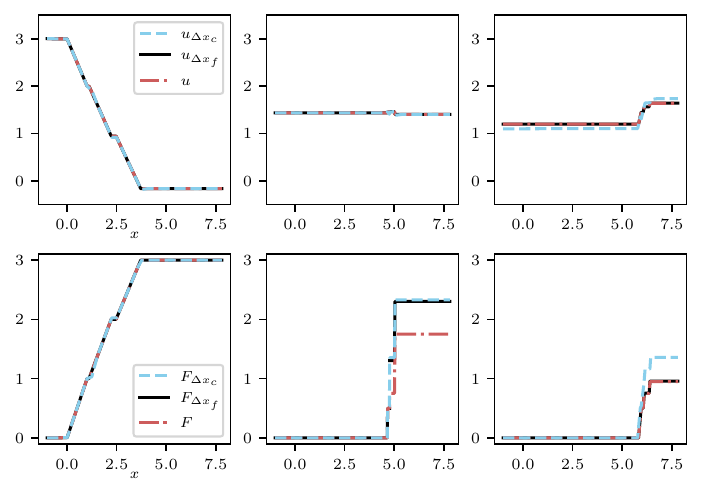}
	\captionsetup{width=.975\linewidth}
	\vspace{-0.4cm}
	\caption{A comparison of $u$ (top row) and $F$ (bottom row), both with dotted red lines, to that of $u_{\Dx_j}$ (top row) and $F_{\Dx_j}$ (bottom row) for $\Dx_c = 4^{-1}$ (dashed blue) and $\Dx_f = 4^{-4}$ (solid black ) for Example~\ref{ex:complexMultipeakon}. The solutions are compared from left to right at $t=0$, $\frac{40}{19}$, and $3$.}
	\label{fig:complexMult}
\end{figure}

\begin{table}[H]
 \setlength\tabcolsep{2.5pt}
\begin{tabular}{ c|ccccccc} 
 \hline
 $k$ & $1$ & $2$ & $3$ & $4$ & $5$ & $6$ & $7$  \\ 
 \hline
 $\mathrm{Err}_k(\nicefrac{3}{2})$ & $2.8 \!\cdot\! 10^{-2}$ & $7.5 \! \cdot \!10^{-3}$ & $1.5 \!\cdot \!10^{-3}$ & $4.6 \! \cdot \! 10^{-4}$ & $8.9 \! \cdot \! 10^{-5}$ & $2.3 \! \cdot \! 10^{-5}$ & $4.0 \!\cdot \!10^{-6}$\\
 $\mathrm{EOC}_k(\nicefrac{3}{2})$ & $-$ & $0.94$ & $1.14$ & $0.88$ & $1.18$ & $0.99$ & $1.25$ \\
 \hline 
 $\mathrm{Err}_k(3)$ & $8.6\! \cdot\! 10^{-2}$ & $1.8 \! \cdot\! 10^{-2}$ & $7.7 \! \cdot 10^{-3}$ & $1.2\! \cdot \! 10^{-3}$ & $4.6 \! \cdot \! 10^{-4}$ & $7.9 \! \cdot\! 10^{-5}$ & $2.3 \! \cdot \! 10^{-5}$ \\
 $\mathrm{EOC}_k(3)$ & $-$ & $1.13$ & $0.61$ & $1.33$ & $0.70$ & $1.26$ & $0.91$ 
 \end{tabular}
 \vspace{-0.35cm}
 \captionsetup{width=.975\linewidth}
\caption{Computed relative errors and experimental orders of convergence for  Example~\ref{ex:complexMultipeakon} with $T=\nicefrac{3}{2}$ (before any wave breaking) and with $T=3$.}
\label{tab:EOCComplexMult}
\end{table}
 \end{example}
 
 \begin{figure}
 		\makebox[\linewidth]{
    \centering
    \begin{subfigure}[t]{2.3in}
    \includegraphics[scale=0.85]{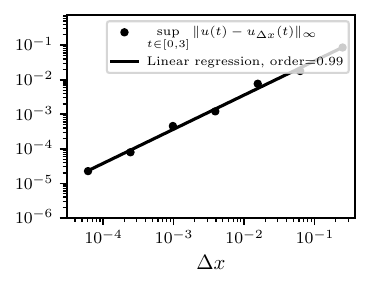}
    \vspace{-0.25cm}
    \captionsetup{width=.925\linewidth}
    \caption{The errors $\displaystyle \sup_{t \in [0, 3]} \!\|u(t) - u_{\Dx}(t)\|_{\infty}$ plotted as a function of the mesh parameter $\Dx$ for Example~\ref{ex:complexMultipeakon}. The LS regression line indicates that the numerical convergence order is about $1$. }
    \label{fig:errorEx2}
    \end{subfigure}
    ~
     \begin{subfigure}[t]{2.2in}
    \includegraphics[scale=0.85]{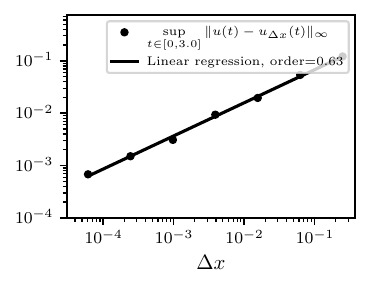}
    \vspace{-0.25cm}
       \captionsetup{width=.995\linewidth}
    \caption{The errors $\displaystyle \sup_{t \in [0, 3]}\! \|u(t) \!- \!u_{\Dx}(t)\|_{\infty}$ from Example~\ref{ex:plainCusp} plotted against the mesh size $\Dx$. The LS slope is approximately equal to $\tfrac{3}{5}$.}
    \label{fig:errorEx3}
    \end{subfigure}
    }
\end{figure}

\begin{example}\label{ex:plainCusp}
The solution to \eqref{eq:HS} associated with the cusp initial data
\begin{align*}
	\bar{u}(x) &= \begin{cases}
	1, &x < -1, \\
	\left|x\right|^{\nicefrac{2}{3}}\!, & -1\leq x \leq1, \\
	1, & 1 < x,
	\end{cases} \\ 
	\bar{F}(x) &= \begin{cases}
	0, & x < -1, \\
	\frac{4}{3} \left(1 + \mathrm{sgn}(x)|x|^{\nicefrac{1}{3}}\right)\!, & -1 \leq x \leq 1, \\
	\frac{8}{3}, & 1 < x, \end{cases} \\ 
	\alpha(x) &= \begin{cases}
	b, & x < -1, \\
	b|x|, & -1 \leq x <0, \\
	0, & 0 \leq x, \end{cases}
\end{align*}
where $b \in [0, 1)$,  experiences wave breaking along a single characteristics for each $t \in [0, 3]$. The exact solution is unavailable, and we therefore compute a numerical solution with $\Dx = 10^{-5}$, denoted  $(u_{\mathrm{ref}}, F_{\mathrm{ref}})(t)$, and use that as a reference solution. The time evolution of this solution is compared to $(u_{\Dx_j}, F_{\Dx_j})$ for $j \in \{c, f\}$ with $\Dx_c$ and $\Dx_f$, as in the previous example, at the times $t=0, \frac{3}{2}$, and $3$ in Figure~\ref{fig:comparisonCusp}. 

Note that $\bar{u}_x$ is less regular than in the previous two examples, especially noteworthy is the fact that $\bar{u}_x \notin L^{\infty}(\R)$. In spite of that, the proof in \cite[Ex. 6.3]{AlphaRate} reveals that $\bar{u}_{x}\! \! \in B_2^{\nicefrac{1}{6}}$ and Theorem~\ref{thm:ConvRateEuler} therefore implies 
\begin{equation*}
	\sup_{t \in [0, T]} \|u(t) - u_{\Dx}(t)\|_{\infty} \leq \mathcal{O}(\Dx^{\nicefrac{1}{24}}).
\end{equation*}
This is 3 times worse than the corresponding error estimate we have for multipeakons, cf.  \eqref{eq:multipeakonRate}. Comparing the computed EOCs in Table~\ref{tab:EOCCusp} for the cusp initial data to the corresponding ones in Table~\ref{tab:EOCComplexMult} for the multipeakon from Example~\ref{ex:complexMultipeakon} suggests that also the numerical convergence order is lower for this example, indicating that the regularity of $\bar{u}_x$ influences the numerical accuracy. This is also confirmed by the slopes of the two LS regression lines in Figures~\ref{fig:errorEx2}--\ref{fig:errorEx3}.

\begin{figure}
	\includegraphics{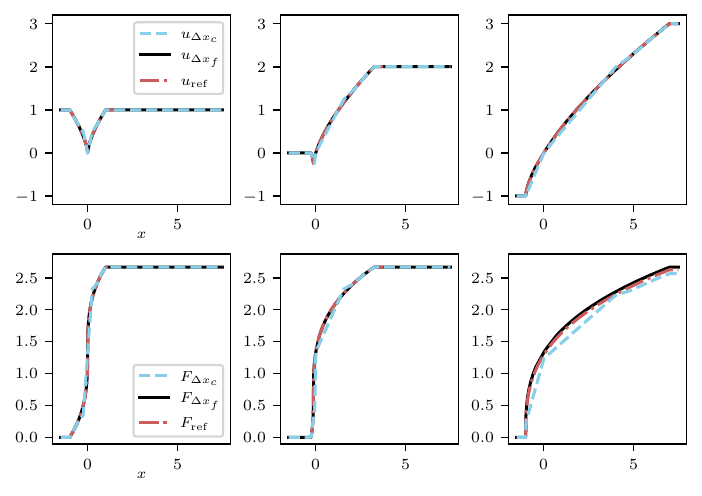}
	\captionsetup{width=.975\linewidth}
	\vspace{-0.4cm}
	\caption{A comparison of $u_{\mathrm{ref}}$ (top row) and $F_{\mathrm{ref}}$ (bottom row), both with dotted red lines, to that of $u_{\Dx_j}$ (top row) and $F_{\Dx_j}$ (bottom row) for $\Dx_c = 4^{-1}$ (dashed blue) and $\Dx_f = 4^{-4}$ (solid black) for Example~\ref{ex:plainCusp}. The solutions are compared from left to right at $t=0$, $\tfrac{3}{2}$, and $3$ with $b=\nicefrac{19}{20}$.} 
	\label{fig:comparisonCusp}
\end{figure}

\begin{table}[H]
 \setlength\tabcolsep{4.0pt}
\begin{tabular}{ c|ccccccc} 
 \hline
 $k$ & $1$ & $2$ & $3$ & $4$ & $5$ & $6$ & $7$ \\ 
 \hline
  $\mathrm{Err}_k(3)$ & $0.12$ & $5.4\! \cdot \!10^{-2}$ & $2.0\! \cdot \! 10^{-2}$ & $9.4\! \cdot\! 10^{-3}$ & $3.1\! \cdot \! 10^{-3}$ & $1.5\! \cdot \! 10^{-3}$ & $6.8\! \cdot\! 10^{-4}$ \\
 $\mathrm{EOC}_k(3)$ & $-$ & $0.587$ & $0.730$ & $0.531$ & $0.800$ & $0.519$ & $0.573$ 
\end{tabular}
\captionsetup{width=.975\linewidth}
\vspace{-0.35cm}
\caption{Computed experimental convergence orders for Example~\ref{ex:plainCusp} with $T=3$ (after all wave breaking occurrences) for $b = \nicefrac{19}{20}$.}
\label{tab:EOCCusp}
\end{table}
\end{example}

\bibliographystyle{plain}

\appendix

\section{Exact solution to the multipeakon example}\label{sec:Multipeakon}
Consider the initial data from Example~\ref{ex:Multipeakon} with any $\alpha \in W^{1, \infty}(\R, [0, 1))$ satisfying $\alpha(\tfrac{11}{4})=\tfrac{3}{4}$ and $\alpha(\tfrac{35}{8})=\tfrac{9}{10}$. By applying the mapping $L$ from Definition \ref{def:MapL}, one obtains
\begin{align*}
	\bar{y}(\xi) &= \begin{cases}
		\xi, & \xi \leq 0, \\
		0, & 0 < \xi \leq 1, \\
		\frac{1}{2}(\xi - 1), & 1 < \xi \leq 3, \\
		\frac{1}{5}(\xi + 2), & 3 < \xi \leq 8, \\
		\xi - 6, & 8 < \xi,
		\end{cases} \\
	\bar{U}(\xi) &= \begin{cases}
		3, & \xi \leq 1, \\
		\frac{1}{2}(-\xi+7), & 1 < \xi \leq 3, \\
		\frac{1}{5}(- 2\xi + 16), & 3 < \xi \leq8, \\
		0, & 8 < \xi,
		\end{cases} \\ 
	\bar{V}(\xi) = \bar{H}(\xi) &= \begin{cases}
		0, & \xi \leq 0, \\
		\xi, & 0 < \xi \leq 1, \\
		\frac{1}{2}(\xi+1), & 1 < \xi \leq 3, \\
		\frac{1}{5}(4\xi - 2), & 3 < \xi \leq 8, \\
		6, & 8 < \xi,
	\end{cases}
\end{align*}
which inserted into \eqref{eq:tau} yields 
\begin{align*}
	\tau(\xi) &= \begin{cases}
	0, & 0 <\xi < 1, \\
	2, & 1 < \xi <3, \\
	1, & 3 < \xi  < 8, \\
	\infty, & \text{otherwise}. 
	\end{cases}
\end{align*}
After solving \eqref{eq:LagrSystem} for $t \in [0, 1)$ with initial data $(\bar{y}, \bar{U}, \bar{V}, \bar{H})$, we find
\begin{align}\label{ex:before}
	y(t, \xi) &= \begin{cases}
		\xi - \frac{3}{4}t^2 +3t, & \xi < 0, \\
		\frac{1}{4}\xi t^2 - \frac{3}{4}t^2 + 3t, & 0 \leq \xi < 1, \\
		\frac{1}{8}(t-2)^2\xi - \frac{5}{8}t^2 + \frac{7}{2}t - \frac{1}{2}, & 1 \leq \xi < 3, \\
		\frac{1}{5}(t-1)^2\xi - \frac{17}{20}t^2 + \frac{16}{5}t + \frac{2}{5}, & 3 \leq \xi < 8, \\
		\xi + \frac{3}{4}t^2 - 6, & 8 \leq \xi,  
	\end{cases} \nonumber \\ 
	U(t, \xi) &= \begin{cases}
		-\frac{3}{2}t + 3, & \xi < 0, \\
		\frac{1}{2}\xi t -\frac{3}{2}t + 3, & 0 \leq \xi < 1, \\
		\frac{1}{4}(t-2)\xi - \frac{5}{4}t + \frac{7}{2}, & 1 \leq \xi < 3, \\
		\frac{2}{5}(t-1)\xi - \frac{17}{10}t + \frac{16}{5}, & 3 \leq \xi < 8, \\
		\frac{3}{2}t, & 8 \leq \xi, 
		\end{cases} \nonumber \\ 
	V(t, \xi) &= \bar{V}(\xi). 
\end{align}
Note that $y(t, \cdot)$ is strictly increasing for $ t\in (0, 1)$, such that we can compute its inverse by solving $x=y(t, \xi)$ in terms of $\xi$. Definition~\ref{def:MapM} therefore implies for $t\in [0,1)$ 
\begin{align}\label{ex:eul1}
	u(t, x) &= \begin{cases}
		- \frac{3}{2}t + 3, & x \leq -\frac{3}{4}t^2 + 3t, \\
		\frac{2}{t}(x - \frac{3}{2}t), & -\frac{3}{4}t^2 + 3t < x \leq -\frac{1}{2}t^2 + 3t, \\
		\frac{2}{t-2}(x -\frac{1}{2}t - 3), & -\frac{1}{2}t^2 + 3t < x \leq -\frac{1}{4}t^2 + 2t +1, \\
		\frac{2}{t-1}(x-\frac{3}{4}t - 2), &  -\frac{1}{4}t^2 + 2t +1 < x \leq \frac{3}{4}t^2 + 2, \\ 
		\frac{3}{2}t, & \frac{3}{4}t^2 + 2<x, 
	\end{cases}   \nonumber \\
	F(t, x) &= \begin{cases}
		0, & x \leq -\frac{3}{4}t^2 + 3t, \\
		\frac{4}{t^2}(x + \frac{3}{4}t^2 - 3t), & -\frac{3}{4}t^2 + 3t < x \leq -\frac{1}{2}t^2 + 3t, \\
		\frac{4}{(t-2)^2}(x + \frac{3}{4}t^2 -4t +1), & -\frac{1}{2}t^2 + 3t < x \leq -\frac{1}{4}t^2 + 2t +1, \\
		\frac{4}{(t-1)^2}(x + \frac{3}{4}t^2 -3t - \frac{1}{2}), & -\frac{1}{4}t^2 + 2t +1 < x \leq \frac{3}{4}t^2 + 2, \\ 	
		6, & \frac{3}{4}t^2 + 2 < x. 
	\end{cases}
\end{align}
At $t=1$, wave breaking takes place for all $\xi \in (3, 8)$, and \eqref{ex:before} yields $y(1, \xi) = \frac{11}{4}$, such that $\alpha(y(1, \xi)) = \frac{3}{4}$ for all $\xi \in (3, 8)$. By computing $V(1, \xi)$ and thereafter solving \eqref{eq:LagrSystem}, one finds, for $t \in [1, 2)$, 
\begin{align*}
	y(t, \xi) &= \begin{cases}
		\xi - \frac{3}{8}t^2 + \frac{9}{4}t + \frac{3}{8}, & \xi \leq 0 \\ 
		\frac{1}{4}t^2\xi - \frac{3}{8}t^2 + \frac{9}{4}t + \frac{3}{8}, & 0 < \xi \leq 1,  \\
		\frac{1}{8}(t-2)^2\xi - \frac{1}{4}t^2 + \frac{11}{4}t - \frac{1}{8}, & 1 < \xi \leq 3, \\
		\frac{1}{20}(t-1)^2\xi - \frac{1}{40}t^2 + \frac{31}{20}t + \frac{49}{40}, & 3 < \xi \leq 8, \\
		\xi + \frac{3}{8}t^2 + \frac{3}{4}t - \frac{51}{8}, & 8 < \xi, 
	\end{cases} \\ 
	U(t, \xi) &= \begin{cases}
		-\frac{3}{4}t + \frac{9}{4}, & \xi \leq 0, \\
		\frac{1}{2}t\xi - \frac{3}{4}t + \frac{9}{4}, & 0 < \xi \leq 1, \\
		\frac{1}{4}(t-2)\xi - \frac{1}{2}t +\frac{11}{4}, & 1 < \xi \leq 3, \\
		\frac{1}{10}(t-1)\xi - \frac{1}{20}t + \frac{31}{20}, & 3 < \xi \leq 8, \\
		\frac{3}{4}t + \frac{3}{4}, & 8 < \xi, 
	\end{cases} \\ 
	V(t, \xi) &= \begin{cases}
		0, & \xi \leq 0, \\
		\xi, & 0 < \xi \leq 1, \\
		\frac{1}{2}(\xi + 1), & 1 < \xi \leq 3, \\
		\frac{1}{5}(\xi + 7), & 3 < \xi \leq 8, \\
		3, & 8 < \xi.
	\end{cases}
\end{align*}

Observe that $y(t, \cdot)$ is strictly increasing for all $t \in (1, 2)$ and hence invertible. Again, by computing its inverse and subsequently using Definition~\ref{def:MapM}, we have for $t\in [1,2)$
\begin{align}\label{ex:eul2}
	u(t, x) &= \begin{cases}
	-\frac{3}{4}t + \frac{9}{4}, & x \leq -\frac{3}{8}t^2 + \frac{9}{4}t + \frac{3}{8}, \\
	\frac{2}{t}(x - \frac{9}{8}t - \frac{3}{8}), & -\frac{3}{8}t^2 + \frac{9}{4}t + \frac{3}{8} < x \leq -\frac{1}{8}t^2 + \frac{9}{4}t + \frac{3}{8}, \\
	\frac{2}{(t-2)}(x - \frac{7}{8}t - \frac{21}{8}), &   -\frac{1}{8}t^2 + \frac{9}{4}t + \frac{3}{8} < x \leq \frac{1}{8}t^2 + \frac{5}{4}t + \frac{11}{8},  \\
	\frac{2}{(t-1)}(x - \frac{3}{4}t -2), &\frac{1}{8}t^2 + \frac{5}{4}t + \frac{11}{8} < x \leq \frac{3}{8}t^2 + \frac{3}{4}t + \frac{13}{8}, \\
	\frac{3}{4}t + \frac{3}{4}, &\frac{3}{8}t^2 + \frac{3}{4}t + \frac{13}{8} < x, 
	\end{cases} \nonumber  \\
	F(t, x) &= \begin{cases}
	0, & x \leq -\frac{3}{8}t^2 + \frac{9}{4}t + \frac{3}{8}, \\
	\frac{4}{t^2}(x + \frac{3}{8}t^2 - \frac{9}{4}t - \frac{3}{8}), & -\frac{3}{8}t^2 + \frac{9}{4}t + \frac{3}{8} < x \leq -\frac{1}{8}t^2 + \frac{9}{4}t + \frac{3}{8}, \\
	\frac{4}{(t-2)^2}(x + \frac{3}{8}t^2 - \frac{13}{4}t + \frac{5}{8}), & -\frac{1}{8}t^2 + \frac{9}{4}t + \frac{3}{8} < x \leq \frac{1}{8}t^2 + \frac{5}{4}t + \frac{11}{8},  \\
	\frac{4}{(t-1)^2}(x + \frac{3}{8}t^2 - \frac{9}{4}t - \frac{7}{8}), & \frac{1}{8}t^2 + \frac{5}{4}t + \frac{11}{8} < x \leq \frac{3}{8}t^2 + \frac{3}{4}t + \frac{13}{8}, \\
	3, &\frac{3}{8}t^2 + \frac{3}{4}t + \frac{13}{8} < x. 
	\end{cases}
\end{align}
The last wave breaking occurs at $(t, x) = (2, \frac{35}{8})$, and takes place for all $\xi \in (1, 3)$, for which $\alpha(y(2, \xi)) = \tfrac{9}{10}$. Thus, after computing $V(2, \xi)$ and then solving \eqref{eq:LagrSystem}, we obtain, for all $t \geq 2$, 
\begin{align*}
	y(t, \xi) &= \begin{cases}
		\xi - \frac{21}{80}t^2 + \frac{9}{5}t + \frac{33}{40}, & \xi \leq 0, \\
		\frac{1}{4}t^2\xi - \frac{21}{80}t^2 + \frac{9}{5}t + \frac{33}{40}, & 0 < \xi \leq1, \\
		\frac{1}{80}(t-2)^2\xi -\frac{1}{40}t^2 + \frac{37}{20}t +\frac{31}{40}, & 1 < \xi \leq 3, \\
		\frac{1}{20}(t-1)^2\xi - \frac{11}{80}t^2 + 2t + \frac{31}{40}, & 3 < \xi \leq 8, \\
		\xi + \frac{21}{80}t^2 + \frac{6}{5}t -\frac{273}{40}, & 8 < \xi, 
	\end{cases} \\ 
	U(t, \xi) &= \begin{cases}
		-\frac{21}{40}t + \frac{9}{5}, & \xi \leq 0, \\
		\frac{1}{2}t\xi - \frac{21}{40}t + \frac{9}{5}, & 0 < \xi \leq 1, \\
		\frac{1}{40}(t-2)\xi -\frac{1}{20}t +\frac{37}{20}, & 1 < \xi \leq 3, \\
		\frac{1}{10}(t-1)\xi - \frac{11}{40}t + 2, & 3 < \xi  \leq 8, \\
		\frac{21}{40}t + \frac{6}{5}, & 8 < \xi,  
	\end{cases} \\ 
	V(t, \xi) &=  \begin{cases}
		0, & \xi \leq 0, \\
		\xi , &0 < \xi \leq 1, \\
		\frac{1}{20}(\xi + 19), & 1 < \xi \leq 3, \\
		\frac{1}{5}\xi + \frac{1}{2}, & 3 < \xi \leq 8, \\
		\frac{21}{10}, & 8 < \xi. 
		\end{cases}
\end{align*}
As before, $y(t, \cdot)$ is strictly increasing for $t \in (2, \infty)$ and hence invertible. Computing its inverse and using Definition~\ref{def:MapM} yields for $t\geq 2$
\begin{align}\label{ex:eul3}
	u(t, x) &= \begin{cases}
		-\frac{21}{40}t + \frac{9}{5}, & x \leq - \frac{21}{80}t^2 + \frac{9}{5}t + \frac{33}{40}, \\
		\frac{2}{t}\big(x-\frac{9}{10}t - \frac{33}{40}\big), & - \frac{21}{80}t^2 + \frac{9}{5}t + \frac{33}{40} < x \leq -\frac{1}{80}t^2 +\frac{9}{5}t +\frac{33}{40}, \\ 
		\frac{2}{(t-2)}\big(x- \frac{7}{8}t -\frac{21}{8}\big), & -\frac{1}{80}t^2 +\frac{9}{5}t +\frac{33}{40} < x \leq \frac{1}{80}t^2 +\frac{17}{10}t + \frac{37}{40}, \\ 
		\frac{2}{(t-1)}\big(x - \frac{69}{80}t - \frac{71}{40}), &  \frac{1}{80}t^2 +\frac{17}{10}t + \frac{37}{40}< x \leq \frac{21}{80}t^2 + \frac{6}{5}t + \frac{47}{40}, \\ 
		\frac{21}{40}t + \frac{6}{5}, & \frac{21}{80}t^2 + \frac{6}{5}t + \frac{47}{40} < x,  
	\end{cases} \nonumber \\ 
	F(t, x) &= \begin{cases}
		0, & x \leq - \frac{21}{80}t^2 + \frac{9}{5}t + \frac{33}{40}, \\
		\frac{4}{t^2}( x+\frac{21}{80}t^2- \frac{9}{5}t - \frac{33}{40}), &  - \frac{21}{80}t^2 + \frac{9}{5}t + \frac{33}{40} < x \leq -\frac{1}{80}t^2 +\frac{9}{5}t +\frac{33}{40}, \\ 
		\frac{4}{(t-2)^2}(x +\frac{21}{80}t^2 -\frac{14}{5}t +\frac{7}{40}), & -\frac{1}{80}t^2 +\frac{9}{5}t +\frac{33}{40} < x \leq \frac{1}{80}t^2 +\frac{17}{10}t + \frac{37}{40}, \\ 
	\frac{4}{(t-1)^2}(x+\frac{21}{80}t^2 - \frac{9}{4}t - \frac{26}{40}), & \frac{1}{80}t^2 +\frac{17}{10}t + \frac{37}{40} < x \leq  \frac{21}{80}t^2 + \frac{6}{5}t + \frac{47}{40}, \\
	\frac{21}{10}, & \frac{21}{80}t^2 + \frac{6}{5}t + \frac{47}{40} < x.   
	\end{cases}
\end{align}
Finally, combining \eqref{ex:eul1}, \eqref{ex:eul2}, and \eqref{ex:eul3} gives us the globally defined $\alpha$-dissipative solution $(u, F)(t)$ with the initial data $(\bar u, \bar F)$ and $\alpha$ from Example~\ref{ex:Multipeakon}. 
\end{document}